\documentclass[11pt]{amsart}
\usepackage{amsmath,amsthm,amsfonts,amssymb,verbatim,eucal, mathrsfs}
\usepackage{tikz-cd}
\usepackage{mathrsfs}
\usepackage[all]{xy}  
\usepackage{mathtools}
\usepackage[outdir=./]{epstopdf}

\usepackage{aliascnt}

\usepackage{hyperref}
\hypersetup{colorlinks=true,citecolor=blue!70!black,linkcolor=red!60!black,linktocpage=true}
 
\usepackage[capitalise,nameinlink]{cleveref}
\crefname{enumi}{part}{parts}
\crefname{prop}{Proposition}{Propositions}
\crefname{thm}{Theorem}{Theorems}
\Crefname{diagram}{Diagram}{Diagrams}
\creflabelformat{diagram}{#2(#1)#3}

\numberwithin{equation}{section}

\newtheorem{thm}[equation]{Theorem} 
\newtheorem{prop}[equation]{Proposition}
\newtheorem{lemma}[equation]{Lemma} 
\newtheorem{cor}[equation]{Corollary}
\newtheorem{example}[equation]{Example}
\newtheorem{remark}[equation]{Remark}
\newtheorem{definition}[equation]{Definition}

\DeclareRobustCommand\longtwoheadrightarrow{\relbar\joinrel\twoheadrightarrow}
\DeclareMathOperator{\Alt}{Alt}
\DeclareMathOperator{\Ext}{Ext}
\DeclareMathOperator{\gr}{gr} 
\DeclareMathOperator{\Ima}{Im}
\DeclareMathOperator{\ima}{Im}
\DeclareMathOperator{\Ker}{Ker}

\DeclareMathOperator{\Span}{Span}
                                              
\DeclareMathOperator{\op}{op}

\newcommand{\X}{_X}
\newcommand{\alphax}{\alpha_{\X}}
\newcommand{\betax}{\beta_{\X}}
\newcommand{\lambdax}{\lambda_{\X}}
\newcommand{\gammax}{\kappa} 
\newcommand{\etax}{\eta}

\newcommand{\bA}{\bar{\A}}

\newcommand{\A}{A}
\newcommand{\chabl}{\cH_{\lambda,\alpha,\beta}}
\newcommand{\chlab}{\cH_{\lambda,\alpha,\beta}}

\newcommand{\pr}{ {\rm{pr}} }
\newcommand{\igamma}{\gamma^{-1}}
\newcommand{\BBB}{B}
\newcommand{\sA}{A }

\newcommand{\R}{\mathcal{R}}
\newcommand{\RR}{\mathscr{R} }
\newcommand{\PP}{\mathscr{P}} 
\renewcommand{\P}{\mathcal{P}}

\newcommand{\LH}{{\rm{lh}}}
\newcommand{\hh}{\rm{hom}}

\newcommand{\sd}{_{\DOT}}
\newcommand{\B}{{\bf B}} 
\newcommand{\BB}{\mathbb{B}} 
\DeclareMathOperator{\AW}{AW}
\DeclareMathOperator{\EZ}{EZ}

\DeclareMathOperator{\AWt}{\AW^{\tau}}
\DeclareMathOperator{\EZt}{{\EZ^{\tau}}}
\newcommand*\xbar[1]{%
  \hbox{%
    \vbox{%
      \hrule height 0.5pt 
      \kern0.2ex
      \hbox{%
        \kern-0.0em
         \ensuremath{#1 }
        \kern-.1em
              }%
    }\hphantom{.}
  }%
}
 
\newcommand{\itau}{{\tau^{-1}}}
\newcommand{\s}{\sigma}
\newcommand{\ttp}{S \ot_{\tau} \Ho}
\newcommand{\ttpp}{(\ttp)}
\newcommand{\hooklongrightarrow}{\lhook\joinrel\longrightarrow} 

\newcommand{\DOT}{\setlength{\unitlength}{1.2pt}\begin{picture}(2.5,2)
               (1,1)\put(2.5,2.5){\circle*{2}}\end{picture}}

\newcommand{\Hom}{\mbox{\rm Hom\,}}

\newcommand{\F}{\mathcal F}

\renewcommand{\ker}{\mbox{\rm Ker\,}}

\newcommand{\ot}{\otimes}
\newcommand{\ott}{\otimes_{\tau}}

\newcommand{\Ho}{H}
\newcommand{\cH}{\mathcal{H}}

\newcommand{\HH}{{\rm HH}}

\begin{document} 
 \nocite{*} 

 \begin{abstract}
   We consider graded deformations and PBW deformations
 of algebras defined over noncommutative algebras.
We explain how fibers of graded 
deformations correspond to
filtered algebras admitting a PBW property, with focus on 
smash product algebras for Hopf algebras 
acting on quadratic algebras.
In particular, we describe all PBW deformations
arising from Hopf actions 
on Koszul algebras, giving Hopf-Koszul Hecke algebras.
Alexander-Whitney and Eilenberg-Zilber maps for twisted
tensor products 
transfer homological information between resolutions 
and convert conditions on Hochschild cocycles and Gerstenhaber brackets
into explicit PBW conditions. 
\end{abstract}

\title[Deformation theory and Hopf actions on Koszul algebras]
        {Deformation theory\\ and Hopf actions on Koszul algebras}

        \date{April 25, 2026} 
  

\thanks{Key words: 
deformation theory, smash product algebras,
  Hochschild cohomology, Hopf algebra, Koszul algebra, Hecke algebra} 

\author{A.\ V.\ Shepler}
\address{Department of Mathematics, University of North Texas,
Denton, TX 76201, USA}
\email{ashepler@unt.edu}
\author{S.\  Witherspoon}
\address{Department of Mathematics\\Texas A\&M University\\
College Station, TX 77843, USA}\email{sjw@tamu.edu} 

\maketitle
\setcounter{tocdepth}{1}
\tableofcontents 

\section{Introduction}\label{Intro}

Momentum in understanding
noncommutative algebras and their deformations is often thwarted by the difficulty of constructing explicit 
chain maps converting between resolutions.
Many noncommutative algebras arise as twisted tensor products,
including
skew group algebras, smash product algebras, universal enveloping
algebras, and Ore extensions. 
We developed technology in~\cite{TTP-AWEZ} for twisted tensor product algebras
generally, showing how to use twisted 
Alexander-Whitney and Eilenberg-Zilber maps to construct
chain maps between resolutions.
(See also~Guccione and Guccione ~\cite{GG}.)
In this paper, we take advantage of these maps to  
present some general results on graded deformations.
We use these maps to convert homological conditions for deformations
expressed on a bar resolution
into homological conditions on a chosen
twisted tensor product resolution tailored to the algebras.
As one consequence, we address a central question of
deformation theory for Hopf algebras acting on Koszul algebras:
which Hochschild 2-cocycles lift to PBW deformations?

Results on PBW properties for nonhomogeneous algebras
 often center on presenting algebras as quotients of free algebras
 over some field or commutative ring.
 Here, we consider instead the deformation theory for 
 algebras presented as quotients of free algebras over some
 {\em noncommutative algebra}.
 We connect PBW deformations in this more general setting
 to fibers of graded deformations;
 see \cref{PBWImpliesDeformation} and \cref{DeformationImpliesPBW},
 and for Hopf actions, \cref{PBW-iff-deformation}.
 Using this perspective, we convert general theoretical results
 on deformation theory
 (see \cref{hom-condns})
 into specific properties of the relations
defining the filtered algebra. 
As an application, we give concrete PBW conditions
for graded deformations of smash product algebras $S\# H$ arising from the action of
 a Hopf algebra $H$ on a Koszul algebra $S$ (see \cref{thm:main-Hopf-left}),
as well as describing which
Hochschild $2$-cocycles lift to such deformations
(see \cref{WhichCocyclesLift}).
We allow deformations of the Hopf algebra action itself in addition to 
the Koszul relations,
giving rise to Hopf-Koszul Hecke algebras.

Our deformations generalize the graded affine
Hecke algebras of Lusztig~\cite{Lusztig88,Lusztig89},
the Drinfeld Hecke algebras of Drinfeld~\cite{Drinfeld} (see also \cite{RamShepler}), the symplectic reflection algebras of 
Etingof and Ginzburg~\cite{EG}, rational Cherednik algebras, and deformations
of Kleinian singularities~\cite{CBH98}, all in the case of a finite group
acting on a polynomial ring.
A number of authors have studied actions of universal enveloping algebras
of Lie algebras or related quantum groups on polynomial rings or quantum
polynomial rings, see ~\cite{EGG05,GK07,KT10,Tsy} for example.
In previous work \cite{SW-Koszul}, we characterized PBW deformations of groups acting on polynomial rings and
Koszul algebras more generally for
which both the group action and the Koszul relations were
deformed.
Walton and the second author~\cite{WW} extended results to Hopf
algebras acting 
on Koszul algebras and studied PBW deformations for which only
the Koszul relations were deformed.
We generalize those results here, describing all possible 
(strong) PBW deformations, see
\cref{6ConditionsImplyWeakPBW}
and \cref{6ConditionsDoublyNoetherianCase}.

We mention other related recent papers.
Khare~\cite{Khare} gave results on
PBW deformations in the case of cocommutative algebras acting on polynomial rings;
for Hopf algebras, his results are a 
special case of ours. 
Flake and Sahi~\cite{Flake19,FlakeSahi} developed a theory of Hopf-Hecke algebras;
these are PBW deformations in the setting of Hopf algebras acting
on polynomial algebras, to which our results will apply. 
See also Letzter, Sahi, and Salmasian ~\cite[Section 8]{LSS} where similar techniques are used to
understand PBW deformations in the setting of twisted tensor product algebras
and actions of quantum universal enveloping algebras. 

To obtain our results here, we apply the 
theory in~\cite{TTP-AWEZ}   of chain maps for {\em twisted tensor products}
to smash products $S\# H$,
where $H$ is a Hopf algebra acting on an algebra $S$.
This allows us to convert between
the reduced bar resolution $\BB_{S\# H}$ on one hand
and a twisted product resolution $\BB_{S}\ott\,\BB_{\Ho} $
on the other hand, where $\tau$ is a suitable twisting map on these resolutions.
When $S$ is a Koszul algebra, it further allows converting to a twisted product
resolution $K_S\ot_{\tau}\BB_H$ where $K_S$ is the Koszul resolution of $S$.
Using these chain maps, necessary and sufficient homological conditions for a PBW deformation,
expressed in terms of the differential and Gerstenhaber bracket,
are translated into algebraic conditions on parameter functions giving
deformations.

\subsection*{Outline}
We begin by establishing
in \cref{sec:deformations} general results on
PBW deformations of algebras given as quotients of tensor algebras over
noncommutative algebras.
In \cref{sec:DoublyNoetherian}, we show how PBW deformations give rise
to graded deformations and appear as their fibers.
We also argue that a weaker PBW condition suffices
when working over finite dimensional algebras or, more generally,  doubly Noetherian
algebras.
We turn to filtered algebras generated by linear and quadratic relations
in \cref{sec:LinearQuadraticAlgs} and explain how 
graded deformations in turn have fibers that are PBW deformations.

In~\cref{sec:smashproducts} we specialize
to Hopf algebras and in \cref{sec:actions-on-Koszul}
to Hopf algebras acting on Koszul algebras.
We construct
twisted tensor product resolutions and chain maps
in ~\cref{sec:spr} (see also \cref{AppendixValuesLowDegree})
drawing on the general theory of~\cite{TTP-AWEZ}
in order to later convert homological information on
graded deformations into PBW conditions.
We relate PBW parameters to cochains in
\cref{sec:ParametersToCochains}
and then establish
homological conditions for PBW deformations
expressed on a twisted product resolution
in \cref{sec:ThreeFromPBW,sec:PBWFromThree}.
Lastly,
we derive  in~\cref{sec:DHAs} concrete
algebraic conditions for PBW deformations
as quotient algebras from the abstract homological data
given by deformation theory.
We also describe which Hochschild cocycles lift to graded deformations.

\subsection*{Notation and Conventions}
We take $k$ to be an arbitrary field 
whose characteristic is not~2, and 
all tensor products are taken over $k$, i.e.~$\otimes=\otimes_k$,
unless denoted otherwise.
We assume all algebras are associative $k$-algebras.
The identity element of $k$
is identified with that of any algebra $A$, i.e., $1_k=1_A$.
We assume all Hopf algebras have bijective antipode denoted $\gamma$
throughout.
We also assume every graded Koszul algebra $S$
is
finitely generated and connected, 
so $S$ is given as the quotient of a tensor 
algebra of a finite dimensional vector space in degree~$1$
by some quadratic relations.
We assume the action of any algebra on a graded algebra
preserves the grading.



\section{PBW deformations versus formal deformations}
\label{sec:deformations}

We begin by recalling some results on
deformation theory over rings and defining PBW deformations over
noncommutative algebras.

\subsection*{Deformations and Hochschild cohomology}
Recall that a deformation of an algebra $\sA$ over $k[t]$ is 
an algebra $\sA_t$ over $k[t]$
isomorphic to $\sA[t]\cong \sA\ot k[t]$ as a $k[t]$-module
with $$\sA_t/t\sA_t \ \cong\ \sA \qquad\text{ as a $k$-algebra.}
$$
The algebra $\sA_t|_{t=c}$ for some $c\in k$
is called a {\em fiber of the deformation} $\sA_t$.
If in addition $\sA$ and $\sA_t$ are graded algebras with $\deg t=1$
and
$\sA_t/ t\sA_t\cong \sA$ as graded algebras, 
then we say that $\sA_t$ is a {\em graded deformation}.
In that case, the fiber $\sA_t|_{t=c}$ is merely a filtered algebra.
Note that we
write $\sA_i$ for the $i$-th
graded component of a graded algebra $\sA$ 
and $\F_i(\BBB)$ for the $i$-th
filtered component of any filtered algebra $\BBB$.

Every deformation $\sA_t$ 
has $k[t]$-linear multiplication given by
$$
a\star a' = \mu_0(a\ot a') +
\mu_1(a\ot a')t + \mu_2(a\ot a')t^2 + \ldots 
\qquad\text{ for } a,a'\in \sA 
$$
with $\mu_0=m:\sA\ot \sA\rightarrow \sA$ the multiplication map on $\sA$
and some $k$-linear maps $\mu_i:\sA\ot \sA\rightarrow \sA$
(after identifying $\sA_t$ with $\sA[t]$ as a $k[t]$-module).
In this case,
$\sA_t\cong \sA[t]$
 as graded vector space
and the associated graded algebra of any fiber,
$\gr (\sA_t|_{t=c})$, is isomorphic to $\sA$ as a graded algebra.
(Recall that the 
{\em associated graded algebra}
of a filtered algebra $
B=\cup_{i\geq 0} \, \F_i(B)$
is
$ \gr B =\bigoplus_{i\geq 0} \F_i/ \F_{i-1}$
with 
product
$(x +\F_i)(y +\F_j)=xy  +
\F_{i+j-1}$ for
$\F_i=\F_i(B)$ and
$\F_{-1}=\{0\}$.)
We may assume without loss of generality that the multiplicative
identity $1_{\sA}$ of $\sA$ is retained as the multiplicative
identity on $\sA_t$.

Recall that if $\sA_t$ is a deformation of $\sA$,
then $\mu_1$ represents an element 
in the Hochschild cohomology
$\HH^2(\sA)=\HH^2(\sA,\sA)=\Ext^{2}_{\sA^e}(\sA,\sA)$
for $\sA^e=\sA\ot \sA^{\text{op}}$ (see \cite{GerstenhaberSchack}).
The Hochschild cohomology can be defined as the homology
arising from applying $\Hom_{\sA^e}(-, \sA)$
to the {\em bar resolution} $\B_{\sA}$ with $n$-th term
$\sA^{\ot (n+2)}$ giving augmented complex
$$
\begin{aligned}
  (\B_{\sA})\sd\ :\ \ 
  & \cdots \longrightarrow (\B_{\sA})_2 \longrightarrow (\B_{\sA})_1 \longrightarrow \sA\ot
  \sA
  \xrightarrow{\ \ m\ \ } 
\sA \longrightarrow 0 \qquad && \text{ (bar) }
\end{aligned}
$$
for $m$ the multiplication on $\sA$ with differential in homological
degree $n$ given by
\begin{equation}\label{eqn:bar-diff}
  a_0\ot\cdots \ot a_{n+1}
  \mapsto \sum_{i=0}^{n} (-1)^i  a_0\ot\cdots\ot a_{i-1}\ot a_ia_{i+1}\ot a_{i+2} \cdots
   \ot a_{n+1} \, .
\end{equation}
Note the {\em reduced bar resolution} $\BB_{\sA}$
has $n$-th term $\sA\ot \bar{\sA}^{\ot n} \ot \sA$
for $\bar \sA=\sA/k1_{\sA}$ giving augmented complex
$$
\begin{aligned}
  (\BB_{\sA})\sd\ :\ \ 
  & \cdots \longrightarrow (\BB_{\sA})_2 \longrightarrow (\BB_{\sA})_1 \longrightarrow \sA\ot
  \sA
  \xrightarrow{\ \ m\ \ } 
\sA \longrightarrow 0 \qquad && \text{ (reduced bar) }
\end{aligned}
$$
with differential induced by that for $\B_{\sA}$ after identifying
$\bar \sA$ with a vector subspace of $\sA$, so $\sA=k1_{\sA}\oplus \bar {\sA}$,
by choosing a section to the standard projection
$\sA\twoheadrightarrow \sA/k1_{\sA}$.

  Conversely, if a cocycle $\mu$ giving a cohomology class in $\HH^2(\sA)$
  is the first multiplication map $\mu_1$ for some deformation of $\sA$,
  then we say that $\mu$ {\em lifts} or {\em integrates}
  to a deformation.  One fundamental question of algebraic deformation
  theory
  centers on determining which cocycles lift to deformations.
Before giving a description of those deformations that lift to graded deformations
of crossed product algebras in the next section, 
we give some general results on fibers of deformations.

\subsection*{Filtered algebras over another algebra}

We consider generally algebras given in terms
of generators and relations, i.e., as quotients, but we work over a
possibly noncommutative
algebra instead of a field $k$.
We adapt ideas of
Braverman and Gaitsgory \cite{Braverman-Gaitsgory} on filtered
algebras
and of noncommutative Gr\"obner basis theory (see \cite{Li2011})
to our setting.

We fix throughout this section a $k$-algebra $H$ and an 
$H$-bimodule $W$ which will serve as a generating set of some filtered
algebra. We give the elements of $W$ degree $1$ and take
the tensor algebra $T_H(W)$ over $\Ho$ as a graded
$k$-algebra with multiplication given by $\ot_H$.
Thus $T_H(W)$ may be regarded as the algebra over $H$
generated by the $H$-bimodule $W$.

The algebra $T_H(W)$ is both filtered and graded
in the usual way with
$i$-th degree graded component $W^{\ot_{\!\Ho} i}$
(so $H$ is the degree $0$ component)
and $i$-th degree filtered component
$$\mathcal{F}_i(T_{\! _H}(W))=\oplus_{j=0}^i \ W^{\ot_{\!\Ho} j}\, .$$
For any (two-sided) ideal $I$ of $T_H(W)$, the quotient
algebra
$T_H(W)/I$ inherits the filtration with
$i$-th filtered component $\F_i= 
\F_i(T_H(W)/I)= (\F_i(T_H(W)) + I )/I$ as usual.

We fix a $k$-subspace
of {\em generating relations} $\PP\subset T_H(W)$
with $\PP\cap H = \{0\}$
and  consider the filtered algebra 
$$ \BBB=T_{\! _H}(W)/(\PP)\, .$$
We compare three graded versions of $\BBB$
(see Li~\cite{Li}), one using  a degree $1$ parameter $t$:
\begin{itemize}\itemindent=-19pt \itemsep=1ex
\item  $\gr \BBB$, 
      the standard associated graded algebra of $\BBB$,
\item $\BBB_{\hh}$ obtained by taking the leading homogeneous part of each generating
  relation,
  \item $\BBB_t$ obtained by 
    homogenizing each generating relation with respect to $t$.
        \end{itemize}
Specifically, each $x$ in $T_H(W)$ of {filtered degree} $n$
can be written as $\sum_{i=1}^n x_i$
for $x_i$ in $W^{\ot_{\!\Ho} i}$ and $x_n\neq 0$ and we set
$$\begin{aligned}
  \LH[x] &:=x_n &&\text{ in } W^{\ot_{\!\Ho} n},
&&\text{ the {\em leading homogeneous part} of $x$, and}\\
  x_t &:=\sum_{i=0}^n x_i\, t^{n-i} &&\text{ in } T_{\! _H}(W)[t],
  &&\text{ the {\em $k[t]$-homogenization}
of $x$.}
\end{aligned}
$$
Then $\LH[x]$ is homogeneous of degree $n$,
$x_t$ is 
homogeneous of degree $n$ for $t$ of degree $1$,
and $x_t$ is filtered of
degree $n$ for $t$ of degree $0$.
Set
\begin{equation}\label{LH}
  \begin{aligned}
  \RR&=\LH[\PP ] :=\{ \LH[x]: x\in \PP\},
 &&\quad\text{the {\em homogeneous version of $\PP$}, and } \\
 \PP_t&=\{ x_t: x\in \PP\},
 &&\quad\text{the {\em $k[t]$-homogenized version of $\PP$}. }
\end{aligned}
\end{equation}
We compare the associated graded algebra of
the filtered algebra $\BBB$,
$$
\hphantom{xxxx} \gr \BBB \ \ =\ \bigoplus_{n\geq 0} \F_n(\BBB)/\F_{n-1}(\BBB)
\qquad\quad \text{({\em associated graded of $\BBB$}),}
\hspace{10ex}\hphantom{x}
$$
to its {\em homogeneous version}
and  {\em $k[t]$-homogenized version} 
given respectively by 
\begin{equation}\label{homogeneousversion}
\begin{aligned}
  &\BBB_{\hh}
  &=&  \ \ T_{\! _H}(W)/(\RR) 
\qquad\text{\ \ \ \ \ \ (\em homogeneous version of  $\BBB$),}
\\
&\BBB_t
&=& \ \ T_{\! _H}(W)[t]/(\PP_t) 
\qquad\ \text{ ($k[t]$-{\em homogenized version of }$\BBB$)}
\, \end{aligned}
\end{equation}
which depend on the choice of generating relations $\PP$.
Note that $\sA=\BBB_{\hh}$ is a graded algebra, $\BBB_t$ is graded for $t$ in
degree $1$, and $\BBB_t$ is filtered for $t$ in degree $0$.

If $I$ is any (two-sided) ideal of $T_H(W)$,
then $(\LH[I])$ is a homogeneous ideal 
and the quotient $T_H(W)/(\LH[I])$ is a graded algebra.
The next lemma follows from~\cite[Theorem~1.6]{Li} 
by checking that the maps given there, as applied to our setting,
are in fact $H$-bimodule maps.
Note that the second statement follows from the first
since $( \LH[\PP] ) \subset ( \LH[(\PP)] )$.

\begin{lemma}
  \label{Epimorphism}
  For any ideal $I$ of $T_{\! _H}(W)$,
there is an isomorphism of $H$-bimodules and 
  graded algebras
  $$
  T_{\! _H}(W)/(\LH[I]) \ \xrightarrow{\ \ \cong \ \ } \
  \gr\big(T_{\! _H}(W)/I \big) 
  \,
  $$
  mapping
  $x+(\LH[I])$ to $x+I+\mathcal{F}_{n-1}$ 
  for $x$ filtered of degree $n$.
Thus for any subspace $\PP$ of $T_H(V)$, there is
an $H$-bimodule graded algebra epimorphism 
\begin{equation*}
\BBB_{\hh} 
\longrightarrow 
\gr \BBB
\qquad\qquad\text{ for }\quad \BBB=T_{\! _H}(W)/(\PP)
\, . 
\end{equation*}
\end{lemma}

\subsection*{PBW deformations}
Braverman and Gaitsgory \cite{Braverman-Gaitsgory} studied conditions under which
there is not only an epimorphism
$\BBB_{\hh} \longrightarrow  
\gr \BBB $ but an isomorphism
in the setting of quotients of free algebras over a field
(i.e.~when $H$ is a field in our notation), 
addressing a question posed by
Bernstein (see Li~\cite{Li}). 
These ideas were extended to the setting of graded rings
by Li~\cite{Li}.
We distinguish between the case when the specific epimorphism of \cref{Epimorphism}
is an isomorphism
and the case of mere existence of some isomorphism.

\vspace{1ex}

\begin{definition}\label{defn:strongPBW}
{\em 
  We say the filtered algebra
$\BBB=T_{H}(W)/(\PP)$ is a {\em PBW deformation}
of its homogeneous version $\BBB_{\hh}=T_{H}(W)/(\RR)$
if its {associated graded algebra}  and homogeneous versions are isomorphic
as graded algebras,
$$
\gr(\BBB) \ \cong\ \BBB_{\hh} 
\, .
$$
We say 
$\BBB$ is a {\em strong PBW deformation}
of $\BBB_{\hh}$
if the canonical graded algebra epimorphism 
$
\BBB_{\hh} 
\longrightarrow
\gr \BBB 
$
of \cref{Epimorphism} is an isomorphism.
}
\end{definition}

\vspace{1ex}

We note that Li~\cite{Li} uses the term  {\em general PBW property}
instead of strong PBW deformation as we do here.   
Other authors use the term {\em PBW type}  or {\em PBW property} for what we call
here strong PBW deformation; 
see for example~\cite{Braverman-Gaitsgory,Flake19,FlakeSahi,Khare}.

\section{PBW deformations as fibers and doubly Noetherian algebras}
\label{sec:DoublyNoetherian}
We argue in this section that strong PBW deformations
are always fibers of graded deformations.  We consider the converse
in the setting of linear-quadratic filtered algebras in the next section.
We also show that the PBW condition is equivalent to the strong PBW
condition
when working over 
doubly Noetherian algebras.

\subsection*{Doubly Noetherian algebras}
We use the terminology of~\cite{YZ} for when an algebra $H$ has
Noetherian enveloping algebra:
\begin{definition}
{\em  An algebra $H$ 
is called {\em doubly Noetherian} if
$H^e=H\ot H^{\op}$ is Noetherian.
}
\end{definition}
All finite dimensional algebras are both Noetherian
and doubly Noetherian.
In general, note a Noetherian algebra need not be doubly Noetherian; 
e.g., see~\cite[Theorem 8.2]{Rogalski}.

Recall that if $\s$ is an automorphism of $H$ and $M$ is an
$H$-bimodule,
the {\em $\s$-twisted $H$-bimodule} { $\, ^{\s\!}M^{\s}$ is the set $M$ with
  $H$-bimodule structure given by
  $h\ot m\ot h'\mapsto \s(h)\, m\, \s(h')$.
\begin{lemma}\label{NoetherianImpliesHopfian}
  Suppose $\s$ is an automorphism of a doubly Noetherian algebra $H$
  and $M$ is a finitely generated $H$-bimodule.
  Then any surjective $H$-bimodule
  homomorphism $M\rightarrow \, ^{\s\!}M^{\s}$ 
  is bijective.  In particular, any surjective endomorphism of a finitely
  generated $H$-bimodule is an automorphism.
\end{lemma}
\begin{proof}
Let
  $f^n:M\rightarrow \, M$ be the $n$-fold
  composition of the surjective map
  $f:M\rightarrow \, ^{\s \!}M^{\s}$ regarded as an additive group homomorphism.
 Then $\ker f\subset \ker  
f^2\subset \ker f^3\subset \ldots$ is a sequence of $H$-subbimodules of $M$.
Since $H$ is doubly Noetherian and $M$ is finitely generated, 
$M$ is a Noetherian $H$-bimodule
(see~\cite[Corollary~1.4]{GoodearlWarfield}).
Hence
$\ker f^n = \ker f^{n+1}$ for some $n$.
For $x$ in $\ker f$, since $f$ is surjective,
there is some $x_2$ in $M$ with $x=f(x_2)$.
Note then that $x_2$ lies in $\ker f^2$.
Continuing, $x_2=f(x_3)$ for some $x_3$ in $M$ and $x_3$ lies in $\ker f^3$.
 Inductively, $x=f^{n}(x_{n+1})$ for some $x_{n+1}$ in $\ker f^{n+1}=\ker f^{n}$
 so $0=f^{n}(x_{n+1})=x$ and $f$ is injective.
 For the last statement, we just take $\s$ to be the identity map.
  \end{proof}

 \begin{remark}\label{rk:khare}
   {\em 
     Note that some hypothesis on $H$ in the last lemma is required.
The conclusion is false in general. 
See e.g.~\cite{Shepherdson}, \cite{Montgomery1983}, and
      \cite[Section~1, Exercise 18]{Lam1999}.
    }
    \end{remark}

 \subsection*{PBW deformations over doubly Noetherian algebras}
 As in the last section, we fix an algebra $H$ and $H$-bimodule $W$
 and consider a $k$-vector space 
 $\PP\subset T_{H}(W)$ of generating relations
 with homogeneous version $\RR=\LH[\PP]$,
    the set of leading homogeneous parts of the elements of $\PP$
    (see \cref{LH}).
    Then $\BBB=T_{H}(W)/(\PP)$ is a filtered algebra and
    its homogeneous version $\BBB_{\hh}=T_{H}(W)/(\RR)$
    is a graded algebra.

\begin{prop}\label{PBWImpliesLH}
The filtered algebra $\BBB$ 
is a strong PBW deformation
of  its homogeneous version $\BBB_{\hh}$ if and only if
$(\LH[\PP])=(\LH[(\PP)])$.
\end{prop}
\begin{proof}
  Following the argument for \cite[Proposition 1.7]{Li},
  we observe that the
  surjective $H$-bimodule 
  graded algebra map of
  \cref{Epimorphism},
 \begin{equation*}
  \BBB_{\hh}=
T_{\! _H}(W)/ (\LH[\PP])
\longrightarrow
T_{\! _H}(W)/ (\LH[(\PP)])
\ \cong\  \gr\big(T_{\! _H}(W)/(\PP) \big)
=\gr \BBB 
\, ,
\end{equation*}
is an isomorphism 
exactly when
$(\LH[\PP])=(\LH[(\PP)])$.
\end{proof}

\begin{prop}\label{prop:doubly-PBW}
  Suppose $H$ is doubly Noetherian
 and
$\BBB_{\hh}$ in each graded degree
is finitely generated as an $H$-bimodule.
Then $\BBB$ is a PBW deformation of
its
homogeneous version $\BBB_{\hh}$ if and only if it is a strong PBW deformation
of $\BBB_{\hh}$.   
\end{prop}
\begin{proof}
Every strong PBW deformation is
a PBW deformation by definition.
For the converse, assume $\BBB$ is a (not necessarily strong) PBW deformation and
compose the surjective graded algebra and $H$-bimodule map
  $\BBB_{\hh}\rightarrow \gr \BBB $
  of \cref{Epimorphism}
  with a graded algebra isomorphism
  $\gr \BBB \rightarrow \BBB_{\hh}$ given by the PBW hypothesis
  to obtain a surjective graded algebra endomorphism
  $f:\BBB_{\hh}\rightarrow \BBB_{\hh}$.
  To obtain an $H$-bimodule map, 
  we note that $f$
  restricts to an automorphism
  $\sigma$
  on $H\subset T_H(W)$ and we take
  $\, ^{\s\!} \BBB_{\hh}^{\s}$, the $\s$-twisted $H$-bimodule.  The map $f$ defines 
  a surjective $H$-bimodule homomorphism 
  $\BBB_{\hh} \rightarrow \, ^{\, \s\!}\BBB_{\hh}
  ^{\s}$ 
that restricts to a bijection on each graded component
    by \cref{NoetherianImpliesHopfian}.
  Hence
  the map of \cref{Epimorphism} must be an isomorphism of graded
  algebras and $B$ is a strong PBW deformation.
\end{proof}

    \begin{remark}
\label{HiddenRelations}
      {\em
        \cref{PBWImpliesLH} gives conditions for hidden
        relations
        in higher degree not immediately suggested by a cursory examination of the
generating relations.  
Following ideas of Braverman and
Gaitgory~\cite{Braverman-Gaitsgory} and Li~\cite{Li}, as well
as the theory of noncommutative
Gr\"obner
bases,
one might say
  that a $k$-vector space of generating relations $\PP\subset T_H(W)$
      for a filtered algebra $T_{H}(W)/(\PP)$
      exhibits {\em no hidden relations in higher degree}
     when  $$  (\PP) \cap 
\F_n
= \sum_{i+m + j = n} \F_i\ot_H (\PP\cap \F^m) \ot_H \F_j 
\quad\text{for } n\geq 0 
\,  $$
for $n$-th filtered component $\F_n=\F_n(T_{H}(W))$.
By \cite[Proposition 1.7]{Li}),
    $\PP$ exhibits no hidden relations in higher
    degree
    if and only if 
    $(\LH[\PP])=(\LH[(\PP)])$.
  }
\end{remark}

\begin{remark}{\em
In previous work, 
we took $W$ to be finite dimensional over $k$ and took $H$ to
be the group ring of a finite group, so each filtered component was finite
dimensional, see~\cite{SW-Koszul}. 
Here, we relax these assumptions and merely take
$H$ to be a doubly Noetherian algebra
over the field $k$ (for example, finite dimensional).
}
\end{remark}


Note that by \cref{prop:doubly-PBW}, 
we may remove the word ``strong'' from the following theorem
whenever $H$ is finite dimensional (or, more generally, doubly Noetherian) and 
$\sA$ in each graded degree is a finitely generated $H$-bimodule.
We use the $k[t]$-homogenized version 
$\BBB_t=T_{H}(W)[t]/(\PP_t)$
of $\BBB$ as in \cref{homogeneousversion}.

\begin{thm}\label{PBWImpliesDeformation}
   If a filtered algebra
   $\BBB=T_{H}(W)/(\PP)$
   is a strong PBW deformation of its homogeneous version $\BBB_{\hh}$,
then 
$\BBB_t$ is a graded deformation of $\BBB_{\hh}$ with fiber $\BBB$
at $t=1$.
  \end{thm}
  \begin{proof}
 By construction,
    $\BBB_t/t\BBB_t$ and $\sA=\BBB_{\hh}=T_H(W)/(\RR)$ are isomorphic as graded algebras
for $\RR = \LH [ \PP]$
and $\BBB_t$ is a graded algebra for $t$ in degree $1$ with fiber
$\BBB$ at $t=1$.
We show that this in turn implies that 
$\BBB_t \cong
\sA[t]$  as $k[t]$-modules.
To construct an isomorphism, we
take the grading and filtration on $T_{H}(W)[t]$ with $t$ in
    degree $0$ throughout the rest of this proof.

Since $k$ is a field, 
the projection map $T_{H}(W)\rightarrow 
    T_{H}(W)/(\RR)$ splits as a $k$-linear map 
    and we fix a choice of $k$-linear section
               $\iota:T_{H}(W)/(\RR)\rightarrow T_{H}(W)$
               (by making a choice of coset representatives).
We extend to a map  of filtered 
    $k[t]$-modules (using that $\RR$ is homogeneous),
    $$\iota[t]: \big(T_{\! _H}(W)/(\RR)\big)[t]
    \longrightarrow T_{\! _H}(W)[t]\, . $$  
    Let $\Upsilon$ be the composition with projection
    onto $T_{H}(W)[t]/(\PP_t)$, a filtered $k[t]$-module map:
$$
\begin{aligned}
  \Upsilon:\ \sA[t]=\big(T_{\! _H}(W)/ (\RR)\big)[t]
  \xrightarrow{ \  \iota [t]\ } 
  T_{\! _H}(W)[t]
  \longrightarrow
T_{\! _H}(W)[t]/(\PP_t)=\BBB_t 
\,  .
\end{aligned}
$$

To show that $\Upsilon$
is bijective, we 
fix $n\geq 0$
and restrict to the $n$-th filtered component:
$$
\Upsilon_n:  \F_n (\sA[t])
\longrightarrow 
\F_n(\BBB_t ) 
\,  . 
$$
Note that for $n=0$, $\Upsilon_n$ is bijective since
$\F_0(\sA[t])=H[t]$
and $\iota$ must be the identity on $H$ as $\RR\cap H$ is empty by assumption.

We fix $n>0$ and  argue that
$\Upsilon_n$ is onto.
Assume inductively that
$$\Upsilon_{n-1}: \F_{n-1}( \sA[t] ) \longrightarrow \F_{n-1}(\BBB_t)$$
is onto.
Fix $x\in T_{H}(W)$ homogeneous of degree $n$
and suppose
$\Upsilon$ takes $x+(\RR)$ to some $x'+(\PP_t)$.
By construction of $\Upsilon$, the element $x-x'$ lies in $(\RR)$, so
without loss of generality
 $$x-x'= \sum_{p\in \PP} a_p \ \LH(p)\ b_p,$$ 
 a finite sum, for some $a_p$, $b_p$ in $T_{H}(W)$
 with each $a_p \, \LH(p)\, b_p$  homogeneous of degree $\leq n$
since $\RR = \LH[\PP]$ is a homogeneous ideal.
Then $x + (\PP_t)$ is 
\begin{equation}\label{CancelTopDegree}
  \begin{aligned}
 x + (\PP_t)
=  x'+ \sum_{p\in \PP} a_p\, \LH(p) \, b_p + (\PP_t)
=
x'-\sum_{p\in \PP} a_p\,  (p_t-\LH(p) )\,  b_p
+ (\PP_t)
\, .
\end{aligned}
\end{equation}
But  each $a_p\, (p_t-\LH(p))\, b_p$ is filtered of degree $<n$
as the leading homogeneous part of $p_t$ and $p$ coincide
(recall $\deg t=0$),
so each
summand on the right side of \cref{CancelTopDegree}
lies in $\ima \Upsilon_{n-1}\subset \ima \Upsilon_{n}$ by induction.
Then as $x'+(\PP_t) \in\ima \Upsilon_n$ as well, we must have $x+(\PP_t)
\in\ima \Upsilon_n$
and
$\Upsilon_n$ is onto.
Hence $\Upsilon$ is onto.

We construct a map for each $n>0$,
$$
\begin{aligned}
  \Upsilon_n': \F_n(\BBB_t)
  \longrightarrow
(\sA[t])_n 
\, ,  
\end{aligned}
$$
by restricting the $k[t]$-module projection map
(a filtered map for $t$ in degree $0$)
$$T_{\! _H}(W)[t]\relbar\joinrel\twoheadrightarrow 
\big(T_{\! _H}(W)/(\RR)\big)[t]
$$
to the $n$-th filtered component
and then projecting to the associated graded algebra
using the fact that $\sA[t]$ is already graded:
Consider
$$
\begin{aligned}
  \Psi_n: 
   \F_n\big(T_{\! _H}(W)[t]\big)
  \relbar\joinrel\twoheadrightarrow 
  \F_n\Big( \big(T_{\! _H}(W)/(\RR)\big)[t]\Big)=\F_n(\sA[t])
  \relbar\joinrel\twoheadrightarrow 
  \F_n(\sA[t])/\F_{n-1}(\sA[t])
  \cong \sA[t]_n
  \,  .
\end{aligned}
$$
We argue that $(\PP_t) \cap \F_n\big(T_{H}(W)[t]\big)$
lies in the kernel
giving rise to an onto
algebra map
$$
\begin{aligned}
\Upsilon'_n: \F_n(\BBB_t) = \F_n\big(T_{\! _H}(W)[t]/(\PP_t)\big)
  \relbar\joinrel\twoheadrightarrow (\sA[t])_n 
\,  .
\end{aligned}
$$

To see that $(\PP_t) \cap \F_n\big(T_{H}(W)[t]\big)
\subset \ker \Psi_n$,
take an element $x(t)$ of the ideal $(\PP_t)$
of $T_{H}(W)[t]$ of filtered degree  $n$.
By \cite[Proposition 1.7]{Li}, 
$$(\PP_t) = (\PP)_t=\{q_t: q \in (\PP)\}$$
as $(\LH[(\PP)]) = (\LH[\PP])$ by \cref{PBWImpliesLH}.
Thus we may assume without loss of generality that
$x(t)$ is the $k[t]$-homogenization of some finite sum
$x=\sum_{p\in \PP} a_p\, p \, b_p$ in $(\PP)$
of filtered degree $\text{fdeg}(x)=n$
with each $a_p$, $b_p$ in $T_H(W)$.
Using \cite[Proposition 1.7]{Li} again,
we may assume each
$
\text{fdeg}(a_p)+
\text{fdeg}(p)+
\text{fdeg}
(b_p)
\leq n$
(see \cref{HiddenRelations}).
Since $\Psi_n\equiv 0$ on $\F_{n-1}(T_H(W)[t])$
and $x(t)=x_t$,
we also may assume
each $\text{fdeg}(a_p \,p \, b_n)=n$,
else subtract off lower degree terms from $x$
without changing $\Psi_n(x(t))$.
Then
$$
n = \text{fdeg}(a_p \, p\, b_p)
\leq 
\text{fdeg}(a_p)+
\text{fdeg}(p)+
\text{fdeg}
(b_p)
\leq
n\, 
$$
and we have equality thoughout.
This in turn implies each $\LH[a_p \, p\, b_p]
= \LH [a_p]\, \LH[p]\, \LH[b_p]$
and $x(t)=\sum_p (a_p\, p\, b_p)_t$.
Set $y(t)=\sum_{p\in \PP} (a_p \, \LH[p]\, b_p)_t$.
Then $\Psi_n(y(t)) = 0$ as $\LH[\PP]=\RR$
and $\Psi_n$ takes coefficients mod $(\RR)$.
Thus
$$
\begin{aligned}
  \Psi_n(x(t))
  =\Psi_n\big(x(t)-y(t) \big)
=\Psi_n\Big(\sum_{p\in\PP}(a_p\, p\, b_p)_t - (a_p\, \LH[p]\, b_p)_t \Big)
=0
\end{aligned}
$$
as each summand of the input has filtered degree $<n$.
Hence $\Upsilon'_n$ is well-defined.

Since $\Upsilon'_n$ is zero on $\F_{n-1}(\BBB_t)$ for all $n$,
we obtain a $k[t]$-module map
$$\Upsilon': \BBB_t\longrightarrow
\gr \BBB_t\longrightarrow \oplus_n (\sA[t])_n\cong \sA[t]$$
(given by projection followed by the map $\oplus_n \Upsilon_n'$).
One may check directly that $\Upsilon' \circ \Upsilon = 1$,
the identity map on $\sA[t]$,
so $\Upsilon$ is not only onto but also 
one-to-one. 
Hence $\Upsilon: \sA[t] \rightarrow \BBB_t$ is a 
$k[t]$-module isomorphism
and $\BBB_t$ is a graded deformation of $\sA$.
\end{proof}

\section{PBW deformations and fibers of linear-quadratic algebras}
\label{sec:LinearQuadraticAlgs}

We again fix throughout this section an algebra $H$ and an 
$H$-bimodule $W$ which will serve as a generating set of some filtered
algebra. We give the elements of $W$ degree $1$ and take
the tensor algebra $T_{H}(W)$ over $H$ as a graded
$k$-algebra (as in the last sections).
We again consider a homogeneous $k$-subspace of relations $\RR$
with $A=T_H(W)/(\RR)$ graded in the usual way, i.e.,
with $i$-th degree component
$A_i=(W^{\ot_{\!\Ho} i} + (\RR))/ (\RR)$.

\subsection*{Linear-quadratic homogeneous algebras}
We say a quotient $\A=T_{H}(W)/(\RR)$ is a
{\em linear-quadratic homogeneous algebra}
when each generating relation in the subspace $\RR$
is homogeneous of degree $1$ or degree $2$, i.e., when
$$\RR \subset W \oplus (W\ot_H W)\, .$$

\subsection*{Basis generating the degree $2$ relations}
Now suppose $W$ is a free $H$-bimodule and let $V$
be the $k$-span of a free basis.
We may identify
the $k$-vector space
$V\ot V$ with the $k$-subspace 
of $W\ot_H W$ generated by $\{v\ot_H v': v, v'\in V\}$,
writing
\begin{equation}\label{eqn:VVWW}
   V\ot V\subset  W\ot_H W\, ,
\end{equation}
since
$$
\begin{aligned}
  V\ot V 
  &\cong k1_H \ot (V\ot k1_H \ot V)\ot k1_H \\
  &\subset \ \ 
  H\, \ot V\ot \, H\ \ot\ V \ot H\
  \cong (H\ot V\ot H) \ot_H (H \ot V \ot H)
  \cong W\ot_H W 
  \, .
\end{aligned}
$$
We say the free basis of $W$ {\em generates the degree 2 relations}
of 
$T_H(W)/(\RR)$ 
when the degree~$2$ part $\RR\cap (W\ot_H W)$
is a subset of $V\ot V$.

\vspace{1ex}

\begin{thm}\label{DeformationImpliesPBW}
  Let $W$ be a free $H$-bimodule with basis generating the degree $2$
  relations of a linear-quadratic algebra $\sA=T_{H}(W)/(\RR)$.
If $\sA_t$ is a graded deformation
of $\sA$, then 
any fiber  $\sA_t|_{t=c}$  ($c\in k$)
  is isomorphic to some strong PBW deformation 
  of $\sA$
  as a filtered algebra.
\end{thm}
\begin{proof}
   We construct a strong PBW deformation of $\sA$
  depending on a choice of section
  to the canonical projection map
  $$\pi_{_\sA}:T_{\! _H}(W)\rightarrow T_{\! _H}(W)/(\RR)=\sA\, ,
  \qquad x\longmapsto \bar{x}:=x+(\RR)
  \, .
  $$
  Define a $k$-linear map depending on a choice of 
 coset representatives
  \begin{equation}\label{section}
    \iota_{_\sA}  : T_{\! _H}(W)/(\RR) \longrightarrow T_{\! _H}(W) 
  \qquad\text{ with }\quad \pi_{_\sA}\, \iota_{_\sA} = 1_{_\sA}\, .
  \end{equation}
    
As $\RR\subset W\oplus (W\ot_H W)$,
we may identify $H$ with $H+(\RR)$ in $\sA$
and we write $h$ for $\bar h$ throughout.
Also, given
any $k$-linear function $f:U\rightarrow \sA$
for a  $k$-vector space $U$,
define $f\, t^i: U\rightarrow \sA[t]$ 
by $u\mapsto f(u)\, t^i$
for ease with notation with the central parameter $t$.

{\bf Multiplication in the graded deformation.}
  Identify $\sA_t$ with $\sA[t]$ as a graded vector space
  with $\deg t=1$ and multiplication $\star$:  For $a,b$ in $\sA$
  with product $ab$ in $\sA$,
\begin{equation}\label{eqn:star-product}
  a\star b=ab + \sum_{i\geq 1} \mu_i(a\ot b)t^i
  \qquad
  \text{for $k$-linear maps $\mu_i:\sA\ot \sA\rightarrow \sA$ graded of degree
    $-i$.}
\end{equation}
  Consider the unique algebra and $\Ho[t]$-bimodule homomorphism
  \begin{equation}
    \label{PsiDef}
    \psi:\
    T_{\! _H}(W)[t] \longrightarrow \sA_t\,
   \end{equation}
  defined by $h\mapsto \bar h$ and
$v\mapsto \bar v$
for $h$ in $\Ho$ and $v$ in $V$.
Note that $\psi$ is a graded algebra map by construction.
We identify some elements in the kernel giving rise to parameter functions
$\lambda$, $\alpha$, $\beta$
that will feature prominently in the rest of the paper.

{\bf Deforming the degree $1$ relations.}
  We first claim that there is a $k$-linear
  function $\lambda: W\rightarrow \Ho$
  with $\psi\big|_{W}=\pi_{_\sA}\big|_{W}+\lambda t$ 
  since $\sA_t$ is graded.
  Indeed, if $w=hvh'$ in $W$ for $h,h'$ in $\Ho$ and $v$ in $V$, then
$$
\psi(hv)=\psi(h\ot_H v)=\psi(h)\star \psi(v) 
=\bar h\star \bar v  =  h \bar v + \mu_1(h\ot \bar v)t 
\, . 
$$
Then as $\deg\mu_1=-1$, $\deg h=0$, and $\deg \bar v=1$,
the element
$\mu_1(h\ot \bar v)$ has degree $0$ and thus must lie in $\Ho$,
so
$$
\psi(hvh')=\psi(hv)\star \psi(h') 
=\big(h\bar v + \mu_1(h\ot \bar v)t \big) \star h'
=h\bar vh' + \mu_1(h\bar v \ot h')t + \mu_1(h\ot \bar v) h' t\, ,
\, 
$$
i.e., $\psi(w)=\bar w+\lambda(w)\, t$ for all $w\in W$
for some $k$-linear function 
$\lambda: W \rightarrow H$.
Hence
\begin{equation}\label{psi}
  \psi\big|_{W}= \pi_{_\sA}+\lambda\, t
  \, .
  \end{equation}

{\bf Parameter functions giving filtered relations.}
We again use the fact that $\sA_t$ is graded:
As $\deg \mu_i=-i$ and $W$ lies in degree $1$
   and $\Ho$ in degree $0$,
   $$\begin{aligned}
     &\mu_1 (\pi_{_\sA}\ot \pi_{_\sA}) (V\ot V)\subset &&(T_{\! _H}(W)/(\RR))_1=W+(\RR),
   \quad\text{ and }\\
   &\mu_2 (\pi_{_\sA}\ot \pi_{_\sA}) (V\ot V)\subset &&(T_{\! _H}(W)/(\RR))_0=H+(\RR)\cong H
   \, .
   \end{aligned}
   $$
   Thus our above map
\begin{equation}\label{ConstructingLambda}
\lambda :W \longrightarrow H,
\qquad \lambda(hvh')=
\mu_1(h\bar v \ot h') + \mu_1(h\ot \bar v) h' \in H+(\RR)\cong H
\end{equation}
for $h\in H,\ v,v'\in V$,
may be used to 
define $k$-linear maps
\begin{equation}\label{ConstructingParameters}
\begin{aligned}
&\alpha:V\ot V\longrightarrow W,
&& \alpha= \iota_{_\sA}\, \mu_1 \, (\pi_{_\sA}\ot \pi_{_\sA})
\\
 &\beta:V\ot V\longrightarrow H,
 &&\beta= \mu_2\, (\pi_{_\sA}\ot \pi_{_\sA})
 -\lambda\, \alpha
\, .
\end{aligned}
\end{equation}
We use these maps to construct a set of filtered relations $\PP$
defining a nonhomogeneous algebra $T_{H}(W)/(\PP)$ that we show
is a strong PBW deformation of $\sA$.  The map $\alpha$ records
values of the first multiplication map $\mu_1$ defining
the deformation $\sA_t$ whereas the map $\beta$
records values of the second multiplication map $\mu_2$.
Note that $\beta$ is adjusted by a correction term (using $\lambda$) to account for
a deformed $H$-action on the output of $\mu_1$.

Let $\RR_1$ and $\RR_2$ be the degree 1 and degree 2 relations defining $\sA$:
$$\RR_1=\RR\cap W ,
\qquad \RR_2=\RR\cap (W\ot_H W)\subset V\ot V\, .$$
  The resulting collection of filtered relations of degree 1 and 2,
  $$
  \PP=
  (1 -\alpha -\beta)(\RR_2)
  \cup (1-\lambda)(\RR_1)
\quad  \subset H\oplus W \oplus (W\ot_H W) , $$
defines a filtered $k$-algebra $$\BBB =  T_{H}(W)/(\PP)\, .$$
The $k[t]$-homogenized version
  $$    \PP_t=
  (1 -\alpha t -\beta t^2)(\RR_2)
  \cup (1-\lambda t)(\RR_1)
  \quad\subset T_{\! _H}(W)[t]
\, 
  $$
  defines a graded $k[t]$-algebra   for $\deg t=1$:
  $$\BBB_t =  T_{\! _H}(W)[t]/(\PP_t)\, .$$

  {\bf Homogenized relations lie in the kernel.}
  We show that $\PP_t$ lies in the kernel of $\psi:
  T_H(W)[t]\rightarrow \sA_t$.
  First consider the relations of degree $1$:
As $\psi\big|_H=1_H$ (for $H$ identified with $H+(\RR)$),
\cref{psi} implies that
$$
\psi\, (1_W-\lambda t)
=\psi\, 1_W -\psi\, \lambda t
=\psi-\lambda t
= \pi_{_\sA}
\qquad\text{ on $W$.}
$$
Thus on $\RR_1=\RR\cap W$,
the function
$\psi\, (1_W-\lambda t)$ must vanish
(since $\pi_{_\sA}(\RR)=0$), so 
\begin{equation}\label{R1}
(1-\lambda t)(\RR_1)\subset \ker \psi
\, .
\end{equation}
The relations of degree $2$ require a more subtle argument.
For $v$, $v'$ in $V$,
   $$
   \begin{aligned}
   \psi(v\ot_{H} v')
   &=\psi(v)\star \psi( v')  = \bar v \star \bar v'
   =
 \bar v \, \bar v' 
   + \mu_1(\bar v \ot \bar v' ) t 
   + \mu_2(\bar v \ot \bar v' )t^2\, 
   .
 \end{aligned}
 $$
 We rewrite this succinctly making the identification
 of $V\ot V$ with a subspace of $W\ot_H W$
 temporarily explicit via the map $f: v\ot v'\mapsto v\ot_H v'$
 for clarity:
    \begin{equation}\label{Onf}
   \begin{aligned}
   \psi \, f
   =\pi_{_\sA}\, f
      + \bar\mu_1 \, t 
   + \bar\mu_2\, t^2\, 
   \qquad\text{ on } V\ot V
\, 
 \end{aligned}
 \end{equation}
 for $\bar\mu_i=\mu_i\, (\pi_{_\sA}\ot \pi_{_\sA})$,
 the function $v\ot v'\mapsto 
 \mu_i(\bar v\ot \bar v')$.
Then on $V\ot V$,
$$
\begin{aligned}
  \psi \, \big(
  f
  - \alpha\,  t
  - \beta\,  t^2
  \big)
= (  \pi_{_\sA}\, f+\bar\mu_1\, t +\bar\mu_2\, t^2)
  -\psi\, \alpha\, t
  -\beta\, t^2
    \, 
\end{aligned}
$$
by \cref{Onf}, as $\psi\, \beta=\beta$ as $\beta$ maps to $H$ and
$\psi\big|_{H}=1_H$. Then by \cref{ConstructingParameters}
and \cref{section},
$$
\begin{aligned}
  \psi \, \big(
  f
  - \alpha t 
  - \beta\,  t^2
  \big)   
&=(  \pi_{_\sA}\, f+\bar\mu_1\, t +\bar\mu_2\, t^2)
  - \psi\, (\iota_{_\sA}\,\bar\mu_1)\,  t
  -(\bar\mu_2   -\lambda\,\iota_{_\sA}\, \bar\mu_1)\, t^2
\\
&=
(  \pi_{_\sA}\, f+\bar\mu_1\, t +\bar\mu_2\, t^2)
  - (\pi_{_\sA}+ \lambda\, t)(\iota_{_\sA}\, \bar\mu_1)\, t
  -\bar\mu_2 \,t^2  +\lambda\,\iota_{_\sA}\, \bar\mu_1\, t^2
  \\
  &=
    \pi_{_\sA}\, f +\bar\mu_1\, t +\bar\mu_2\, t^2
    - \pi_{_\sA}\, \iota_{_\sA}\, \bar\mu_1\, t
    - \lambda\, \iota_{_\sA}\, \bar\mu_1\, t^2
  -\bar\mu_2 \,t^2  +\lambda\,\iota_{_\sA}\, \bar\mu_1\, t^2
  \\
  &= \pi_{_\sA}\, f
\end{aligned}
$$
using \cref{psi} for the second equality
noting that $\iota_{_\sA}\, \bar\mu_1$ on $V\ot V$ lies in $W$.
Thus
$$
\psi\big( v\ot_H v' -\alpha(v\ot v')t - \beta(v\ot v')t^2 \big)
=
\pi_{_\sA}(v\ot_H v')
\quad\text{ for all $v,v'$ in $V$}.
$$
We reidentify $V\ot V$ with a subspace of $W\ot_H W$
and write, with a slight abuse of notation, 
$$
\psi\, (1-\alpha t-\beta t^2)=\pi_{_\sA}
\qquad\text{ on }\quad V\ot V\, .
$$
Then since $\pi_{_\sA}(\RR)=\RR/\RR=0$,
\begin{equation}\label{R2}
(1
  - \alpha t
  - \beta t^2)(\RR_2)
  \subset
  \ker \psi
  \, .
\end{equation}
Hence 
$\PP_t\subset \ker \psi$ by \cref{R1,R2}.

{\bf Strong PBW deformation.}
The graded algebra map
$\psi$ of \cref{PsiDef}
thus induces a surjection of graded algebras
for $B_t$ the $k[t]$-homogenized version of $B$
(see \cref{LH}):
      $$ \BBB_t =  T_{\! _H}(W)[t]/(\PP_t)
   \ \longrightarrow\ T_{\! _H}(W)[t]/\ker \psi
   \cong \sA_t
   \, .
   $$
One may verify that
taking the fibers at $t=1$ gives
a surjective map of filtered algebras,
 $$ \Theta:\
 \BBB 
 =\BBB_t|_{t=1}  \  \longrightarrow\ \sA_t|_{t=1} \,  ,
 $$ 
which in turn induces a surjective map of the associated graded 
algebras, 
 $$ \Theta_{\gr}:\ 
 \gr \BBB \
 \longrightarrow\ \gr (\sA_t|_{t=1}) \cong \sA 
 \, ,$$
 which takes $x+(\PP)+\mathcal{F}_{n-1}$
 to $\psi(x)$ when 
$x$ in $T_H(W)$ is homogeneous of degree $n$.
 
 To conclude that $\BBB$ is a strong PBW deformation of $\sA$,
 we argue
that $\Theta_{\gr}$ is left inverse to the  canonical graded
algebra epimorphism
of \cref{Epimorphism},
 $$
 \Theta_{\gr}':
 \sA 
 \longrightarrow
 \gr(\BBB)
 \, ,
 $$
which maps an element $x+(\RR)$
 to
 $x + (\PP) + \mathcal{F}_{n-1}$
 for $x$ of filtered degree $n$.
Under the composition $\Theta_{\gr}\, \Theta_{\gr}'$,
 the elements $h+(\RR)$ and $v+(\RR)$ in $\sA$
 are mapped to
$\psi(h) 
 =h+(\RR)$ and $\psi(v)
 =v+(\RR)$, respectively, for $h$ in $H$ and $v$ in $V$.
Thus $\Theta_{\gr}\, \Theta_{\gr}'=1_{_\sA}$ as
 $\Theta_{\gr}\, \Theta_{\gr}'$ is an algebra homomorphism
 and $\sA$ is generated by these cosets (as $W$ is a free $H$-bimodule
 generated by $V$).
 Thus $\Theta_{\gr}'$ is also one-to-one
 and $\BBB$ is a strong PBW deformation by definition with
 $\sA\cong \gr(\BBB)$  as graded algebras.

{\bf PBW deformation isomorphic to fiber of deformation.}
Lastly, we argue by induction that the isomorphism $\sA\cong \gr \BBB$
forces the homomorphism of filtered algebras
$$
\Theta: \BBB   \longrightarrow \sA_t|_{t=1} \, $$
to be an isomorphism.
Consider the restriction to the 
$i$-th filtered pieces $\F_i(\BBB)$ and $\F_i(\sA_t|_{t=1})$,
a vector-space map:
$$
\Theta^i: \F_i(\BBB)     \longrightarrow \F_i(\sA_t|_{t=1}) \,  . $$
For $i=0$,
we just have the identification map $\Theta^0: \Ho \rightarrow \Ho$. 
Assume by induction that $\Theta^i$ is an isomorphism
of vector spaces.  Then the $5$-Lemma
applied to the commutative diagram
$$
\begin{aligned}
  0\ \  \longrightarrow &&&\ \  \BBB^{i} &&&
  \lhook\joinrel\longrightarrow 
  &&&\ \  \ \BBB^{i+1} &&&
  \relbar\joinrel\twoheadrightarrow 
  &&& \ \  \ \ \BBB^{i+1}/\BBB^i &&
  \longrightarrow\ \  0
  \\ &&& \ \ \ \big\downarrow\ \  &&&  &&& \ \ \ \ \big\downarrow \ \
  &&& &&& \ \ \  \ \ \ \ \big\downarrow 
  \\
  0\ \ \longrightarrow &&& (\sA_t|_{t=1})^{i}\  &&&
  \lhook\joinrel\longrightarrow 
  &&& (\sA_t|_{t=1})^{i+1} &&&
  \relbar\joinrel\twoheadrightarrow 
  &&& (\sA_t|_{t=1})^{i}/ (\sA_t|_{t=1})^{i+1} \ &&
  \longrightarrow\ \  0\ 
\end{aligned}
$$
(connecting two short exact sequences)
implies that $\Theta^{i+1}$
is also a vector space isomorphism
since
$\gr \BBB \cong \gr(\sA_t\big|_{t=1})$
as vector spaces, where we have denoted each $i$-th filtered
component with a superscript $i$.
Thus $\Theta$ is bijective and defines an isomorphism 
of filtered algebras, $\BBB \cong \sA_t|_{t=1}$. 

Finally, note that the fiber $B_t | _{t=0}$ is isomorphic to $A$.
The fiber $A_t | _{t=c}$
for $c\in k - \{0\}$ 
may be handled just as the case $t=1$
since $\gr (A_t | _{t=c})\cong A$.
\end{proof}

\section{Hopf algebra actions and smash products}
\label{sec:smashproducts}

We now apply our results to the setting of
Hopf algebras acting on quadratic algebras.
We see that strong PBW deformations of the resulting smash product algebras
are fibers of graded deformations and vice versa.
In the next section, we will specialize to Hopf algebras
acting
on Koszul algebras in order to study Hopf-Koszul Hecke algebras
as PBW deformations.

\subsection*{Hopf algebra}
Let $\Ho$ be a Hopf algebra over the field $k$ 
with bijective antipode $\gamma$, 
coproduct $\Delta$, 
and counit $\varepsilon$.
We write 
the coproduct $\Delta$  (symbolically)
using Sweedler notation:
$\Delta(h)=\sum h_1\ot h_2$.
We identify
$\bar H=H/k1_H$ 
with a subspace of $H$ by choosing
a section of the vector space map
$H \twoheadrightarrow H/ k 1_H$
(so $H= \bar H \oplus k1_H$).
The section $\bar H$ may be chosen as the kernel of the counit
$\varepsilon: H\rightarrow k$.
In the case $H = kG$ is a group algebra, another good choice
is to take $\bar H$ to be the vector space spanned by all
nonidentity group elements. 

\subsection*{Hopf algebra actions}
Let $S$ be a $k$-algebra carrying a Hopf action of $\Ho$, i.e.,
$S$ is a (left) Hopf module algebra over $\Ho$.
We denote the action of $\Ho$ on $S$ by $s\mapsto\ {}^{h} s$ for $h$ in $\Ho$
and $s$ in $S$.
Thus $S$ is a $k$-algebra which is a left $\Ho$-module
satisfying 
$\, ^{h}(ss')= \sum \, ^{h_1}(s)\, ^{h_2}(s')$
and 
$\ ^{h}1_S = \varepsilon(h) 1_S$
for $h$ in $\Ho$ and $s,s'$ in $S$.
We take the induced right action of $\Ho$ on $S$ as well given by
$s \cdot h = {}^{\gamma^{-1}(h)}s$.
When $S$ is graded, we assume the action of $\Ho$ preserves
the grading, i.e.,
$\deg( ^{h}s)= \deg(s)$ for $s$ in $S$ homogeneous
and $h$ in $H$.

\subsection*{Smash product algebra}
The smash product algebra
$S\# \Ho$ is the free $S$-module with basis given by
a vector space basis of $\Ho$ 
and multiplication
\[
   (sh) \cdot (s'h') = \sum s ({}^{h_1}s')\, h_2 h'
   \quad\text{for all}\quad s,s' \in S\ \text{ and }\ h, h'\in H . 
 \]

 \vspace{1ex}
 
 \subsection*{Smash Relations}
 Now suppose that $S$ is a finitely generated (graded) quadratic algebra.
Then
$S$ is generated by a finite 
dimensional vector space $V$ in degree $1$
(with grading preserved by the action of $H$)
and $S\# H$ is graded with $H$ in degree $0$.
In addition,
\begin{equation}\label{quad-relns}
  \text{ $S = T(V)/(\R)$
for some $k$-subspace of {\em quadratic relations}
$\R\subset V\ot V$.}
\end{equation}
The role that the action of $H$ on $S$ plays in the structure of
the ring $S\# H$ is captured by the space of
{\em smash relations} in $T_H(H\ot V\ot H)$, 
\[
  \R'
  = \Span_k \{ h\ot v\ot 1_H - 1_H\ot \tau(h\ot v)  : 
   h\in \Ho, v\in V  \} \subset \Ho\ot V\ot \Ho  ,
    \]
where $\tau(h\ot s) = \sum \ ^{h_1}s\ot h_2$
for $s\in S$, $h\in \Ho$.
As the identity element $1_H$ of $H$ acts as
the identity map on $S$, its contribution to $\R'$ is 0:
\begin{equation}\label{TauOnScalars}
   1_H\ot  v\ot 1_H  - 1_H\ot \tau(1_H\ot v)  
  \ \ = \ \ 1_H\ot v\ot 1_H - 1_H \ot v\ot 1_H \ \ = \ \ 0 .
\end{equation}
Accordingly, in the definition of $\R'$, we may replace $H$
by $\bar H$ since
the relations are $k$-linear, i.e.,
\begin{equation}\label{smash-relns}
   \R' = \Span_k \{ h\ot v\ot 1_H - 1_H\ot \tau(h\ot v)  : 
   h\in \bar{\Ho}, v\in V  \} \subset \Ho\ot V\ot \Ho 
   \ \ \text{ (smash relations)} . 
 \end{equation}

 \vspace{1ex}

\subsection*{Smash product as a quotient algebra}
In the next lemma, we represent the smash product $S\# \Ho$
as a quotient of the tensor algebra  over $H$ of
the $\Ho$-bimodule $\Ho\ot V\ot \Ho$.
Throughout, we identify $\R\subset V\ot V$ 
with a subspace of the tensor algebra $T_{H}(\Ho\ot V \ot \Ho)$
by identifying 
$V\ot V$ with
$$
(k1_H\ot V \ot k1_H)\ot_H (k1_H\ot V \ot k1_H)\subset
T^2_\Ho(\Ho\ot V\ot \Ho)
\, .
$$

\vspace{1ex}

\begin{lemma}
  \label{RewritingSmashProduct}
  There is an isomorphism of graded algebras
\[
  S\# \Ho \cong T_{\! _H}(\Ho\ot V\ot \Ho)/(\R\cup \R').
  \]
  \end{lemma}
  \begin{proof}
Extend the vector space map
  $k1_\Ho\ot V\ot 1_\Ho k\rightarrow S\ot H$,
  $1_\Ho\ot v\ot 1_\Ho\mapsto v\ot 1_H$
to an $H$-bimodule map
$H\ot V \ot H \rightarrow S\# H$.
This then extends to a graded algebra homomorphism
$T_{H}(H\ot V\ot H)\rightarrow S\# H$
mapping the degree~$0$ component
isomorphically to the subalgebra $H$ of $S\# H$.
Since $S\# H$ is generated by $V$ and $H$, this algebra
homomorphism is surjective.
Clearly the kernel contains the ideal $(\R\cup \R')$
and we obtain a surjective graded algebra homomorphism
$$
f: T_{\! _H}(H\ot V\ot H)/(\R\cup \R')
\rightarrow S\# H
\, .
$$
We argue the kernel of this map is trivial.

First notice that for $h,h'$ in $H$ and $v$ in $V$, 
$$
1_H\ot \tau(h\ot v)h'
\equiv
h\ot v \ot h'
\equiv
h\ot \tau(\tau^{-1}(v\ot h') )
\equiv
h\tau^{-1}(v\ot h') \ot 1_H 
\quad\text{modulo } (\R')
\, 
$$
and hence
$
k1_H\ot V \ot H
\equiv H\ot V\ot H
\equiv H\ot V \ot k1_H$
modulo $ (\R')$.
Thus modulo $(\R')$,
 $ (T_{H}(W))_n$ is
$$\begin{aligned}
  (H\ot V\ot H) & \ot_H (H\ot V\ot H)\ot_H 
 \cdots \ot_H(H\ot V\ot H) 
  \\  &\equiv
 (k1_H\ot V\ot k1_H) \ot_H   
\cdots\ot_H
(k1_H\ot V\ot k1_H) 
\ot_H (k1_H\ot V\ot H)
\, .
\end{aligned}
$$
We identify this set as a vector space with $V^{\ot n}\ot H$.
Now we take a further quotient by $(\R ' \cup \R)$.
Note that $\R\subset V\ot V$, so $(\R)$ may be identified with
$(\R)_S \ot H$ under the above identification, where 
$(\R)_S$ is the ideal of $T(V)$ generated by $\R$.
Hence after taking this further quotient, 
we obtain the vector space $T(V)/(\R)_S \ot H \cong S\ot H$.   
Thus the degree $n$ component 
$(T_H(H\ot V\ot H)/ (\R\cup\R'))_n$ is a finitely generated
free right $H$-module
with a basis corresponding to a vector space basis
of the degree $n$ component $(T(V)/(\R))_n=(V^{\ot n}+(\R))/(\R)$ of
$S$.
Then $f$ is a homomorphism of free right $H$-modules
of the same rank
taking a basis to a basis and thus is injective.
\end{proof}

\vspace{1ex}

\subsection*{PBW deformations of smash products}
Recall that
we discussed filtered algebras over $H$
and their homogeneous and graded versions
in \cref{sec:deformations}.
Every filtered algebra $\cH$ with associated graded 
algebra $\gr \cH \cong S\# H$ is isomorphic
to a quotient of 
$T_H(H\ot V \ot H)$, so we lose no generality
by viewing $S\# H$ as a quotient of $T_H(H\ot V\ot H)$
using \cref{RewritingSmashProduct}
and then appealing to \cref{defn:strongPBW}.

\vspace{1ex}

  \begin{definition}{\em 
  We say a filtered algebra $\cH$ is a {\em (strong) PBW deformation} of $S\# H$
(over $H$) if $\cH$ is a (strong) PBW deformation of 
$T_{H}(H\ot V \ot H)/(\R' \cup \R)$ as identified with $S\# H$
as a graded algebra.
In this case, 
$\gr \cH \ \cong \ S\# H$ as graded algebras and
 $$\cH \ \cong\ S\ot H\qquad\quad\text{
   as $k$-vector spaces.}
 $$
}
\end{definition}

\vspace{2ex}

For $S$ a Koszul algebra, we call a PBW deformation of $S\# H$ 
 a {\em Hopf-Koszul Hecke algebra}.
 Special cases include the graded affine Hecke algebras defined by 
 Lusztig~\cite{Lusztig88,Lusztig89}
 and the symplectic reflection algebras introduced by
 Drinfeld~\cite{Drinfeld}, then Etingof and Ginzburg~\cite{EG},
for $H$ the group ring of a finite group.
 We note that these cases in the literature are in fact
 strong PBW deformations (see~\cref{defn:strongPBW}).
 See~\cite{EGG05,FlakeSahi,GK07,Khare,KT10,LSS,WW}
 for other examples.

Every PBW deformation $\cH$ of $S\# H$
is a quotient of $T_{H}(H\ot V\ot H)$ by an ideal of
filtered relations
with a choice of ideal generators 
under which the homogeneous version $\cH_{\hh}$
is $T_H(H\ot V \ot H)/(\R' \cup \R)\cong S\# H$.
Thus every PBW deformation $\cH$ is a quotient 
\begin{equation}\label{eqn:Hlab}
   \chabl \ =\ T_{\! _H}(\Ho\ot V \ot \Ho)/ (\P\cup \P')
\end{equation}
with $V$ in degree $1$ and $\Ho$ in degree $0$ for vector spaces
$$
\begin{aligned}
  \P & =  \{ r -\alpha(r)-\beta(r) : r\in \R\} & & \text{\em (deformed
    quadratic relations)\quad and } \\
 \P' & =  \{ r - \lambda(r) : r \in \R' \} & & \text{\em (deformed smash relations)}
\end{aligned}
$$
determined by some $k$-linear parameter functions 
\[
  \alpha: \R\rightarrow \Ho\ot V\ot \Ho , \ \ \
   \beta: \R \rightarrow \Ho , \ \ \
   \lambda: \R'\rightarrow \Ho \, .
 \]

 \vspace{1ex}
 
 \subsection*{Fibers of graded deformations for smash products}
We show how the PBW deformations of smash products (over $H$)
are fibers of graded deformations and vice versa.
Recall that finite dimensional algebras are always doubly Noetherian.

\vspace{1ex}

 \begin{cor}\label{PBW-iff-deformation}
   Let $\Ho$ be a Hopf algebra acting on a finitely
   generated
    quadratic algebra $S$. Then: 
   \begin{itemize}
   \item Every fiber
of a graded deformation of $S\# \Ho$ is isomorphic 
     to a strong PBW deformation of $S\# \Ho$ as a filtered algebra. 
     \item 
   Every strong PBW deformation of $S\#\Ho$ is isomorphic to a fiber of some 
   graded deformation of $S\#\Ho$ as a filtered algebra.
        \item 
          Every PBW deformation of $S\#\Ho$ is a strong PBW deformation
          provided $H$ is 
   doubly Noetherian.
   \end{itemize}
 \end{cor}
 \begin{proof}
The result follows from  \cref{PBWImpliesDeformation,DeformationImpliesPBW}
with $W=H\ot V \ot H$.
For the third statement, we use
\cref{prop:doubly-PBW}
noting that each
graded component of $T_{H}(W)/(\R\cup \R')$
is a finitely generated $H$-bimodule,
see the proof of \cref{RewritingSmashProduct}.
   \end{proof}
 

\section{Hopf actions on Koszul algebras}\label{sec:actions-on-Koszul}

We again take a Hopf algebra $\Ho$ acting on a graded quadratic
algebra $S$ with action preserving the grading, but now 
consider the case when $S$ is a Koszul algebra.
We explain that the Koszul resolution and the bar resolution of the Koszul
algebra
both carry the Hopf action so that we may use chain maps
in later sections to convert homological information
on graded deformations into concrete algebraic conditions
for PBW deformations.


\subsection*{Koszul algebra and Koszul resolution}
Fix $S = T_k(V)/(\R)$ generated by a finite
dimensional vector space $V$ with
{\em quadratic Koszul relations} given as some subspace
$\R\subset V\ot V$,
where $T_k(V)$ is the tensor algebra of $V$ over $k$.
We take the grading on  $T_k(V)$ with $V$ in degree $1$.
We will use the {\em Koszul resolution} $(K_S)\sd=K\sd$ for $S$,
\begin{equation}
  \label{koszulres}
  K\sd\ : \ \ \cdots\longrightarrow K_3\stackrel{d_3}{\longrightarrow}
  K_2 \stackrel{d_2}{\longrightarrow} 
  K_1 \stackrel{d_1}{\longrightarrow}
  K_0 \longrightarrow S\longrightarrow 0
\end{equation}
where $K_0 = S\ot S$, $K_1=S\ot V\ot S$,
$K_2=S\ot \R \ot S$, and 
$K_n=S\ot \tilde K_n\ot S$ for
\begin{equation}\label{Koszulterms}
  \tilde K_n = \bigcap _{j=0}^{n-2} (V^{\ot j}\ot \R 
               \ot V^{\ot (n-2-j)} ) 
               \quad\text{ for } n\geq 2\,
             \end{equation}
with differential given by the differential for the (reduced) bar 
resolution (see~\cref{eqn:bar-diff})
using the canonical inclusion of each 
$K_n$ into $(\BB_S)_n$.
Here
we identify $V$ with the vector subspace $(V+(\R))/(\R)$ of $S$ in
degree~1.
We also identify $\bar S=S/k1_S$ with the $k$-span of the positively
graded
elements of $S$ so that $V\subset \bar S$.
 
\subsection*{Smash products as twisted tensor products}
The smash product algebra
$S\# \Ho$
may be realized as
a twisted tensor product:
$$ S\# \Ho \cong S \ott \Ho $$
for twisting map
$$
\tau: \Ho \ot S \longrightarrow S\ot \Ho
\quad\text{ defined by }\quad
h\ot s \longmapsto \sum\ ^{h_1} s \ot h_2 
$$
for $h\in \Ho$ and $s\in S$ with inverse twisting map (see \cite{TTP-AWEZ})
$$
{}_{}\hspace{-9ex}
\tau^{-1}: S\ot H \longrightarrow H \ot S
\quad\text{ defined by }\quad
s\ot h \longmapsto \ \sum h_2 \ot { ^{\igamma(h_1)} s} \, .
$$

\subsection*{Hopf algebra preserves the Koszul relations}
The action of $\Ho$ on $S$ restricts to an action  
of $\Ho$ on $V$ since  it preserves the grading. 
This extends to an action on $V\ot V$, 
$$\, ^h(v\ot v')=\sum \, ^{h_1}v\ot \, ^{h_2} v'
\quad\text{ for $h$ in $\Ho$ and $v,v'$ in $V$.}
$$
The Hopf algebra $\Ho$ 
must preserve the set of quadratic Koszul relations 
$\R\subset V\ot V$
defining the Koszul algebra $S=T(V)/(\mathcal{R})$
since it preserves the grading on the quotient $S$: 
\begin{equation}
   \label{hopfpreserveskoszulrelations}
  ^h \mathcal{R}
\subset\R\quad \text{ for all }
h\in \Ho\, 
\end{equation}
since
$\, ^h(v\ot v' + (\R))
=\, ^h\big( (v+\R) \cdot (v' + (\R)\big)
= \, ^h(v\ot v')+(\R) \, 
$ 
for  $v$, $v'$ in $V$.

\subsection*{Standard embedding of Koszul into bar resolution}
We use the standard embedding of the Koszul resolution $(K_S)\sd=K\sd$
into the reduced bar resolution $\BB_{S}$ of $S$ and denote this
injective chain map by
\begin{equation}\label{KoszulStandardEmbedding}
  \iota_{_S}: K_S\longrightarrow \BB_S\, .
\end{equation}
(See~\cref{sec:deformations} for details on the bar $\B_S$ and reduced
bar $\BB_S$ resolutions of $S$.)
Note that
\begin{equation}\label{eqn:iotaSn}
  (\iota_{_S})_n: K_n\hooklongrightarrow S\ot V^{\ot n}\ot S
\longtwoheadrightarrow S\ot S^{\ot n} \ot S
=(\B_S)_n \longtwoheadrightarrow (\BB_S)_n
\end{equation}
in degree $n$ is given as $1\ot \tilde{\iota}^{ \ot n}\ot 1$
for $\tilde{\iota}:
V\hookrightarrow T(V)\twoheadrightarrow T(V)/(\mathcal{R})=S
\twoheadrightarrow \bar{S}$.
%


\subsection*{Koszul resolution carries an action of Hopf algebra}
In view of \cref{Koszulterms} and
\cref{hopfpreserveskoszulrelations}, 
each term $K_i$ of the Koszul resolution 
$K\sd=(K_S)\sd$ of $S$ is a left $H$-module 
with action denoted $x\mapsto \, ^hx$
induced from the action of $H$ on $S\ot V^{\ot i}\ot S\supset K_i$
given by 
\begin{equation}\label{HopfActionKoszulResolution}
  ^h(s_1\ot v_2\ot \cdots \ot v_{n+1}\ot s_{n+2})=
\sum \  ^{h_1}s_1\ot\, ^{h_2}v_2\ot\cdots \ot \, ^{h_{n+1}}v_{n+1}\ot 
\, ^{h_{n+2}}s_{n+2}
\, . 
\end{equation}
The Koszul resolution 
$K_S$ of $S$ then 
{\em carries the action of $H$} in the sense of 
\cite{TTP-AWEZ}: 
$$
\, ^{h} (sxs')=\sum\ \,(  ^{h_1}s)\, ( ^{h_2}x)\, ( ^{h_3} s') 
\text{
  for $x$ in $K_i$, $s, s'$ in $S$, and $h$ in $H$.}
$$

\subsection*{Reduced bar resolution of $S$ carries an action of Hopf algebra}
Since the action of $H$ preserves the subspace $k\cdot 1_S$ of $S$, it preserves
the kernel of the projection map $\pr_{\BB_S}: \B_S \rightarrow \BB_S$
given by
$1_S\ot \pr_{\bar S}\ot\cdots\ot  
\pr_{\bar S}\ot 1_S$ 
for 
$\pr_{\bar S}: S\twoheadrightarrow \bar S$. 
Each term of $\BB_S$ is thus also a left $H$-module 
with action denoted $x\mapsto \, ^hx$
given by 
\begin{equation}\label{HopfActionReducedBar}
  ^h(s_1\ot s_2\ot \cdots \ot s_{n+2})=
 \sum \  ^{h_1}s_1\ot \, ^{h_2}s_2\ot\cdots \ot 
\, ^{h_{n+1}}s_{n+2}
\, 
\end{equation}
The reduced bar resolution
$\BB_S$ then also
{\em carries the action of $H$} in the sense of 
\cite{TTP-AWEZ}: 
$$
\, ^{h} (sxs')=\sum\ \, ( ^{h_1}s)\, ( ^{h_2}x)\, ( ^{h_3} s') 
\text{
  for $x$ in $(\BB_S)_i$, $s, s'$ in $S$, and $h$ in $H$. }
$$

\section{Chain maps for Hopf actions on Koszul algebras}\label{sec:spr}

We apply the theory of \cite{TTP-AWEZ} to PBW deformations arising
from a Hopf algebra $H$ acting on a graded Koszul algebra $S$
with action preserving the grading on $S$.
We recall from~\cite{SW-twisted,TTP-AWEZ} the twisted tensor product
construction for resolutions and chain maps between resolutions for
$S\# H$. 
These constructions will be used in the remainder of the paper to obtain
information on deformations.

\subsection*{Converting from 
  reduced to unreduced bar resolutions}
As before, we identify 
$\bar{\Ho}=H/k1_H$ with a subspace of $H$ using a choice of section of
the projection $H\twoheadrightarrow  H/ k1_H$.
We identify $\bar S=S/k1_S$ with the subspace of the graded
algebra $S$
generated by positively graded elements, $\bar S = \oplus_{i>0} \
(V^{\ot i}+(\R))/(\R)$.
Then when working in $A=S\#H\cong S\ot H$ (vector space isomomorphism),
we may identify $\bar H$ and $V$ with vector subspaces of $\bar A\subset A$,
\begin{equation}
  \label{ChoiceOfSection}
  \Ho=\bar{\Ho}\oplus k1_\Ho \subset \bar A,\quad
  S=\bar{S}\oplus k1_\Ho\subset \bar A,
  \quad\text{ and }\quad
  S\# H\cong \ttp=\overline{\ttp}\oplus k(1_{S\ott H})
\end{equation}
after identifying $\bar A$ with the subspace
$(\bar S\ot \bar H)\oplus (\bar S \ot H)\oplus (S\ot 
\bar H)$ of $\ttp$.
Note here that $1_{\ttp}=1_S \ot 1_H$.
This gives embeddings of $\ttpp$-bimodules in each degree:
$$(\BB_\Ho)_n\subset (\B_\Ho)_n,\quad
(\BB_S)_n\subset (\B_S)_n, 
\quad\text{ and }\quad
(\BB_{\ttp})_n \subset (\B_{\ttp})_n
\, .
$$
%

\subsection*{A clarifing liberty with notation}
To avoid cumbersome notation and
clarify arguments,
whenever we are
working in $A\ot A$ or $\bar A \ot \bar A$,
we identify the $k$-vector spaces
\begin{equation}\label{liberty1}
  \text{
  $V\ot \bar H$, \ $\bar H \ot V$, \ $V\ot V$, and $\R$
    with 
  subspaces of $\bar A \ot \bar A \subset A\ot A$}
\end{equation}
for 
$\bar H$ and $V$ identified with 
subspaces of $\bar A\subset A$, see \cref{ChoiceOfSection}.
When working in the bar complexes $\B_A$ and $\BB_A$,
we  identify $A$ with $S\ot H$ as a vector space 
and suppress tensor symbols in $A$  to avoid confusion 
with tensor symbols in the bar complex, writing 
$aa'$ for the product of $a$ and $a'$ in $A$. 
We also suppress extra tensor factors of $1_A$ on the outside
and identify
\begin{equation*} 
 \text{ 
   $V\ot \bar H$, \ $\bar H \ot V$, \ $V\ot V$, and $\R$}
\end{equation*}
    with 
    $k$-vector subspaces of
    $(\BB_A)_2 = A\ot \bar A \ot \bar A \ot A
  \subset A\ot A \ot A \ot A = (\B_A)_2$
via 
\begin{equation}\label{liberty2}
  \begin{aligned}
V\ot \bar H \ &\cong\ k1_A \ot V \ot \bar H \ot
k1_A 
&\subset\ &(\BB_A)_2 \subset (\B_A)_2\, ,
\\
\bar H\ot V \ &\cong\  k1_A \ot \bar H \ot V \ot k1_A 
&\subset\ &(\BB_A)_2 \subset (\B_A)_2\, ,
\\
V\ot V \ &\cong\ k1_A \ot V \ot V \ot k1_A 
&\subset\ &(\BB_A)_2  \subset (\B_A)_2
\, .
\end{aligned}
\end{equation}
For example, we write $v\ot w$ for $1_A\ot (v\ot 1_H)\ot (w\ot
1_H)\ot 1_A$ in $\BB_A$ with $v$, $w$ in $V$.
When working in $\B_A$, the set of smash relations (see
\cref{smash-relns})
\begin{equation*}
   \R' = \Span_k \{ h\ot v\ot 1_H - 1_H\ot \tau(h\ot v)  : 
   h\in \bar H, v\in V  \} \subset \Ho\ot V\ot \Ho 
 \end{equation*}
is identified as a $k$-vector space with 
 \begin{equation}\label{liberty3}
   \R' = \Span_k \{ h\ot v - \tau(h\ot v)  : 
   h\in \bar H, v\in V  \} \subset (H\ot V)\oplus (V\ot H) 
  \subset (\B_A)_2 
   \, . 
 \end{equation}
 When working in $\BB_A$,
 we identify $\R'$ in $(\B_A)_2$
 with $\pr_{\BB_A} \R'$ in $(\BB_A)_2$, the image under the projection map
 $\pr_{\BB_A}:\B_A\rightarrow \BB_A$,
 so that
 \begin{equation}\label{liberty4}
   \R' = \Span_k \{ h\ot v- (1\ot \pr_{\bar H})\, \tau(h\ot v)  : 
   h\in \bar{\Ho}, v\in V  \} \subset (\bar H\ot V)\oplus (V\ot \bar
   H)   \subset (\BB_A)_2
 \end{equation}
 for projection map $\pr_{\bar H}: H\rightarrow \bar H$.
 Note that ``identify'' is appropriate terminology here
 as the projection map $\pr_{\BB_A}$ is injective on $\R'\subset \B_A$.

\subsection*{Twisting two resolutions}
We review the construction of the {\em twisted product resolution} obtained by twisting
the reduced bar or Koszul resolution of $S$
with the reduced bar resolution of the Hopf algebra $\Ho$,
see~\cite{SW-twisted,TTP-AWEZ}. 
For resolutions 
$$
\begin{aligned}
  C\sd:\quad \cdots \longrightarrow & \ \, C_2 \longrightarrow
  C_1 \longrightarrow C_0 \longrightarrow 0
\ \ \text{ of } S \text{ as an $S$-bimodule and}\\
D \sd:\quad \cdots \longrightarrow & \ D_2 \longrightarrow D_1 \longrightarrow D_0 \longrightarrow 0
\ \ \text{ of } H \text{ as an $H$-bimodule}\, 
\end{aligned}
$$
with $D\sd=\BB_H$  and either $C\sd=K_S$ or $C\sd=\BB_S$,
we may form
the twisted product resolution $C\sd \ott D\sd$
(see \cite{SW-twisted,TTP-AWEZ})
as 
the total complex $(C\ott D)_{\sd}$ of the double complex 
$C\sd\ot D\sd$, 
\begin{equation}\label{eqn:res-X}
 (C\sd\ot D\sd)_n = \bigoplus_{i+j=n} C_i \ot D_j\,
  ,
\end{equation}
with $(S\# \Ho)$-bimodule structure on each $C_i \ott D_j$ 
defined as follows~\cite[Example 10.8]{TTP-AWEZ}.
The reduced bar resolution $D\sd = \BB_H$ of $H$
is a left $H$-comodule via the map $\rho : \BB_H \rightarrow H\ot \BB_H$
given in each degree~$n$ by 
\[
  \rho (h^0\ot h^1\ot\cdots \ot h^{n+1} ) = 
   \sum (h_1^0h_1^1\cdots  h_1^{n+1} ) \ot h_2^0 \ot h_2^1\ot 
    \cdots\ot h_2^{n+1}
\]
for $h^i\in \bar{H}$ ($1\leq i\leq n$)
and $h^0, h^{n+1}\in H$.
Here, we suppress notation for projection $\pr_{\bar H}$
applied to  $h^1_2,\ldots, h^n_2$
which is required 
as the coproduct on $H$
does not necessarily take $\bar H$ to $\bar H \ot \bar H$.
It can be checked that 
the properties of a left comodule are satisfied due to 
coassociativity 
of $\Delta$, compatibility of $\Delta$ with multiplication,
and the counit property
by invoking a choice of internal direct sum of vector spaces
$H = \bar H\oplus k 1_H$.
This comodule structure gives rise to the left and right actions of $S$ and $H$
on each $C_i\ot_{\tau} \BB_H$:
\begin{equation}
  \label{TwistedResolutionAction} 
  \begin{aligned}
   s \cdot (x\ot y)\cdot h
   \ &= \    s x\ot y h,
\\  h \cdot (x\ot y)\cdot s 
   \ &= \
   \sum\ ( {}^{h_1}x) \ ( {}^{h_2y_1}s) \ot 
   h_3 y_2  
\quad\text{ for }\ \ 
h\in \Ho,\ x\in C , \ y\in \BB_H,\ s\in S\, ,
\end{aligned}
\end{equation}
after identifying $h$ with $1_S\ot h$ and $s$ with $s\ot 1_H$
in $S\#H$ for comodule structure on $\BB_H$ denoted briefly by
$y\mapsto \sum y_1\ot y_2$.
Then $(C\ott \BB_H)\sd$ augmented by $S\# \Ho$
is an exact sequence of $(S\# \Ho)$-bimodules,
see~\cite[Lemma~3.5]{SW-twisted} and~\cite{TTP-AWEZ}.

\subsection*{Twisted product of Koszul and bar resolutions}
The twisted product resolution
$X\sd=K_S\ott \BB_H$ 
combining the Koszul resolution $K_S$ of $S$ with the reduced
bar resolution $\BB_\Ho$ of $\Ho$
has $n$-th term
\begin{equation}\label{eqn:xij}
  X_n=\bigoplus_{i+j=n}X_{i,j}
  \qquad\text{ for } X_{i,j} = (K_S)_i\ot (\Ho\ot\bar{\Ho}^{\ot j}\ot \Ho) 
  \qquad i,j\geq 0\, . 
\end{equation}


\begin{remark}\label{LibertyOnX}
  {\em 
Note that $X\sd$ is a {\em free} $(S\# \Ho)$-bimodule resolution of $S\# \Ho$:
$K_S$ is a free $S^e$-module in each degree and $\BB_H$ is a free
$\Ho ^e$-module in each degree, and it can be shown directly that $X_{i,j}$
is a free $(S\# \Ho)$-bimodule 
(with basis given by tensoring free basis elements of $(K_S)_i$ with those
of $(\BB_H)_j)$.
See~\cite[Section 9]{TTP-AWEZ}:
The twisting map and its inverse may be used
to provide an isomorphism of vector spaces
\[
 (K_S\ot \BB_H)_n \ \cong \  \bigoplus_{\ell + i =n}
   (S\# H) \ot \left( \cap_{j=0}^{\ell -2}  (V^{\ot j}\ot \R\ot V^{\ot (\ell -2 -j)}
   \right) \ot \bar{H}^{\ot i}\ot (S\# H)
\]
compatible with the $(S\# H)$-bimodule structures.
}
\end{remark}

\subsection*{Chain maps for Hopf algebras acting on Koszul algebras}
Twisted Alexander-Whitney and Eilenberg-Zilber maps
were used in~\cite{TTP-AWEZ} to produce chain maps between
resolutions as stated in the following theorem: 

 \begin{thm} \cite[Theorem 11.8]{TTP-AWEZ}\label{HopfKoszulConversion}
   Let $X\sd$
be the twisted product resolution for 
$S\# \Ho$ twisting the Koszul resolution $K_S$ of $S$
with the reduced bar resolution
$\BB_\Ho$ of the Hopf algebra $\Ho$.
There are $(S\# \Ho)$-bimodule graded chain maps
$$
\pi: \BB_{S\#\Ho}   \longrightarrow X\sd 
\quad\text{and}\quad 
\iota: X\sd \longrightarrow  \BB_{S\#\Ho}
$$
satisfying 
$
\pi \, \iota 
= 1_{X}\, . 
$
\end{thm}
We give some needed properties of these maps $\iota$ and $\pi$ in
\cref{AppendixValuesLowDegree}.
We will use these maps to
generalize 
\cite[Theorem 2.5]{SW-Koszul}
and describe all filtered PBW deformations of $S\# \Ho$.

\section{PBW deformations of Hopf actions on Koszul algebras}
\label{sec:ParametersToCochains}
We again consider a Hopf algebra $\Ho$ acting on a graded Koszul 
algebra $S$.
In the next few sections, we 
simultaneously generalize some of the results of 
\cite{Khare}, \cite{SW-Koszul}, and \cite{WW}.

\subsection*{Parameters for PBW deformations of Hopf actions
on Koszul algebras}
We consider PBW deformations of smash products $S\# \Ho$ modeled on Drinfeld
Hecke algebras and Lusztig's affine graded Hecke algebras.
Recall that $\R$ is the space of Koszul relations defining $S=T(V)/(\R)$ and $\R'$ is the space of smash relations (see \cref{smash-relns}),
  $$  \R' = \Span_k \{ h\ot v\ot 1_H - 1_H\ot \tau(h\ot v)  : 
   h\in \bar\Ho, v\in V  \} \subset \Ho\ot V\ot \Ho  \, .
$$
Set $A = S\# H \cong T_H(H\ot V\ot H)/(\R \cup \R')$
(see \cref{RewritingSmashProduct}).

Any PBW deformation of $A=S\# H$ (over $H$) must have the form of a 
quotient algebra
\begin{equation}\label{eqn:Hlab2}
   \chabl = T_{\! _H}(\Ho\ot V \ot \Ho)/ (\P\cup \P')
\end{equation}
(as in \cref{eqn:Hlab})
for vector spaces
$$
\begin{aligned}
  \P & =  \{ r -\alpha(r)-\beta(r) : r\in \R\} & & \text{(deformed
    Koszul relations)}\\
 \P' & =  \{ r - \lambda(r) : r\in \R' \} & & \text{(deformed smash relations)}
\end{aligned}
$$
determined by $k$-linear parameter functions 
\[
  \alpha: \R\longrightarrow H\ot V\ot H , \ \ \
   \beta: \R \longrightarrow H , \ \ \
   \lambda: \R'\longrightarrow H .
 \]
 We may and do replace $\alpha$,
without changing $\chabl$, 
 by a parameter function
 \begin{equation}\label{ReplaceParameter}
\alpha:  \R\longrightarrow V\ot \Ho \ \ \
\end{equation}
with codomain $V\ot \Ho$ identified with the subspace $k1_H\ot V\ot \Ho\subset 
\Ho\ot V\ot \Ho$;
see \cref{ChoiceParameterCodomain} in 
\cref{AppendixParameterCodomain}.
Note that one could instead use 
     a parameter function $\R\rightarrow H\ot V$
without changing $\chabl$; again, see
  \cref{AppendixParameterCodomain}.
   
     \subsection*{Identifying parameter functions with cochains}
     \label{ParametersAsCochains}
 We use the $k$-linear functions 
\[
   \alpha: \R \longrightarrow V\ot \Ho, \ \ \
    \beta: \R\longrightarrow \Ho, \ \ \
    \lambda: \R'\longrightarrow \Ho
  \]
to define $2$-cochains on the twisted product resolution
  $X\sd=K_S\ott\BB_\Ho$
  of
  $\A=S\# \Ho$,
\begin{equation}\label{CochainParameters}
   \alphax: X_2 \longrightarrow \A, \ \ \
    \betax: X_2\longrightarrow \A, \ \ \
    \lambdax: X_2\longrightarrow \A
\end{equation}
  i.e.,  $\A$-bimodule homomorphisms
  $X_2=X_{0,2}\oplus X_{1,1}\oplus X_{2,0}\rightarrow \A$
   (see \cref{eqn:xij}).

The parameter functions $\alpha,\beta$ extend to 
 $\A$-bimodule homomorphisms on 
$$X_{2,0}=(S\ot \R \ot S) 
\ot (\Ho\ot\Ho)
\cong \A\ot \R\ot \A
$$
(an isomorphism of 
$\A$-bimodules given by iterations of the inverse
of the twisting map $\tau$)
and hence
define the $2$-cochains on $X_{2,0}$
\begin{equation}\label{CochainAlpha}
 \begin{aligned}
   &\alphax:  X_{2,0}\longrightarrow \A,
   \qquad
 &1_S\ot r\ot 1_S
 \ot
 1_H\ot 1_H
 \longmapsto \alpha(r) \in \A, \\
 &\betax:  X_{2,0}\longrightarrow \A,
 \qquad
 &1_S\ot r\ot 1_S
 \ot
 1_H\ot 1_H \longmapsto \beta(r)\in \A
 \,  
 \end{aligned}
 \end{equation}
after identifying the codomain of the parameter functions $\alpha$, $\beta$,  
and $\lambda$ 
with subspaces of $A$ using the vector space
isomorphism $A=S\# H\cong S\ot H$ and identifying
$V$ with $V\ot k1_H$ and $H$ with $k1_S\ot H$.
These maps extend uniquely to 2-cochains 
 $$
 \alphax:  X_2\longrightarrow \A
 \quad\text{ and }\quad 
 \betax:  X_2\longrightarrow \A 
 \qquad\text{ 
 with $\alphax\equiv 0 \equiv \betax$ on $X_{1,1}$ and 
 $X_{0,2}$} . 
$$

The parameter function $\lambda$ similarly defines
an $A$-bimodule homomorphism on 
$$
X_{1,1}
=(S \ot V\ot S) \ot (\Ho\ot \bar H\ot \Ho)
\cong
A\ot V\ot \bar H \ot A
$$
(again, an  isomorphism of 
$A$-bimodules given by iterations of the twisting map $\tau$)
and hence defines a $2$-cocycle on $X_{1,1}$
\begin{equation}\label{CochainLambda}
 \begin{aligned}
\lambdax: X_{1,1} &\longrightarrow A,
\\ \ \ 
(1_S\ot v\ot 1_S)
\ot
(1_H\ot h\ot 1_H)
&\longmapsto
\lambda\big(1_H\ot v\ot h -\tau^{-1}(v\ot h)\ot 1_H\big)
\in H
\, ,
\end{aligned}
\end{equation}
for $h$ in $\bar H$ and $v$ in $V$, which extends uniquely to a $2$-cochain on $X$
$$
 \lambdax:  X_2\longrightarrow \A 
 \qquad\text{ 
 with $\lambdax\equiv 0$ on $X_{2,0}$ and 
 $X_{0,2}$}
\, .
 $$

 \subsection*{Gerstenhaber bracket on twisted product resolution}
\label{subsec:Gb} 
We consider the Gerstenhaber bracket of $2$-cochains
which induces a graded Lie bracket on
Hochschild cohomology,
expressed in terms of the bar complex of $A=S\# H$:
$$
[\mu,\mu']_{_{\BB_{\! A } }}\  :=\  \mu\circ \mu' + \mu'\circ \mu
\qquad\text{ for $2$-cochains } \mu, \mu' \text{ on } \BB_A
$$
for circle product $\circ$ defined by
\begin{equation}\label{circle-prod-defn}
  (\mu\circ\mu' )( 1\ot x\ot y\ot z \ot 1) 
  = \mu(1\ot\mu'(x\ot y) \ot z\ot 1) -
  \mu (1\ot x\ot \mu'(y\ot z)\ot 1)
\end{equation}
for $x,y,z$ in $A$
(see \cite{GerstenhaberSchack})
for $\mu$ extended to $\B_A$ by composing with
$\pr_{\BB_A}:\B_A\rightarrow \BB_A$.
This Gerstenhaber bracket on $\BB_A$ defines a bracket operation
on the resolution $X$
using chain maps $$\iota :X\longrightarrow \BB_A
\qquad\text{ and }\qquad
\pi : \BB_A\longrightarrow X$$ as in  
\cref{HopfKoszulConversion}: we simply send
input in $X$ to
$\BB_A$ using $\iota$, compute  this 
traditional Gerstenhaber bracket of functions on $\BB_A$, 
and then lift the result back to $X$ with $\pi$:
\begin{equation}\label{eqn:Xbracket}
[\eta, \eta']_{\X} = \iota^*[\, \pi^*\eta,\ \pi^*\eta'\ ]_{_{\BB_{\! A}}}
\qquad\text{ for any $2$-cochains } \eta, \eta' \text{ on } X
\, .
\end{equation}

\section{Three cohomological conditions from PBW deformations}
\label{sec:ThreeFromPBW}

We again take a Hopf algebra $\Ho$ acting on a graded Koszul 
algebra $S$ and
consider filtered algebras with homogeneous version $S\# H$.
Here we generalize one implication 
of~\cite[Theorem 5.3]{SW-Koszul}
connecting PBW quotients (over $H$)
to deformations.  We generalize the other implication in the
next section.  We consider equality of
cochains,
not just equality of cohomology classes,
on the twisted product resolution $X=K_S\ott \BB_\Ho$
formed from the Koszul resolution 
  $K_S$ of $S$ and the reduced bar resolution $\BB_\Ho$
  of $H$  (see \cref{eqn:xij}).
 We work with
the filtered algebra $\chabl$ defined in 
\cref{eqn:Hlab2} by some set of parameter functions 
$\alpha,\beta,\lambda$
which define cochains $\alphax, \betax, \lambdax$ on $X$, 
see \cref{CochainAlpha,CochainLambda}.

Recall from \cref{PBW-iff-deformation} that if
$H$ is doubly Noetherian, for example when $H$ is finite dimensional, 
then every PBW deformation
of $S\# H$ is a strong PBW deformation.
In that case, 
the word ``strong" can be removed from 
the hypothesis of the following theorem. 

\vspace{2ex}

\begin{thm}\label{ForwardsDirection}
Any strong PBW deformation of $S\#H$ is 
$\chabl$ for 
some parameters $\alpha,\beta,\lambda$, 
with
   \begin{itemize}
    \item[(a)]\ \ $d^*(\alphax+\lambdax) = 0$, 
    \item[(b)]\ \ $[\alphax+\lambdax,\ \alphax + \lambdax]=2d^*\betax$, 
    \item[(c)]\ \ $[\lambdax + \alphax,\ \betax]= 0$\, 
  \end{itemize}
as cochains, for bracket $[\ ,\ ]=[\ ,  \ ]_{\X}$ on the twisted product resolution 
$X=K_S\ott \BB_\Ho$.
\end{thm} 

\vspace{1ex}

\begin{proof}
We will make heavy use of explicit chain maps 
constructed in~\cite{TTP-AWEZ},
the twisted tensor product analogs of
Alexander-Whitney and Eilenberg-Zilber maps. 

  Recall we write $A=S\# H$ as $T_H(H\ot V\ot H)/(\R \cup \R')$, 
  see \cref{quad-relns}, \cref{smash-relns}, and \cref{RewritingSmashProduct}.
  We again identify $V$ with $k1_H\ot V \ot k1_H\subset H\ot V\ot H$
  and view $\R\subset V \ot V$ as a subspace of $T_H(H\ot V\ot H)$
  in degree $2$.  Likewise, we identify the image of $\alpha$
  with a subspace of $k1_H\ot V\ot H\subset H\ot V \ot H$.
  As a strong PBW deformation of $\A$,
  the algebra $\cH$ can be written as a quotient algebra $\chabl$ for some
   parameters $\alpha,\beta,\lambda$
  (see \cref{sec:smashproducts}), 
  and
  \cref{PBWImpliesDeformation} implies that
$$
\begin{aligned} 
\A_t=(\chabl)_t &=T_{\! _H}(H\ot V\ot H)[t]/( r- \alpha(r)t-\beta(r)t^2, \
r'-\lambda(r')t: r\in \R, r'\in \R')
\end{aligned}
$$
is a graded deformation of $\A$.
By the proof of \cref{PBWImpliesDeformation},
there is a $k$-vector space isomorphism
$$\Psi:\A[t]\longrightarrow (\chabl)_t$$
that extends an arbitrary choice of 
section
\begin{equation}
\A=T_{\! _H}(H\ot V\ot H)/(\R\cup \R') 
\longrightarrow T_{\! _H}(H\ot V \ot H)
\end{equation}
to a graded $k[t]$-vector space isomorphism
given by the composition 
$$
\A[t]\longrightarrow T_{\! _H}(H\ot V \ot H)[t]
\rightarrow T_{\! _H}(H\ot V\ot H)/(\P_t\cup \P'_t) =(\chabl)_t
\, 
$$
for $\P_t=\{r-\alpha(r)t-\beta(r)t^2: r\in \R\}$ and
$\P'_t=\{r-\lambda(r)t: r\in \R'\}$.
We identify $$V\cong k1_H\ot V \ot k1_H \quad\text{and}\quad H\cong k1_H\ot k1_S\ot H$$
with
subspaces of $H\ot V\ot H\subset T_H(H\ot V \ot H)$
and choose the section in \cref{ChoiceOfSection}
so that $h+(\R\cup \R') \mapsto h$,
$v+(\R\cup \R') \mapsto v$ for $v$ in $V$ and $h$ in $H$,
denoting these cosets in $A$ simply by $h$ and $v$, respectively.

Let $\star$ be the multiplication on $A[t]$
induced by the vector space isomorphism $A[t]\cong A_t$.
On one hand, the multiplication $\star$ is given by the
parameters $\lambda$, $\alpha$, and $\beta$
defining 
$(\chabl)_t$ as
$$a\star a'= a'' \quad\text{ in } A[t]
\quad\text{whenever}\quad 
\Psi(a) \cdot \Psi(a')=\Psi(a'')
\quad\text{ in }
(\chabl)_t
\, .
$$
On the other hand, 
the multiplication $\star$
is given in terms of multiplication maps $\mu_i:A\ot A\rightarrow A$ 
of degree $-i$
for $i\geq 1$ as $A_t$ is a deformation:
$$
a\star a' = m(a\ot a') +
\mu_1(a\ot a')t + \mu_2(a\ot a')t^2 + \ldots 
\qquad\text{ for } a,a'\in A \, , 
$$
for $m:A\ot A\rightarrow A$ the multiplication map on $A$.
We describe $\mu_1$ and $\mu_2$
in terms of $\lambda,\alpha,\beta$.

We identify as usual each $k$-linear map
    $\mu_i: \A\ot \A\rightarrow \A$
    with the $\A$-bimodule map $\mu_i: \A\ot A\ot A\ot \A\rightarrow \A$.
    Note that 
    each $\mu_i$ vanishes on both $A\ot k1_A$ and $k1_A\ot A$
    as $1_{\A}$ is the identity of $\A_t$.
    Thus each $\mu_i$ further corresponds to 
    a cochain on the reduced bar resolution $\BB_{\A}$
that we also denote by $\mu_i$:
    $$
      \mu_i: (\BB_{\A})_2=\A \ot \bA\ot \bA\ot \A
      \longrightarrow \A \, .
          $$
          When working in $(\BB_{\A})_2$, we identify
        $V\ot \bar H$, $\bar H \ot V$, $V\ot V$, $\R$, and $\R'$
    with 
  subspaces of $(\BB_A)_2=A\ot \bar A \ot \bar A \ot A$, see
  \cref{liberty1,liberty2,liberty3,liberty4}.
    
    {\bf Relations of degree $1$.}
For a relation
$r=h\ot v\ot 1_H-1_H\ot \tau (h\ot v)$ in $\R'$,
in $(\chabl)_t$, 
\begin{equation*}
\begin{aligned}
\Psi(h)\cdot \Psi(v) 
- & \sum \Psi(\, ^{ h_1}v)\cdot \Psi(h_2)
\\&=
h\ot v\ot 1_H -\sum 1_H\ot\, ^{h_1}v\ot h_2 
+(\P_t\cup\P_t')
\\&=
r+ ( \P_t \cup \P_t')
\\&=
\lambda(r)\, t + ( \P_t \cup \P_t') 
\\&=
\Psi(\lambda(r)) \, t
\end{aligned}
\end{equation*}
so, on one hand, in $A[t]$,
$$
h\star v 
-  \sum\, ^{ h_1}v \star h_2  
=
\lambda\big( h\ot v \ot 1_H 
-1_H\ot \tau(h\ot v)
\big)
\, t
\, .
$$
On the other hand, since $\deg \mu_i=-i$, 
  $$
  \begin{aligned}
   h\star v 
   -  \sum\, ^{ h_1}v \star h_2
 &=
  m(h\ot v) + \mu_1( h\ot v )\, t 
 -
\sum 
   m(\, ^ {h_1} v\ot h_2) + \mu_1(\, ^{h_1}v\ot h_2)\, t  
   \\ &=
   m\big(h\ot v - \tau(h\ot v)\big) 
   + \mu_1\big( h\ot v  -\tau(h\ot v)\big)\, t
     \\ &=
  \mu_1\big( h\ot v  -\tau(h\ot v)\big)\, t
   \, ,
 \end{aligned}
 $$
 as $\R'\subset \ker m$
 (as $\R'$ lies in the ideal of relations defining $A$).
 Thus in $A[t]$,
 \begin{equation}\label{FirstMultMapToLambda}
   \mu_1\big(h\ot v - \tau(h\ot v)\big)
   \ =\ 
   \lambda (h\ot v \ot 1_H - 1_H\ot
   \tau(h\ot v))
   \, .
   \end{equation}

 We apply this observation 
 to $1_H\ot v\ot h - \tau^{-1}(v\ot h)\ot 1_H$
 which lies in $\R'$
 for $h$ in $\bar H$ and $v$ in $V$
 and use the graded chain maps 
$\iota: X\rightarrow \BB_A$ of 
\cref{HopfKoszulConversion}
(as detailed in~\cref{AppendixValuesLowDegree})
to convert cochains
     on the degree 2 term of the reduced bar resolution
     $(\BB_{A})_2$
    to cochains
    on the resolution $X\sd$.
    On $X_{1,1}$, \cref{ExplicitValues} implies
    (see \cref{IotaOnX11,CochainLambda}) 
$$
    \begin{aligned}
    (\iota^*\mu_1)(1_S\ot v \ot 1_S\ot 1_H\ot h \ot 1_H)
    &= \mu_1(v\ot h-\tau^{-1}(v\ot h)) \\
    &=    \lambda\big(1_H\ot v\ot h - \tau^{-1}(v\ot h)\ot 1_H\big)\\
    &= \lambdax(1_S\ot v\ot 1_S\ot 1_H\ot h\ot 1_H) 
    \, 
    \end{aligned}
    $$
    as $\mu_1(1_H\ot *)=0$ and
    $$
    \lambdax =  \iota^* \mu_1
    \qquad\text{ as cochains on } X_{1,1}
    \, .
$$
      
{\bf Relations of degree $2$.}
    Similarly, 
we write a relation $r$ in $\R$ as $r=\sum_i v_i\ot v_i'$.
Then on one hand, in $\chabl$, 

$$
\begin{aligned}
\sum_i \Psi(v_i)\cdot \Psi(v_i')
&=
\sum_i v_i\ot v_i' + (\P_t\cup\P_t')
=
r+ (\P_t\cup\P_t')
\\ &=
\alpha(r)\, t+\beta(r)\, t^2+ (\P_t\cup\P_t')
=\Psi\Big(
\alpha(r)\, t+\beta(r)\, t^2\Big)
\end{aligned}
$$
and hence, in $\A[t]$,
$$
\begin{aligned}
\sum_i v_i\star v_i'
=
\alpha(r)\, t+\beta(r)\, t^2
\, .
\end{aligned}
$$
On the other hand, viewing $\R$ as a subset of $\bar A\ot \bar A \subset A\ot A$
(see \cref{liberty1}),
$$
    \begin{aligned}
      \sum_i  v_i\star v'_i 
&=
      \sum_i m(v_i\ot v'_i) + \mu_1(v_i\ot v'_i)t + \mu_2(v_i\ot
      v'_i)t^2
\\&= m(r) + \mu_1(r)t + \mu_2(r)t^2 
= \mu_1(r)t + \mu_2(r)t^2 
    \end{aligned}
    $$
    as $m\equiv 0$ on $\R\subset A\ot A$.
Thus
$
\alpha(r)\, t+\beta(r)\, t^2 = \mu_1(r)t + \mu_2(r)t^2 
\ \text{ for } r \in \R,
$
       i.e.,
\begin{equation}\label{MultiplicationMapsAreAlphaBeta}
    \begin{aligned}
            \alpha
      =
      \mu_1      \qquad\text{ and }
      \qquad
      \beta
      =
      \mu_2
      \qquad\text{ on $\R$}
      \, .
    \end{aligned}
\end{equation}
     On $X_{2,0}$ using the identification of \cref{liberty2},
     \cref{ExplicitValues} implies (see \cref{eqn:r-to-r}) that
$$\iota (1_S\ot r\ot 1_S\ot 1_H\ot 1_H) =
r 
\quad\text{ in } (\BB_A)_2
\qquad\text{ for } r \text{ in } \R \, .
$$
Then by~\cref{CochainAlpha},
$$
\alphax = \iota^* (\mu_1)
\quad\text{ and }\quad
\betax = \iota^*(\mu_2)
\qquad\text{ as cochains on $X_{2,0}$}
  \, .
$$
Then since
$\alphax\equiv \betax\equiv \lambda \equiv 0$
on $X_{0,2}$ and $\alphax\equiv \betax\equiv 0$ on $X_{1,1}$
and $\lambdax\equiv 0$ on $X_{2,0}$,
    $$
    \alphax+\lambdax = \iota^*(\mu_1)
    \qquad\text{and }\qquad
    \betax = \iota^*(\mu_2)
    \, 
    $$
    as cochains on $X_2 = X_{0,2}\oplus X_{1,1}\oplus X_{2,0}$.
    
\vspace{2ex}

{\bf Condition~(a).}
As $\A_t$ is a deformation,  the    cochain $\mu_1$ 
  is a Hochschild 2-cocycle with $[\mu_1,\mu_1]=2\delta^*(\mu_2)$
  and $[\mu_1,\mu_2]$ a coboundary
  (see \cite{GerstenhaberSchack}). 
  It follows that $d^*(\alphax +\lambdax)=0$ since $\mu_1$ is a cocycle,
proving Condition~(a) of the theorem.
 
\vspace{2ex}

{\bf Condition~(b).}
A degree argument shows 
that each side of the equation in Condition~(b) is~0
on both $X_{0,3}$ and $X_{1,2}$:
$\alphax$ and $\betax$ are zero on $X_{0,2}$ and $X_{1,1}$
while $\lambdax$ is zero on $X_{0,2}$ and $X_{2,0}$.
For the right side of the equation, $\betax$ is evaluated on the image
of the differential $d$, and the image of $d$ on
$X_{0,3}\oplus X_{1,2}$ is in $X_{0,2}\oplus X_{1,1}$. 
For the left side of the equation, the bracket is a graded commutator
of circle products (see~\cref{circle-prod-defn})
involving $\alphax + \lambdax$ (see also \cref{eqn:Xbracket}).
Recall from \cref{HopfKoszulConversion} that $\pi\, \iota = 1_X$.
On $X_{0,3}$, both $\alphax$ and $\lambdax$ are 0
on the first two or last two tensor factors (since no elements
of $V$ are in these factors),
and therefore so is any such circle product.
On $X_{1,2}$, any such circle product is also~0:
$\alphax$ takes as input elements of $\R$, and so is $0$ on 
any two tensor factors of elements in $X_{1,2}$,
while $\lambdax$ takes as input pairs of elements of $V$ and $\bar H$,
and outputs elements of $H$, on which $\alphax$,
$\lambdax$ are both $0$.

It remains to show the equation holds 
on $X_{2,1}$ and on $X_{3,0}$. 
Consider $X_{2,1}$ first.
All maps considered are $A$-bimodule homomorphisms,
so to avoid extra components of $A$ on the outside cluttering notation,
we identify
every $A$-bimodule homomorphism 
$\A\ot M \ot A \rightarrow  N$
(giving chain maps and cocycles)
with the corresponding $k$-linear map $M\rightarrow N$
and identify each $A$-bimodule $A\ot M \ot A$ with $M$,
for any $k$-vector spaces $M$ and $N$.
We again use the graded chain maps $\pi: \BB_A\rightarrow X$ and 
$\iota: X\rightarrow \BB_A$ of 
\cref{HopfKoszulConversion}
with $\pi \, \iota =1_X$.

We employ a key observation in degrees $2$ and $3$.
Since $\pi\, \iota$ is the identity map on $X$,
$\iota^*\pi^*$ is the identity map on cochains on $X$,
so $\iota^*\pi^*\iota^*\mu=\iota^*\mu$ for any cochain $\mu$
on $\BB_A$:
\begin{equation}\label{KeyObservation}
 \pi^*\iota^*(\mu)
= \mu
\qquad\text{ on $\ima\iota$}
\, .
\end{equation}

By \cref{ExplicitValues},
$\R'\oplus \R\subset\ima\iota_2\subset (\BB_A)_2$
(with the identifications 
\cref{liberty1,liberty2,liberty3,liberty4}), so
$$\pi^*\iota^*\mu_1 = \mu_1
\qquad\text{ on }\quad
\R\oplus \R'\subset (\BB_A)_2
\, 
$$
by our key observation \cref{KeyObservation}.
This then implies that
\begin{equation}\label{simplermu}
  \pi^*\iota^*\mu_1\ot 1 -  1\ot \pi^*\iota^*\mu_1
  = \mu_1\ot 1-1\ot \mu_1
  \qquad\text{ on }\ \ima\iota_3\big|_{_{X_{2,1}}}\subset (\BB_A)_3
\end{equation}
since by \cref{ExplicitValues},
in our simplified notation of \cref{liberty4}, 
\begin{equation}\label{subset}
  \begin{aligned}
    \ima \iota_3\big|_{_{X_{2,1}}}
    \subset
   \big( (V\ot \R')\oplus (\bar H\ot \R) \big) 
  \cap 
  \big( (\R' \ot V)\oplus (\R\ot \bar H ) \big)
  \subset (\BB_A)_3
  \, .
\end{aligned}
\end{equation}

We now consider the square bracket of $\alpha+\lambda$,
a cochain on $X_{3}$, using \cref{simplermu}
(see~\cref{eqn:Xbracket} and \cref{BracketDefinition}):
\begin{equation}\label{SuppressingInBracket}
  \begin{aligned}
  [\alphax + \lambdax, \, \alphax+\lambdax]_{_{X_{2,1}}} 
  & = \ \iota^*\, [\pi^*(\alphax+\lambdax),\, \pi^*(\alphax + \lambdax) ] \\
   & = \ \iota^*\, [\pi^*\iota^*(\mu_1),\pi^*\iota^*(\mu_1)]\\
  & =\
        2 \iota^*\, \Big(\pi^*\iota^*(\mu_1) \, \pr_{\BB_A}\,
        \big(\pi^*\iota^*(\mu_1)\ot 1
        - 1\ot \pi^*\iota^*(\mu_1)\big)\Big) \\
  & = \
        2\iota^*\big(\pi^*\iota^*(\mu_1) \, \pr_{\BB_A}\, (\mu_1\ot 1 - 1\ot \mu_1)\big)\, 
        \, ,
        \end{aligned}
        \end{equation}
for projection map
$\pr_{\BB_A}:\B_A\rightarrow \BB_A$.
We claim this coincides with
$$2\iota^*\big(\mu_1  (\mu_1\ot 1 - 1\ot \mu_1)\big)\, .$$
To verify this claim, we lift to $\B_A$ and work in $A\ot A$ instead
of $\bar A \ot \bar A$.
First 
denote by $m: \A\ot\A\rightarrow A$ the multiplication map on $A$
and observe that its kernel 
is the $A$-bimodule 
generated by $\R$ and $\R'$ viewed as vector subspaces of 
$A\ot A$ (see \cref{liberty1}).
Next observe that for any $2$-cocycle $\mu$ on $\B_A$,
\begin{equation}\label{CocylesAndMultMaps}
m(\mu\ot 1 - 1\ot \mu)=\mu(m\ot 1 -1\ot m)
\end{equation}
since $d^*\mu=0$ 
implies that
\begin{equation*}
\begin{aligned}
  m(\mu\ot 1 - 1\ot \mu)( a \ot b \ot c )
  &=\mu(a\ot b)c-a\mu(b\ot c)
  =\mu(ab\ot c)-\mu(a\ot bc)\\
  &=\mu(m(a\ot b)\ot c)-\mu(a\ot m(b \ot c))
\quad\text{ for }\ a,b,c\in \A\, .
\end{aligned}
\end{equation*}

Next we lift the image of $\iota$ on $X_{2,1}$
from the reduced bar to the bar resolution of $A$.
The set of relations $\R'$ in $\bA\ot\bA$ 
is the projection of the set of relations $\R'$ in $A\ot A$,
see \cref{liberty3,liberty4},
using an abuse of notation.
Since the projection map $\pr_{\BB_A}: \B_A\rightarrow \BB_{A}$
is injective on $\R'\subset A\ot A\subset \B_A$,
we may lift each element of $\R'$ in $\bA\ot \bA$ uniquely to one of $\R'$ in $A\ot A$.
We extend this to a lift of the subspace in \cref{subset}
by replacing 
every 
 $h\ot v-(1\ot \pr_{\bar H})\tau(h\ot v)$ by 
 its unique preimage $h\ot v-\tau(h\ot v)$.
 We obtain a lift of the
 image of $\iota_3$ on $X_{2,1}$
 to $(\B_A)_3$:
 $$\text{Lift}\big(\ima\iota_3\big|_{X_{2,1}}\big)
 \subset  
 \big( (V\ot \R')\oplus (H\ot \R) \big) 
  \cap 
  \big( (\R' \ot V)\oplus (\R\ot H ) \big)  
  \subset A\ot A  \ot A
  $$
  which
   lies in the kernel of the map $m\ot 1-1\ot m$.
  \cref{CocylesAndMultMaps} then implies that
$$
\begin{aligned}
\text{Lift} \big(\ima\iota_3\big|_{X_{2,1}}\big)
&\subset\
\ker (m\ot 1 -1\ot m)\\
&\subset\ \ker \big(\mu_1(m\ot 1 -1\ot m)\big)
=
\ker \big(m(\mu_1\ot 1 - 1\ot \mu_1)\big)
\end{aligned}
$$
for $\mu_1$ extended to a map $A\ot A\rightarrow A$
by setting $\mu_1(1_A\ot *)=0=\mu_1(*\ot 1_A)$ as usual.
Then 
$$
(\mu_1\ot 1 - 1\ot \mu_1) \ \text{Lift} \big(\ima\iota_3\big|_{_{X_{2,1}}}\big)
\subset\ 
\ker m
\subset\ A (\R\oplus \R')A
\, .
$$
 
But this set must also lie in 
the degree $1$ component $(A\ot A)_1$
of $A\ot A$ since
$\deg\mu_1=-1$ and $\text{Lift}(\ima\iota_3|_{X_{2,1}})$ lies in the degree
$2$ component $(A\ot A\ot A)_2$.
Hence
$$
(\mu_1\ot 1 - 1\ot \mu_1)\ \text{Lift}\big(\ima\iota_3\big|_{_{X_{2,1}}}\big)
\subset A(\R\oplus \R')A \cap (A\ot A)_1
= \R'
\subset A\ot A
\, .
$$

To transfer this back to brackets on $\BB_A$, we
apply the projection map $\pr_{\BB_A}$
which is $A\ot A\rightarrow \bA\ot \bA$ in our simplified notation.
By \cref{ExplicitValues},
$$\pr_{\BB_A}
(\mu_1\ot 1 - 1\ot \mu_1)
\big(\ima\iota_3\big|_{_{X_{2,1}}}\big)
=
\pr_{\BB_A}
(\mu_1\ot 1 - 1\ot \mu_1)
\text{Lift}\big(\ima\iota_3\big|_{_{X_{2,1}}}\big)
\subset \R' \subset \ima\iota_2
\, ,
$$
so on $\ima\iota_3\big|_{_{X_{2,1}}}$,
$$
 \pr_{\BB_A}\, 
(\mu_1\ot 1 - 1\ot \mu_1)\
\subset \ima\iota_2
\, .
$$
Our key observation \cref{KeyObservation} that
$\pi^*\iota^*(\mu_1)=\mu_1$
on $\Ima\iota$ then verifies
the claim:
\label{SupressProjectionInBrackets}
\begin{eqnarray*}
  2\iota^* \big(\pi^*\iota^*(\mu_1) \, \pr_{\BB_A}\,
  (\mu_1\ot 1 - 1\ot \mu_1)\big)
  & = & 
        2\iota^*\big(\mu_1 \, \pr_{\BB_A}\, (\mu_1\ot 1 - 1\ot \mu_1)\big)
        \quad\text{ on $X_{2,1}$.}
\end{eqnarray*}

We now may apply the standard deformation theoretic identity
$[ \mu_1 , \mu_1 ] = 2 d_3^*(\mu_2)$
(see~\cite{GerstenhaberSchack}) and suppress mention of 
$\pr_{\BB_A}$ as is customary to rewrite the Gerstenhaber bracket,
continuing our earlier calculation:
\begin{eqnarray*}
  [\alphax + \lambdax, \alphax+\lambdax]_{_{X_{2,1}}}
  & = &
          2\iota^* \big(\pi^*\iota^*(\mu_1)  (\mu_1\ot 1 - 1\ot \mu_1)\big)\\
  &  = & 2 \iota^*\big( \mu_1(\mu_1\ot 1 - 1\ot \mu_1)\big)
        \  =\ \iota^* [\mu_1, \mu_1]\\
& = & 2 \iota^* d^* \mu_2
 \ =\ 2 d^* \iota^* \mu_2
 \ =\ 2 d^* \beta
 \, 
\end{eqnarray*}
and Condition (b) on $X_{2,1}$ holds.

We now consider the bracket on $X_{3,0}$.
We argue as in the case of $X_{2,1}$
using again \cref{KeyObservation} and
 \cref{ExplicitValues}.
One may verify as before that
$$\begin{aligned}
  [\alphax + \lambdax , \alphax + \lambdax ] _{_{X_{3,0}}} 
  &= 2 \iota^* ( \pi^*\iota^* (\mu_1)  (\mu_1\ot 1 - 1\ot\mu_1))\\
  &= 2 \iota^* (\mu_1  (\mu_1\ot 1 - 1\ot\mu_1))
  \, .
\end{aligned}
$$
The image of $\mu_1\ot 1 - 1\ot \mu_1$ on $\iota (X_{3,0})$
lies in $\Ker m$,
hence in $(\R \oplus \R')\subset (\B_A)_2$,
which is contained in $\Ima \iota_2$ by \cref{ExplicitValues}
upon projecting to $(\BB_A)_2$.
We conclude that
$$[\alphax + \lambdax , \alphax+\lambdax ] _{_{X_{3,0}}} =
2d^* \beta\big|_{_{X_{3,0}}}
$$
and Condition~(b) holds on $X_{3,0}$.
Therefore Condition~(b) holds on all of $X_3$. 
\vspace{2ex}

{\bf Condition (c).}
Now observe that
the bracket $[\lambdax+\alphax, \betax]$ is cohomologous to 
$\iota^*([\mu_1,\mu_2])$,
and by general deformation theory results, $[\mu_1,\mu_2]$ 
is a coboundary.
Thus $[\alphax+\lambdax,\betax]$ is a coboundary,
i.e., there is some 2-cochain $\xi$ with $[\alphax + \lambdax, \betax]=d^*(\xi)$.
But by the definitions of $\lambda, \alpha,\beta$, the coboundary 
$[\alphax +\lambdax,\betax]$ is of graded degree~$-3$, and 
the only 2-cochain on $X_2$ of graded degree~$-3$ is $0$.
Hence Condition~(c) holds. 
\end{proof}

\section{PBW deformations from three cohomological conditions}
\label{sec:PBWFromThree}

Again, let $\Ho$ be a Hopf algebra acting
on a graded Koszul algebra $S$ (with action preserving the grading).
We seek to generalize~\cite[Theorem 5.3]{SW-Koszul}
and we show here that the converse of \cref{ForwardsDirection} 
holds.
Again, we consider equality of
cochains, not just equality of cohomology classes,
on the twisted product resolution $X=K_S\ott \BB_\Ho$
formed from the Koszul resolution 
  $K_S$ of $S$ and the reduced bar resolution $\BB_\Ho$
  of $H$  (see \cref{eqn:xij}).
We 
consider the filtered algebra $\chabl$ defined in 
Eqn.~\eqref{eqn:Hlab2} by some set of parameter functions 
$\alpha,\beta,\lambda$
which define cochains $\alphax, \betax, \lambdax$ on $X$, 
see \cref{CochainAlpha,CochainLambda}.

\vspace{2ex}

  \begin{thm}\label{BackwardsDirection}
For any choice of parameter functions $\alpha,\beta,\lambda$, 
the filtered algebra $\chabl$ defined in 
Eqn.~\eqref{eqn:Hlab2} is a strong
PBW deformation of $S\# \Ho$
if
   \begin{itemize}
    \item[(a)]\ \ $d^*(\alphax+\lambdax) = 0$, 
    \item[(b)]\ \ $[\alphax+\lambdax, \ \alphax + \lambdax]=2d^*\betax$, 
    \item[(c)]\ \ $[\lambdax + \alphax,\  \betax]= 0$ 
  \end{itemize}
  as cochains on the twisted product resolution 
$X=K_S\ott \BB_\Ho$, for bracket $[\ ,\ ]=[\ ,  \ ]_{\X}$.
\end{thm}

\vspace{1ex}

\begin{proof}
Assume that Conditions (a), (b), and (c) hold.
We show that $\chlab$ is a strong PBW deformation of $\A=S\# \Ho$
using the
graded chain maps given in
\cref{HopfKoszulConversion},
$\pi: \BB_A\rightarrow X$ and 
$\iota: X\rightarrow \BB_A$
with $\pi \, \iota =1_X$.

{\bf Three conditions give a graded deformation.}
Let $\delta$ be the differential on $\BB_A$ and $d$ the differential
on $X$; recall that these differentials preserve the
grading on the algebra $A$ (with $H$ in degree 0). 
Define cochains on $\BB_A$
$$\mu_1=\pi^*(\alphax + \lambdax)
\quad\text{ and }\quad \mu_2=\pi^*(\betax)\, ,
\qquad\text{ and set }\quad
\omega
=\delta_3^*(\mu_2)-\tfrac{1}{2}[\mu_1,\mu_1]
  \,  . 
  $$

 By Condition~(a),  $\alphax+\lambdax$
is a 2-cocycle, and so $\mu_1$ is a  2-cocycle
on $\BB_A$.
Thus $\mu_1$  
defines a first level deformation $A_1$
of $A$ (see \cite{GerstenhaberSchack}).

We now argue that Condition~(b)  implies that 
this first level deformation $A_1$ can be extended to the second 
level.
By~\cref{HopfKoszulConversion} and the definition
of bracket on $X\sd$, 
$$
\begin{aligned}
  \iota^*_3(\omega)
  &= \iota^*_3 \, \delta_3^*\, \pi_2^*(\betax)
-\tfrac{1}{2}\, \iota_3^*\, [\pi^*_2(\alphax+\lambdax),\, \pi^*_2(\alphax+\lambdax)]\\
&= d_3^*\, \iota^*_2 \, \pi_2^*(\betax)
-\tfrac{1}{2}[\alphax+\lambdax,\, \alphax+\lambdax]_{\X} \\
&= d_3^*(\betax )
-\tfrac{1}{2}[\alphax+\lambdax,\, \alphax+\lambdax]_{\X}  \\
&= 0\, .
\end{aligned}
$$
Hence
$\omega$ is a coboundary:  
$\omega=\delta^*(\mu)$ for some 2-cochain $\mu$ on $\BB_A$,
and $\mu$ has graded degree $-2$ by its definition.
Next note that as $\iota$ is a chain map,
$$
d^*\, \iota^*(\mu) = \iota^* \delta^*( \mu) =\iota^*(\omega) =0\, ,
$$
so $\iota^*(\mu)$ is a 2-cocycle. 
Hence there is a 2-cocycle $\mu'$ on $\BB_A$
with $\iota^*\mu'=\iota^*\mu$.
We replace $\mu_2$ by $\tilde{\mu}_2=\mu_2-\mu+\mu'$ so that
$\iota^*(\tilde{\mu}_2)=\beta$ and 
$$2\, \delta^*(\tilde{\mu}_2)= 2\, \delta^*(\mu_2+\mu')-2\, \omega
= 2\, \delta^*(\mu_2) - 2\, \omega = [\mu_1,\mu_1] $$
by the definition of $\omega$,  
since $\mu'$ is a cocycle. 
We see that the obstruction to lifting $A_1$
to a second level deformation,
using the multiplication map $\tilde{\mu}_2$, is 0, so  
$\mu_1$ and $\tilde{\mu}_2$ together define
a second level deformation $A_2$ of $A$.

Finally, Condition~(c) implies  $A_2$ lifts
to a third level deformation of $A$: 
We add the coboundary $\mu'-\mu$ to $\mu_2$.
This adds a coboundary to $[\mu_2, \mu_1]$.
It follows that 
$[\tilde{\mu}_2, \mu_1] = \delta^*_3(\mu_3)$
for a cochain $\mu_3$ on $\BB_A$, and $\mu_3$ has 
graded degree $-3$.
Therefore the obstruction to lifting $A_2$ to a third level deformation
vanishes and the multiplication maps
$\mu_1, \tilde{\mu}_2, \mu_3$ indeed
define a third level deformation $A_3$ of $A$.

We use an argument of Braverman and Gaitsgory~\cite{Braverman-Gaitsgory}. 
By~\cite[Proposition 1.5]{Braverman-Gaitsgory}, 
the obstruction to lifting $A_3$ to a fourth level deformation
of $A$ lies in $\HH^{3,-4}(A)$.
Since the chain map $\iota$ has graded degree~0,
applying  $\iota^*$ to a cochain on $\BB_A$ representing this 
obstruction gives
a cochain on $X_3$ of graded degree $-4$. 
By its definition, 
$X_3$ is generated by elements of graded degree 3
or less, and so this cochain must vanish.
(Recall $S$ is Koszul, and see~\cite[Definition~3.4]{Braverman-Gaitsgory}.) 
Thus the obstruction is a coboundary, and it follows that 
the deformation $A_3$ lifts
to a fourth level deformation $A_4$ of $A$.
Continuing in this fashion, we find that
the obstruction to lifting an $i$-th level deformation 
lies in $\HH^{3, -(i+1)}(A)=0$ for $i\geq 3$
and the obstruction is a coboundary.
Therefore 
$A_i$ lifts to $A_{i+1}$, 
an $(i+1)$-st level deformation of $A$, for all $i\geq 1$.

Thus Conditions (a), (b), and (c) imply that the
$2$-cocycles $\mu_1$ and $\mu_2$ lift
(integrate) to a graded deformation 
$\A_t$ of $\A$.
By \cref{PBW-iff-deformation},
every fiber of this graded deformation is isomorphic to a strong 
PBW deformation of $A=S\# H$.

{\bf Presenting the graded deformation
  as a PBW quotient algebra.}
We identify the graded deformation 
$\A_t$ of $\A$
with the vector space $\A[t]$ via multiplication given by
$$
a\star a'=a a'+\mu_1(a\ot a')t+\mu_2(a\ot a')t^2 + \mu_3(a\ot a')t^3+\ldots
\qquad\text{
  for all $a,a'\in A$}
$$
for some $k$-linear maps $\mu_i:A\ot A\rightarrow A$
with $\mu_1$, $\mu_2$ induced from 
the $A$-bimodule maps $\mu_1$, $\mu_2:A\ot \bar A\ot \bar A \ot A
\rightarrow A$ by setting $\mu_i(1_A\ot *)=0=\mu_i(*\ot 1_A)$.
By \cref{PBW-iff-deformation}, as $\A_t$ is a graded deformation,
the fiber
$\A_t|_{t=1}$ is isomorphic as a filtered algebra
to a strong PBW deformation
$$
\A_t\big|_{t=1} \cong \cH_{\lambda',\alpha',\beta'}
$$
of $\A$ for some $k$-linear parameter functions
given in the proof of \cref{DeformationImpliesPBW}
(see \cref{quad-relns}, \cref{smash-relns}, 
and 
  \cref{RewritingSmashProduct}):
\[
  \lambda': \R'\longrightarrow \Ho,\ \ \ 
  \alpha': \R \longrightarrow V\ot \Ho, \ \ \
    \beta': \R\longrightarrow \Ho
    \, .
\]

{\bf PBW deformation coincides
  with $\cH_{\lambda,\beta,\alpha}$.}
We show that the parameters $\lambda',\alpha',\beta'$
coincide with $\lambda,\alpha,\beta$.
When working in the reduced bar complex $\BB_A$, we identify
$$\text{ 
  $V\ot \bar H$, \ $\bar H \ot V$, \ $V\ot V$, \ $\R'$, and $\R$}
$$
    with 
    vector subspaces of 
    $(\BB_A)_2=A\ot \bar A \ot \bar A \ot A $,
    see \cref{liberty1,liberty2,liberty3,liberty4},
    and identify the $A$-bimodule maps
    (see \cref{CochainAlpha,CochainLambda}) 
$$\pi^* \lambdax,\ \pi^* \alphax, \ \pi^* \betax: (\BB_A)_2=A\ot \bar A
\ot \bar A \ot A \longrightarrow A$$
with $k$-linear maps $\bar A\ot \bar A
\rightarrow A$
to suppress superfluous tensor factors of $1$.

We find explicit values for $\lambda', \alpha', \beta'$
from the
proof of \cref{DeformationImpliesPBW}
with $W=H\ot V\ot H$, $\RR_1=\R'\subset W$,
and $\RR_2=\R\subset V\ot V$
using the projection map $\pi_{_A}:T_H(W)\rightarrow A$
and a section $\iota_{_A}: A\rightarrow T_H(W)$
right inverse to $\pi_{_A}$ as in the proof
chosen to map the product $vh$ in $A$ to $1_H\ot
v\ot h$ in $W\subset T_H(W)$
for $v$ in $V$ and $h$ in $H$.

We first notice that $\lambda\equiv \lambda'$, 
since for arbitrary
$$r=h\ot v\ot 1_H - 1_H\ot \tau(h\ot v) 
\quad\text{ in $\RR_1$},
$$
with $h$ in $\bar H$ and $v$ in $V$,
\cref{ConstructingLambda} 
implies that
$$\lambda'(r) = \mu_1(hv\ot 1_H)+\mu_1(h\ot v) 
-\sum \mu_1(\ ^{h_1}v\ot h_2)-\mu_1(1_H\ot \ ^{h_1}v)h_2$$
which by \cref{PiInLowDegree} is
$$
\begin{aligned}
\mu_1\big(h\ot v
-\tau(h\ot v)\big)
&=
\pi^*(\alphax+\lambdax)(1\ot \pr_{\bar H})(1\ot 1 - \tau)(h\ot v)
=
(\pi^* \lambdax)(h\ot v)\\
\end{aligned}
$$
for projection map $\pr_{\bar H}: H\rightarrow \bar H$
as $\mu_1(1_H\ot *)=0=\mu_1(*\ot 1_H)$
and $\alphax\equiv 0$ on $X_{1,1}$.
By \cref{PiInLowDegree} again, this is
$$
\begin{aligned}
  -\lambdax &\big(\sum (1_S\ot \, ^{h_1}v\ot 1_S)\ot (1_H\ot \pr_{\bar
    H} \, h_2\ot 
  1_H) \big)\, .
  \end{aligned}
  $$
We add and subtract 
$1\ot \, ^{h_1}v\ot 1\ot 1\ot (h_2-\pr_{\bar H} h_2)\ot 1$  to each summand
noting that the second component lies in $k1_H$,
pass from $\lambdax$ to $\lambda$ using \cref{CochainLambda},
 and rewrite the resulting expression in terms of 
 $\tau^{-1}$.  Since $\tau^{-1}$ acts as the swap map when
 any component of the input lies in $k1_H$ and we are tensoring over $k$,
  the terms with $h_2-\pr_{\bar H} h_2$ cancel again, leaving
 \begin{equation}\label{PiOnR'}
   \begin{aligned}
     \mu_1\big(h\ot v -\tau(h\ot v)\big)
     &=
         -\sum\lambda  \big(1_H\ot\, ^{h_1}v\ot h_2
      -\tau^{-1}(\, ^{h_1}v\ot h_2)\ot 1_H\big)
 \\  &\ =\
  -\lambda\big(1_H\ot \tau(h\ot v)-\tau^{-1} \tau(h \ot v)\ot 1_H\big)
\\ &\ =\
\lambda\big(h\ot v\ot 1_H-1_H\ot \tau(h\ot v)\big)
\end{aligned}
\end{equation}
which is just $\lambda(r)$.
Hence $\lambda\equiv \lambda'$ on $\RR_1$.

Now we check that $\alpha\equiv \alpha'$ and $\beta \equiv \beta'$ on
$\RR_2$.
For $r$ in $\RR_2\subset V\ot V$,
\cref{ConstructingParameters} and 
\cref{PiInLowDegree} imply that
$$\begin{aligned}
  \alpha'(r)
  &=\iota_{_A}\,  \mu_1\, (\pi_{_\A}\ot\pi_{_\A}) (r) 
=\iota_{_A} \, \pi^*(\alphax+\lambdax) (r)\\
&=\iota_{_A} (\alphax+\lambdax)\, \pi(r)\\
&=\iota_{_A} (\alphax+\lambdax)(1_S\ot r \ot 1_S\ot 1_H\ot 1_H)
\\ 
&=\iota_{_A}\, \alphax(1_S\ot r \ot 1_S\ot 1_H\ot 1_H)
\, 
\end{aligned}
$$
with identifications of \cref{liberty2}
as $\lambdax\equiv 0$ on $X_{2,0}$.
This is just $\alpha(r)$ under identifications $A\cong S\ot H$
and $V\ot H \cong k1_H\ot V \ot H$ and the definition of $\alphax$.
Likewise, from \cref{ConstructingParameters},
$$\begin{aligned}
  \beta'(r)
  &=\big(\mu_2
  -\lambda'\, \iota_{_A}\,  \mu_1 \big)(\pi_{_\A}\ot\pi_{_\A}) (r)
 \\ &=\big(\pi^*\, \betax
  - \lambda'\, \iota_{_A}\,
  \pi^*\, (\alphax+\lambdax)\big)(r)
 \\ &=\big(\betax- \lambda'\, \iota_{_A}\, (\alphax+\lambdax)\big)
(1_S\ot r \ot 1_S\ot 1_H\ot 1_H)\\
&=\beta(r)- \lambda'(1_H\ot \alpha(r))
\, .
\end{aligned}
$$
But $1_H\ot \alpha(r)$ lies in $k1_H\ot V\ot H$,
and by \cref{ConstructingLambda}
and \cref{PiInLowDegree},
\begin{equation}\label{BetaIsMu2}
\begin{aligned}
\lambda'(1_H\ot v\ot h)
&=
\mu_1(v \ot h) + \mu_1(1_H\ot v) h
=
\mu_1(v \ot h)\\
&=
\pi^*\, (\alphax+\lambdax)(v\ot h)
\\
&=(\alphax+\lambdax)\, \pi\, (v\ot h)
=0,
\end{aligned}
\end{equation}
and thus $\beta'\equiv \beta$ on $\RR_2$ as well.
Thus $\A_t$ is a graded deformation with fiber $A_t\big|_{t=1}$
isomorphic as a filtered algebra to a strong PBW deformation
$\cH_{\lambda',\alpha', \beta'}=\cH_{\lambda,\alpha,\beta}$.
\\
\end{proof}

\vspace{1ex}

\cref{ForwardsDirection,BackwardsDirection}
imply the next corollary giving conditions
for a quotient algebra $\chabl$
defined in 
Eqn.~\eqref{eqn:Hlab2} to be a strong PBW deformation.
Recall that if $H$ is finite dimensional, or, more generally, doubly  Noetherian, then
we may remove the adjective ``strong'' from the hypothesis by
\cref{PBW-iff-deformation}.

\vspace{2ex}

\begin{cor}\label{hom-condns} 
Let $\Ho$ be a Hopf algebra 
acting 
on a graded Koszul algebra $S$. 
A filtered algebra $\cH$  is a strong PBW deformation of $S\# \Ho$ if and only if 
$\cH=\chabl$
for some parameters $\alpha,\beta,\lambda$ satisfying 
  \begin{itemize}
    \item[(a)]\ \ $d^*(\alphax+\lambdax) = 0$, 
    \item[(b)]\ \ $[\alphax+\lambdax,\ \alphax + \lambdax]=2d^*\betax$, 
    \item[(c)]\ \ $[\lambdax + \alphax,\ \betax]= 0$ 
  \end{itemize}
  as cochains on the twisted product resolution 
$X=K_S\ott \BB_\Ho$, for bracket $[\ ,\ ]=[\ ,  \ ]_{\X}$.
  \end{cor}

\vspace{1ex}
  
\subsection*{Lifting first multiplication maps}
  Recall that for any $k$-algebra $B$,
  a given $k$-linear map
  $\mu_1: B\otimes B \rightarrow B$
  {\em lifts} (or {\em integrates})
  to a graded deformation if there exists a graded deformation $B_t\cong B[t]$
  with multiplication as in~\cref{eqn:star-product}. 
  Conditions for $\mu_1$ to lift are given in terms of the differential
  $d$ on the bar resolution and the Gerstenhaber bracket $[\ , \ ]$
  on Hochschild cochains.
  In the proof of \cref{BackwardsDirection},
  we saw that the argument of Braverman and Gaitsgory
  \cite{Braverman-Gaitsgory}
  extends to our setting to show all higher obstructions to integrating a potential
  first multiplication map $\mu_1$ vanish
  for $S\# \Ho$:

  \vspace{1ex}
  
\begin{prop}
Let $\Ho$ be a Hopf algebra acting 
on a graded Koszul algebra $S$ as a Hopf module algebra. 
A $k$-linear map $$\mu_1: (S\# \Ho)\otimes (S\# \Ho) \longrightarrow
(S\# \Ho)$$
lifts to a graded deformation of $S\# \Ho$ if and only if there exists some
$k$-linear map \\ $\mu_2: (S\# \Ho)\otimes (S\# \Ho) \rightarrow (S\# \Ho)$ with
  \begin{itemize}
    \item[(a)]\ \ $d^*(\mu_1) = 0$, 
    \item[(b)]\ \ $[\mu_1, \mu_1]=2d^*\mu_2$, 
    \item[(c)]\ \ $[\mu_1, \mu_2] = 0$.
    \end{itemize}
  \end{prop}
  
  \section{Explicit PBW conditions for
  Hopf actions on Koszul algebras}
\label{sec:DHAs}
We again fix a Hopf algebra $H$ acting on a graded Koszul
algebra $S=T(V)/(\R)$.
Using homological
techniques of earlier sections,
we generalize~\cite[Theorem 2.5]{SW-Koszul}
and describe all strong PBW deformations of $S\# H$
(over $H$)
in terms of explicit conditions
on parameter functions.

Recall that we define
$\chabl$ as the quotient algebra
(see \cref{eqn:Hlab,smash-relns,eqn:Hlab2}) 
\begin{equation}
   \chabl = T_{\! _H}(\Ho\ot V \ot \Ho)/ (\P\cup \P')
\end{equation}
for vector spaces
$$
\begin{aligned}
  \P & =  \{ r -\alpha(r)-\beta(r) : r\in \R\} & & \text{(deformed
    Koszul relations)}\\
 \P' & =  \{ r - \lambda(r) : r\in \R' \} & & \text{(deformed smash relations)}
\end{aligned}
$$
determined by some arbitrary $k$-linear parameter functions 
\[
  \alpha: \R\rightarrow V\ot \Ho , \ \ \
   \beta: \R \rightarrow \Ho , \ \ \
   \lambda: \R'\rightarrow \Ho .
 \]

 We identify $\lambda$ with a function on $\bar H\ot V$ to give the
 PBW conditions compactly.
 See also \cref{thm:main-Hopf-right}
for a right (versus left here) version 
(identifying $\lambda$ with a function on $V\ot \bar H$ instead)
 for easier comparison
with results in the literature.
To simplify notation,
we apply the multiplication map $T_k(A)\rightarrow A$,
$\ a_1\ot\cdots\ot a_n\mapsto a_1\cdots  a_n$,
to each summand in each condition below after identifying $V$ and $H$ with
subspaces of $A$.

\vspace{1ex}

\begin{thm}    
\label{thm:main-Hopf-left}
A filtered algebra $\cH$ is a strong PBW deformation of $S\# \Ho$ if
and only if $\cH= \chabl$ for some
linear parameter functions 
$\alpha,\beta,\lambda$ satisfying
\begin{itemize}
\item[$(1)$] $1\ot\lambda - \lambda (m\ot 1) + (\lambda\ot 1)(1\ot\tau) = 0$,
  \item[$(2)$] 
  $\lambda (\lambda\ot 1) + (\lambda\ot 1)(1\ot\alpha) 
    =     (1\ot\beta) - (\beta\ot 1)(1\ot \tau)(\tau\ot 1)$,
  \item[$(3)$] $(1\ot \alpha) - (\alpha\ot 1)(1\ot\tau)(\tau\ot 1)
  = (\lambda\ot 1)   + (1\ot\lambda)(\tau\ot 1)$,
       \item[$(4)$]
       $(\alpha\ot 1)\big((1\ot\tau)(\alpha\ot 1) - (1\ot \alpha)\big) 
       - (1\ot \lambda)(\alpha\ot 1) 
       = 1\ot\beta - \beta\ot 1$,
  \item[$(5)$] $(\beta\ot 1)\big((1\ot\tau)(\alpha\ot 1) - (1\ot\alpha)\big)
    = -\lambda (\beta\ot 1)$,
    \item[(6)] $\alpha\ot 1 - 1\ot\alpha =0$,
    \end{itemize}
    upon application of the multiplication map 
    $T_k(A)\rightarrow A$ to each summand,
    for $\lambda:\R'\rightarrow H$
identified with the map $H\ot V\rightarrow H$, 
$h\ot v \mapsto h\ot v \ot 1_H-1_H\ot \tau(h\ot v)$.
    
    Here, the maps in (1) are defined on $H\ot H\ot V$,
the maps in (2) and (3) on $H\ot \R$,
the map in (6) on $(V\ot \R)\cap (\R\ot V)
\subset V\ot V\ot V$, and (6) implies that the maps in (4) and (5) are
defined
on $(V\ot \R)\cap (\R\ot V)$.
  \end{thm}

  \vspace{2ex}
  
  Before proving the theorem, we note that it implies
  a characterization for Hochschild cocycles which lift (i.e.,
  integrate) to the first multiplication map of a graded deformation
  by the proofs of \cref{ForwardsDirection}
  (see
  \cref{FirstMultMapToLambda,MultiplicationMapsAreAlphaBeta})
  and \cref{BackwardsDirection}
  (see
 \cref{PiOnR',ConstructingParameters,BetaIsMu2}).
  \begin{cor}\label{WhichCocyclesLift}
      A Hochschild $2$-cocycle $\mu_1$ on $S\# H$
      lifts to a graded deformation of $S\#H$ if and only if there exist
      parameter functions $\alpha$, $\beta$, $\lambda$ satisfying
      the six conditions of \cref{thm:main-Hopf-left} with
      $$\begin{aligned}
      \mu_1(r) = \alpha(r)
      \quad\text{ and }\quad
      \mu_1\big(h\ot v - \tau(h\ot v)\big)
      = \lambda\big(h\ot v\ot 1_H -1_H\ot \tau(h\ot v)\big)
\end{aligned}
$$
for all $r$ in $\R$, $h$ in $H$, and $v$ in $V$.
These values then completely determine $\mu_1$
up to a coboundary,
and the second multiplication map $\mu_2$
of the deformation
may be chosen as the Hochschild $2$-cochain
with
$$
\mu_2(r)= \beta(r)
\quad \text{ for all }
r \text{ in } \R\, .
$$
 \end{cor}

 \vspace{2ex}
 
  \begin{proof}[Proof of \cref{thm:main-Hopf-left}]
    First observe that we may replace $H$ by $\bar H$ in the
    domains of the maps in the six
    conditions for $H/k1_H\cong \bar H\subset H$:
    A short computation confirms that Conditions (1), (2), and (3)
    trivially hold
    if we replace any $H$ in
    a tensor component of the domain
by $k1_H$ since $\lambda\equiv 0$ on $k1_H\ot V$
and
$\tau$ is just the swap map on any input $1_H\ot v$;
hence we need only check 
Condition (1)
   on domain $\bar H\ot \bar H \ot V$ and Conditions (2) and (3)
   on domain $\bar H\ot \R$.

  We show that  Conditions (a), (b), and (c) of \cref{hom-condns}
    are equivalent
    to the six conditions of \cref{thm:main-Hopf-left}
    (with $\bar H$ instead of $H$ in the domains).
 
Again, let $m=m_{A}:A\ot A\rightarrow A$ denote the multiplication map on
  the smash product $A=S\# \Ho$, which involves the twisting map
  $\tau: \Ho\ot S \rightarrow S\ot \Ho$,
  $h\ot  s \mapsto \sum \, ^{h_1}s \ot h_2$.

 {\bf Identifications.}
 Since all functions on resolutions of $A$
  involved 
  are $A$-bimodule 
  maps, we need only consider $A$-bimodule bases of the various free resolutions.
We thus suppress superfluous tensor factors of $1$
and make certain vector space identifications to avoid excessively
technical notation.
When working in the reduced
bar resolution $\BB_A$, we identify 
\begin{equation*} 
  \text{ 
    $V\ot \bar H$, \ $\bar H \ot V$, \ $V\ot V$, \ $\R'$, and $\R$}
\end{equation*}
    with 
    $k$-vector subspaces of
    $$
    (\BB_A)_2 = A\ot \bar A \ot \bar A \ot A
    \ \subset\  A\ot A \ot A \ot A = (\B_A)_2
    $$
    by ignoring extra tensor factors of $1_A$ on the outside
    and viewing $V,\bar H$ as subsets of $\bar A\subset A$,
    see \cref{liberty1,liberty2,liberty3,liberty4}.
    We make similar identifications when working in $(\BB_A)_3$,
    viewing $\R\ot V$ and $\bar H\ot V \ot \bar H$ as subspaces of
    $(\BB_A)_3$,
    for example. 

  Likewise,  by \cref{LibertyOnX},
  each $X_{i,j}$ has an $A$-bimodule basis consisting of
  elements in $1_S\ot V^{\ot i}\ot 1_S\ot 1_H\ot \bar H^{\ot j}\ot 1_H$,
  so we identify the $A$-bimodule $X_{i,j}$
  with the $k$-vector space $V^{\ot i}\ot \bar H^{\ot j}$
    to avoid extra factors of $1_S$ and $1_H$.
  Thus we regard
  $H\ot H$ as a subspace of $X_{0,0}$,
  $V\ot \bar H$ as a subspace of $X_{1,1}$,
  $\R$ as a subspace of $X_{2,0}$,
$\R\ot \bar H$ as a subspace of $X_{2,1}$, and 
$ (\R\ot V)\cap(V\ot \R)$ as a subspace of $X_{3,0}$,
and each of these subspaces provides an $A$-bimodule basis.

We also identify as usual
any $A$-bimodule homomorphism 
$A\ot U\ot A\rightarrow U'$
with a $k$-linear map
$U\rightarrow U'$ of the same name
for any $k$-vector space $U$ and $A$-bimodule $U'$.
For any $n$-cochain $\omega$ on $X$ or on $\BB_A$,
we view in this way 
$\omega\ot 1$ and $1\ot \omega$ as $(n+1)$-cochains
on $\BB_A$.
We extend any cochain on $\BB_A$ to one on $\B_A$
by composing first with the projection $\pr_{\BB_A}:\B_A\rightarrow \BB_A$.

{\bf One-sided versions of parameter function $\lambda$.}
  We define abbreviated (one-sided) versions of the parameter function
  $\lambda:\R'\rightarrow H$
  to simplify expressions: Consider
\begin{equation}\label{lambda'}
\begin{aligned}
& \lambda_{_L}: \ H\ot V\longrightarrow H, \quad 
& h\ot v \mapsto\lambda\big(h\ot v\ot 1_H - 1_H\ot \tau(h\ot v)\big)
\\
& \lambda_{_R}: V\ot H\longrightarrow H, \quad 
& v\ot h \mapsto\lambda\big(1_H\ot v\ot h -  \itau(v\ot h)\ot 1_H \big)
\end{aligned}
\end{equation}
so that  
\begin{equation}\label{LambdaSwitchSides}
  \lambda_{_L}=-\lambda_{_R}\circ \tau
  \quad\text{ and }\quad
  \lambda_{_R}=-\lambda_{_L}\circ \itau
  \, .
  \end{equation}
For parameter functions $\lambda:\R'\rightarrow H$,
$\alpha:\R\rightarrow V\ot H$, 
and $\beta:\R\rightarrow H$, we may write under these identifications
the corresponding cochains $\lambdax, \alphax, \betax$ on $X$ (see
\cref{CochainAlpha,CochainLambda})
with a slight abuse of notation as
\begin{equation}\label{CochainsToParameters}
\begin{aligned}
\lambdax &= \lambda_{_R}  = -\lambda_{_L}\itau
&&\quad\text{ as maps } V\ot \bar H\rightarrow H,
\\
\alphax &=m\, \alpha 
&&\quad\text{ as maps } \R\rightarrow V\ot H\rightarrow A,
\quad\text{ and }
\\
 \betax &=\beta 
 &&\quad\text{ as maps } \R\rightarrow H
 \, .
\end{aligned}
\end{equation}

We use the chain map $\pi: \BB_A\rightarrow X$
of \cref{HopfKoszulConversion}
and express the 2-cocycle
$\pi^* ( \lambdax ) : A\ot \bA\ot\bA\ot A \rightarrow A$
on $(\BB_A)_2$ in terms of $\lambda_{_L}$
with a slight abuse of notation:
  By \cref{PiOnR'},
  \begin{equation}\label{PiLambdax}
  \begin{aligned}
    \pi^* (\lambdax)\
    & = & \lambda_{_L}
  &\quad\text{ as maps } \bar H\ot V \longrightarrow H\, ,
  \quad\text{ and }\\
  \pi^* (\lambdax) \itau\  & = & - \lambda_{_R}
 &\quad\text{ as maps } V\ot \bar H \longrightarrow H
 \,  .
\end{aligned}
\end{equation}

{\bf The differential on $X$.}
The differential $d$ on
$X$ is induced
from the differential on
$\BB_S\ott \BB_H$ constructed from the reduced bar resolutions for $S$
and $H$ (see \cite{TTP-AWEZ}).
On $(\BB_S)_1\ot (\BB_H)_2\subset \BB_S\ott \BB_H$,
the differential is given by,
for $h$, $h'$ in $\bar H$ and $v$ in $V$,
$$
\begin{aligned}
 1_S &\ot v\ot 1_S \ot 1_H \ot h\ot h'\ot 1_H
  \\& \longmapsto\ \ \
\big( (v\ot 1_S-1_S\ot v) \ot (1_H\ot h\ot h'\ot 1_H)\big) 
\\&\ \ \ \ \ \
-\big( (1_S\ot v\ot 1_S)\ot (h\ot h'\ot 1_H 
- 1_H\ot h\cdot_{\bar H} h'\ot 1_H+1_H\ot h\ot h')\big)
\, .
\end{aligned}
$$
Thus we may write $d$ succinctly on $X_{1,2}$ under the above
identifications
in terms of the $A$-bimodule structure of $X$ recorded by 
left and right module maps
(see \cref{TwistedResolutionAction})
$$\rho_{_L}:A\ot X\rightarrow X
\quad\text{ and }\quad
\rho_{_R}: X\ot A\rightarrow X
\, .
$$
As a map $V\ot \bar H\ot \bar H\subset X_{1,2}\rightarrow
X_{1,1}\oplus X_{0,2}$,
\begin{equation}\label{d}
\begin{aligned}
  d =
  \rho_{_L}-
  \rho_{_R}(\pr_{\bar H}\ot \pr_{\bar H}\ot 1)
  (1\ot\itau)(\itau\ot 1) 
  - \rho_{_L}(\itau\ot 1)+ (1\ot \pr_{\bar H})
  ( 1\ot m)-\rho_{_R} 
\, 
\end{aligned}
\end{equation}
for $\pr_{\bar H}: H\rightarrow \bar H$ the projection map
(where $\rho_{_L}$ and $\rho_{_R}$ in the first two terms
indicate the $A$-module structure on $X_{0,2}$
but in the third and last terms
indicate the $A$-module structure on $X_{1,1}$).
But any $2$-cochain $\gammax$ on $X$
is an $A$-bimodule map
($\gammax\,\rho_{_L}=\rho_{_L}(1\ot \gammax) 
=m(1\ot \gammax)$, for example) and 
thus, on $V\ot \bar H\ot \bar H\subset X_{1,2}$,
the cocycle
$d^*(\gammax)$ is
\begin{equation*}
\begin{aligned}
  \gammax \big(\rho_{_L}-
  &
  \rho_{_R} (\pr_{\bar H}\ot \pr_{\bar H}\ot 1)
  (1\ot\itau)(\itau\ot 1) 
  - \rho_{_L}(\itau\ot 1)+ (1\ot \pr_{\bar H})
  (1\ot m)-\rho_{_R} \big) 
\\ &=\ \ \
m(1\ot \gammax)-
m(\gammax\ot 1) (\pr_{\bar H}\ot \pr_{\bar H}\ot 1)
(1\ot\itau)(\itau\ot 1)
\\ & \ \ \
- m(1\ot \gammax)(\itau\ot 1)+\gammax (1\ot \pr_{\bar H})
(1\ot m)-(\gammax\ot 1) 
\, ,
\end{aligned}
\end{equation*}
i.e.,
\begin{equation}\label{donX12}
\begin{aligned}
  d^*(\gammax) 
 =\ \ \ 
&(1\ot \gammax)-
(\gammax\ot 1) (\pr_{\bar H}\ot \pr_{\bar H}\ot 1)
(1\ot\itau)(\itau\ot 1) 
\\  - &(1\ot \gammax)(\itau\ot 1)+\gammax  (1\ot \pr_{\bar H}) (1\ot m)
-(\gammax\ot 1) 
\end{aligned}
\end{equation}
 upon projection to $A$ using the multiplication map
 $m:A\ot A\rightarrow A$.
 
Now consider the differential on $X_{2,1}$
induced from that on $(\BB_S)_2\ot_{\tau} (\BB_H)_1$:
$$
\begin{aligned}
1_S\ot v\ot v'\ot 1_S \ot 1_H\ot h\ot 1_H
\ \mapsto\ \ \ 
 &(v\ot v'\ot 1_S)\ot (1_H\ot h\ot 1_H)\\
-&(1_S\ot (v\cdot_S v')\ot 1_S)\ot (1_H\ot h\ot 1_H)\\
+&(1_S\ot v\ot v')\ot (1_H\ot h\ot 1_H)\\
 +&(1_S\ot v\ot v'\ot 1_S)\ot (h\ot 1_H)\\
 - &(1_S\ot v\ot v'\ot 1_S)\ot (1_H\ot h) 
 \, . 
\end{aligned}
$$
We write $d$ on $X_{2,1}$ again using the $A$-bimodule structure, as
above,
but on
$\R\ot \bar H\subset X_{2,1}$:
\begin{equation}
\begin{aligned}
d 
 &=
 \rho_{_L} - (m\ot 1) + \rho_{_R} (1\ot \pr_{\bar H}\ot 1)
 (1\ot \itau) + \rho_{_L}(\itau\ot 1 
)(1\ot \itau) - \rho_{_R} 
\end{aligned}
\end{equation}
where $\rho_{_L}$ 
and $\rho_{_R}$ in the first and third terms
give the $A$-bimodule structure on $X_{1,1}$
but in the last two terms
give the $A$-bimodule structure on $X_{2,0}$.
Note the second summand is $0$ as $m$ vanishes on $\R$.
Then 
for any cochain $\gammax$ on $X_2$,
on $\R\ot \bar H\subset X_{2,1}$,
the cocycle
$d^*\gammax$ is
\begin{equation*}
\begin{aligned}
 & \gammax \big(\rho_{_L} +
  \rho_{_R}(1\ot \pr_{\bar H}\ot 1)
 (1\ot \itau) + \rho_{_L}(\itau\ot 1 )(1\ot \itau) - \rho_{_R} \big)
\\ &=
m(1\ot \gammax) + m(\gammax\ot 1) (1\ot \pr_{\bar
  H}\ot 1)
(1\ot \itau)
+
m(1\ot \gammax)(\itau\ot 1 )(1\ot \itau) - m(\gammax \ot 1)\, ,
\end{aligned}
\end{equation*}
i.e.,
\begin{equation}
\begin{aligned}\label{donX21}
d^*\gammax
&= (1\ot \gammax) + (\gammax \ot 1) (1\ot \pr_{\bar H}\ot 1)
(1\ot \itau) +
(1\ot \gammax)(\itau\ot 1 )(1\ot \itau) - (\gammax \ot 1)
\end{aligned}
\end{equation}
upon projection to $A$ using the multiplication map
 $m:A\ot A\rightarrow A$.

{\bf Condition (a).}
The cochain $d^*(\alphax+\lambdax)$ 
is the zero function 
if and only if it is 0 on each of $X_{3,0}$, $X_{2,1}$, $X_{1,2}$, and 
$X_{0,3}$.
Since $d(X_{0,3})$
trivially intersects $X_{1,1}\oplus X_{2,0}$ on which
$\alphax +\lambdax$ is supported, $d^*(\alphax +\lambdax)\equiv 0$ on $X_{0,3}$.
Similarly, since $\alphax\equiv 0$ on $X_{0,2}\oplus X_{1,1}$,
$d^*(\alphax)\equiv 0$ on $X_{1,2}$.

{\bf Condition (a), differential on $X_{1,2}$.}
Since $\lambdax|_{X_{0,2}}\equiv 0$, \cref{donX12} gives that
on $X_{1,2}$
$$
\begin{aligned}
  d^* \lambdax 
  & =
  -(1\ot \lambdax)(\itau\ot 1)+\lambdax (1\ot \pr_{\bar H})
  (1\ot m)-(\lambdax\ot 1)
\, ,
\end{aligned}
$$
and thus
$d^*(\alphax+\lambdax)|_{X_{1,2}}\equiv 0$ exactly when
(see \cref{TauOnScalars})
\begin{equation}\label{SwitchSides}
\begin{aligned}
  0=(1\ot \lambda_{_R})(\itau\ot 1)
  -
  \lambda_{_R}(1\ot m)+(\lambda_{_R}\ot 1)
\quad\text{ on } V\ot \bar H\ot \bar H 
\,
\end{aligned}
\end{equation}
upon projection to $A$ using the multiplication map $m:A\ot 
A\rightarrow A$.

We compose with $(\tau\ot 1)(1\ot \tau)$ on the right 
and use the fact that  $\tau$ 
is a twisting map
and thus ``commutes'' with the multiplication in $H$,
i.e.~(see \cite{TTP-AWEZ} for example), 
\begin{equation*}\label{twistingcondition}
(1\ot m_H)(\tau\ot 1)(1\ot \tau) = \tau(m_H\ot 1) 
\, ,
\end{equation*}
to rewrite this condition in terms of maps on $\bar H\ot \bar H\ot V$
instead of $V\ot \bar H\ot \bar H$:
$$
\begin{aligned}
0&=(1\ot \lambda_{_R})(1\ot \tau) 
-\lambda_{_R}\tau(m\ot 1) 
+(\lambda_{_R}\ot 1) (\tau\ot 1)(1\ot \tau)
\\ &=(1\ot \lambda_{_R}\tau ) 
-\lambda_{_R}\tau(m\ot 1) 
+(\lambda_{_R}\tau\ot 1) (1\ot \tau)
\\ &=-(1\ot \lambda_{_L}) 
+\lambda_{_L}(m\ot 1) 
-(\lambda_{_L}\ot 1) (1\ot \tau)
\quad\text{ on }  \bar H\ot \bar H\ot V
\, .
\end{aligned}
$$
Thus $d^*(\alphax + \lambdax) |_{_{X_{1,2}}} = 0$ 
if and only if 
$$0\equiv (1\ot \lambda_{_L}) 
-\lambda_{_L}(m_H\ot 1) +(\lambda_{_L}\ot 1)(1 \ot \tau)  
\, 
\qquad\text{ on } \bar H\ot \bar H\ot V 
\, 
$$
upon projection to $A$ using the multiplication $m:A\ot 
A\rightarrow A$, i.e., Condition~(1) holds. 

{\bf Condition (a), differential on $X_{2,1}$.}
On $\R\ot V\subset X_{2,1}$, as $\alphax|_{X_{1,1}}\equiv 0$ and $\lambdax|_{X_{2,0}}\equiv
0$,
\cref{donX21} gives
  $$
\begin{aligned}
  d^*(\alphax +\lambdax) 
  & =
   (1\ot \lambdax) 
   +(\lambdax \ot 1) (1\ot \pr_{\bar H}\ot 1)
   (1\ot \itau)
   \\ & \ \ 
 +(1\ot \alphax) (\itau\ot 1)(1\ot \itau) 
 -(\alphax\ot 1)
 \, .
 \end{aligned}
 $$
 Thus (see \cref{CochainsToParameters}),
 $d^*(\alphax +\lambdax) \equiv 0$  
on $X_{2,1}$ exactly when (using \cref{TauOnScalars})
$$
 (1\ot \lambda_{_R}) 
 +(\lambda_{_R} \ot 1)(1\ot \itau) 
 =-(1\ot\alpha) (\itau\ot 1)(1\ot \itau) 
 +(\alpha\ot 1)
 \quad\text{ on $\R\ot \bar H$}
 $$
upon projection to $A$ by repeated applications
 of the multiplication map $m:A\ot A\rightarrow A$.
We again compose with $(1\ot \tau)(\tau\ot 1)$
on the right and use \cref{LambdaSwitchSides}
to express this in terms of a map 
 on $H\ot \R$ instead:
$d^*(\alphax +\lambdax) \equiv 0$  
on $X_{2,1}$ exactly when
 $$
 (1\ot \lambda_{_L})(\tau\ot 1) 
+ (\lambda_{_L} \ot 1)
 =(1\ot\alpha) 
 -(\alpha\ot 1) (1\ot \tau)(\tau\ot 1)
 \quad\text{ on } \bar H\ot \R $$
 upon projection to $A$ by repeated applications of the
 multiplication map
 $m:A\ot A\rightarrow A$, i.e.,
exactly when
Condition~(3) holds.

{\bf Condition (a), differential on $X_{3,0}$.}
Finally, 
on $X_{3,0}$, $d^*(\lambdax)$ vanishes since $\lambdax\equiv 0$
on $X_{2,0}$.
Consider $d^*(\alphax)$ on  $(\R\ot V)\cap (V\ot \R)\subset X_{3,0}$:
Since $m(r)=0$ for all $r$ in $\R$,
\begin{equation}\label{image}
  \begin{aligned}
 d^*(\alphax) &(1_S\ot x\ot 1_S\ot 1_H\ot 1_H)\\
 &=  \alphax (x\ot 1_S\ot 1_H\ot 1_H - 1_S\ot x\ot 1_H\ot 1_H)\\
 &=  m(1\ot \alpha-\alpha\ot 1)(x)
 \,  .
\end{aligned}
\end{equation}
So $d^*(\alphax + \lambdax)|_{X_{3,0}}= 0$ if and only if
$ 1\ot \alpha - \alpha\ot 1$ has image $0$ in $A$
as a map on $(\R\ot V)\cap (V\ot \R)$ upon projection to $A$, that is,
Condition~(6) holds.

{\bf Brackets in terms of parameter functions.}
Conditions~(b) and (c) in \cref{hom-condns}
involve bracket conditions on $X$
which we express now in terms of the parameter
functions $\lambda$, $\alpha$, and $\beta$.
   Recall that 
   the bracket on $X$ is the lift
   as  
   in~\cref{eqn:Xbracket} of the bracket
   on $\BB_A$
   via the chain maps $\pi:\BB_A\rightarrow 
   X$ and $\iota:X\rightarrow \BB_A$ 
of \cref{HopfKoszulConversion}.
For cocycles $\gammax, \etax$ on $X$,
\begin{equation}\label{BracketDefinition}
  \begin{aligned}
    [\gammax, \etax]
    &=\iota^*[\pi^*\gammax,  \pi^*\etax]_{_{\BB_{\! A} }}\,
    \\ &=
   \iota^*\Big(
   (\pi^*\gammax)\, (\pi^*\etax\ot 1 - 1\ot \pi^*\etax) 
   +
   (\pi^*\etax)\, (\pi^*\gammax\ot 1 - 1\ot \pi^*\gammax)\Big)
   \, .
   \end{aligned}
 \end{equation}

 Also observe that for Condition~(b) in \cref{hom-condns},
 as
 $\lambdax$ and $\alphax$ each have homological 
degree 2, 
$[\alphax,\lambdax]=[\lambdax,\alphax]$, and so 
\[
  [\alphax + \lambdax, \alphax +\lambdax]
  =[\alphax,\alphax]+2[\alphax,\lambdax]+[\lambdax,\lambdax]
  \, .
\]

 {\bf Brackets on $X_{2,1}$.}
On $\R\ot \bar H$, we argue that 
$$
\begin{aligned}
[\lambdax , \lambdax]
&=  2 \lambda_{_L} (\lambda_{_L}\ot 1) (\itau\ot 1)(1\ot\itau)
= 2 \lambda_{_R}\, \tau \, (\lambda_{_R}\ot 1)(1\ot\itau),
\\
  [\alphax,\lambdax ] 
  &=  (\lambda_{_L} \ot 1) (1\ot \alpha)(\itau\ot 1)(1\ot \itau )
=  -(\lambda_{_R} \tau \ot 1) (1\ot \alpha)(\itau\ot 1)(1\ot \itau ),
\\
[\alphax,\alphax]
&= 0\,
   \end{aligned}
   $$
   viewing $\R\ot \bar H\subset X_{2,1}$ for the bracket expressions.
 Observe that for any $2$-cocycles $\kappa$, $\eta$ on $X$,
 \cref{ExplicitValues} implies (see \cref{ThreeTerms})
 that,  on  $\R\ot \bar H\subset X_{2,1}$, 
    $$
   \begin{aligned}
    \iota^* \big(
    (\pi^*\gammax)& \,  (\pi^*\etax\ot 1 - 1\ot \pi^*\etax)\big)
     \\ =& \ 
    (\pi^*\gammax)\, (\pi^*\etax \ot 1 - 1\ot \pi^*\etax) 
    \, \pr_{\BB_A}\, 
    \big((1\ot 1\ot 1) - (1\ot \itau)+(\itau\ot 1)(1\ot \itau)\big)\, .
       \end{aligned}
   $$
Suppressing notation for the projection
  $\pr_{\BB_A}:\B_A\rightarrow \BB_A$, we note that 
  $\pi$ is zero on $V\ot \bar H\subset (\BB_A)_2$ by
  \cref{PiInLowDegree},
  so both
  $(1\ot \pi^*\eta) (1\ot 1\ot 1)$
  and $(\pi^*\eta\ot 1)(1\ot \itau)$ 
are zero on $\R\ot \bar H$
$\subset X_{2,1}$ leaving just 
 \begin{equation}\label{FourTerms}
   \begin{aligned}
    \iota^* \big(
    (\pi^*\gammax)\, & (\pi^*\etax\ot 1 - 1\ot \pi^*\etax)\big)
    \\ &=\ 
    (\gammax\pi)  \big(\etax\pi\ot 1)
    + (\gammax\pi) (1\ot \etax\pi) 
    (1\ot \itau)
    + (\gammax\pi) (\etax\pi \ot 1)
    (\itau\ot 1)(1\ot\itau)
    \\ & \ \ \ \ \ 
   -(\gammax\pi) (1\ot \etax\pi) (\itau\ot 1)(1\ot\itau)\, .
   \end{aligned}
  \end{equation}

{\bf Brackets on $X_{2,1}$: Square bracket of $\alphax$.}
  For $\gammax=\etax=\alphax$, \cref{FourTerms} gives zero because
  the image of  $\pi$ on $\bar H\ot V$ (as a subspace of $(\BB_A)_2$)
  lies in $X_{1,1}$ whereas $\alphax|_{X_{1,1}}\equiv 0$ 
  and likewise the image of $\pi$ on $\R$ (as a subspace of $(\BB_A)_2$)
  lies in $X_{0,2}$ whereas $\alphax|_{X_{0,2}}\equiv 0$.
  Thus $[\alphax,\alphax]=0$ here.

  {\bf Brackets on $X_{2,1}$: Square bracket of $\lambdax$.}
 For $\gammax=\etax=\lambdax$, we observe that 
the image of $\pi$ on $\R$ lies in $X_{2,0}$ whereas $\lambdax|_{X_{2,0}}\equiv 0$,
so the first and fourth terms of  
\cref{FourTerms} vanish.
The second is also $0$ since the image of
$(1\ot \lambdax\pi)(1\ot\itau)$ on $\R\ot \bar
H$
lies in $V\ot \bar H$ where $\pi$ vanishes by \cref{PiInLowDegree},
leaving only the third term 
(see~\cref{PiLambdax})
  $$
  \begin{aligned}
    (\pi^*\lambdax) (\pi^*\lambdax\ot 1)(\itau\ot 1)(1\ot\itau)
     &=
    \lambda_{_L} (\lambda_{_L}\ot 1) (\itau\ot 1)(1\ot \itau)
    \, .
   \end{aligned}
   $$
   Hence (see \cref{LambdaSwitchSides})
   on $\R\ot V\subset (\BB_A)_3$,
   $$
  [\lambdax,\lambdax] 
   = 2 \lambda_{_L} (\lambda_{_L}\ot 1) (\itau\ot 1)(1\ot\itau)
= 2 \lambda_{_R} \tau (\lambda_{_R}\ot 1)(1\ot\itau)
   \, .
   $$

{\bf Brackets on $X_{2,1}$: Mixed bracket of $\alphax$ and $\lambdax$.}
  For $\gammax=\alphax$ and $\etax = \lambdax$, all four terms of
   \cref{FourTerms} vanish for similar reasons.
  For $\gammax=\lambdax$ and $\etax=\alphax$,
observe that the middle two terms of \cref{FourTerms} likewise vanish.  The
   first also vanishes since 
   the image of $\alphax$ lies in $m(V\ot H)\subset A$,
   \cref{PiInLowDegree}  implies that $\pi$ takes $m(V\ot H)\ot H$ to
   $X_{0,2}$,
      and $\lambdax\equiv  0$
   on $X_{0,2}$, leaving only the last term
   of \cref{FourTerms}:
   $$
   \begin{aligned}
     -(\pi^*\lambdax)(1\ot \pi^*\alphax)(\itau\ot 1)(1\ot \itau) 
     =
     -\lambdax\, \pi\, (1\ot m)(1\ot \alpha)(\itau\ot 1)(1\ot \itau) 
\, .
      \end{aligned}
      $$
      Here, we used the fact that
      $(\itau\ot 1)(1\ot \itau)$ takes $\R\ot H$ to $\bar H\ot \R$
      (after projecting $H \rightarrow \bar H$)
   whereas for $r$ in $\R\subset (\B_A)_2$ (see \cref{PiInLowDegree}),
under our identifications,
\begin{equation}\label{PiOnR}
  \pi^*(\alphax) (r) 
=\alphax\pi (r) 
=\alphax(r) 
=m\, \alpha(r) 
\, .
\end{equation}
   But as a map $\bar H\ot V\ot \bar H\subset (\BB_A)_3\rightarrow X$,
   \cref{PiInLowDegree} gives
   $$
   \pi (1\ot m)= \rho_{_R}(\pi\ot 1) + \pi'
     $$
for $\rho_{_R}: X_{1,1}\ot A\rightarrow X_{1,1}$ giving the right $A$-module 
structure of $A$ on $X$ and some map $\pi':(\BB_A)_3\rightarrow X_{0,2}$.
Then as
$\lambdax|_{X_{0,2}}\equiv 0$, we may write,
as a map on $\R\ot \bar H\subset X_{2,1}$,
     $$
   \begin{aligned}
         \lambdax \,\pi\, (1\ot m)(1\ot \alpha)(\itau\ot 1)(1\ot
     \itau)
         & =
     \lambdax \rho_{_R} (\pi\ot 1) (1\ot \alpha)(\itau\ot
     1)(1\ot \itau)
     \\  & =
     m (\lambdax\ot 1) (\pi \ot 1) (1\ot \alpha)(\itau\ot 1)(1\ot
     \itau)
     \\  & =
     m (\pi^* \lambdax \ot 1) (1\ot \alpha)(\itau\ot 1)(1\ot
     \itau)
     \end{aligned}
       $$       
       as $\lambdax$ is an $A$-bimodule map
       (so $\lambdax \rho_{_R} =\rho_{_R}(\lambdax\ot 1)=m(\lambdax\ot
       1)$).
       But $\pi^*\lambdax=\lambda_{_L}$ by~\cref{PiLambdax},
  so
    we may express
     the bracket
on $\R\ot H\subset X_{2,1}$ as
$$
   [\alphax, \lambdax]=
   (\lambda_{_L} \ot 1) (1\ot \alpha)(\itau\ot 1)(1\ot \itau )
\, 
   $$
   upon projection to $A$ via
$m:A\ot A\rightarrow A$, giving 
the brackets on $X_{1,2}$ as claimed.

{\bf Brackets on $X_{3,0}$.}
We argue that
on $(V\ot \R)\cap (\R\ot V)\subset X_{3,0}$,
\begin{equation}\label{BracketsX30}
\begin{aligned}
  [ \alphax , \alphax ]
  &=
2 (\alpha\ot 1)(1\ot \tau)(\alpha\ot 1) 
-2(\alpha\ot 1)(1\ot \alpha),
\\ 
[\alphax, \betax]
  &=
  (\beta\ot 1) 
\big(
(1\ot \tau)(\alpha\ot 1) -(1\ot \alpha) \big),
\\
[\alphax, \lambdax]
&=
-(1 \ot\lambda_{_L}) (\alpha\ot 1),
\\
[\lambdax, \betax]
&= \lambda_{_L}(\beta \ot 1 ),
\quad\text{ and  }
  \\ [\lambdax, \lambdax]
  & = 0
\end{aligned}
\end{equation}
upon projection to $A$ via $m:A\ot A\rightarrow A$,
under our identifications. 

We assume Condition~(6) holds (which we showed above is equivalent to
Condition (a)), i.e., we assume that the image of 
$\alphax\ot 1 - 1\ot \alphax$
on $(V\ot \R)\cap (\R\ot V)\subset X_{3,0}$
projects to zero in $A$ under $m:A\ot A\rightarrow A$.
This implies that on $(V\ot \R)\cap (\R\ot V)$,
\begin{equation}\label{DefinedOn}
\Ima\big((1\ot\tau) (\alpha\ot 1) - (1\ot \alpha)\big)
\subset \R\ot H,
\end{equation}
and hence the maps in Conditions (4) and (5) are defined
on $(V\ot \R)\cap (\R\ot V)$.

To find the indicated brackets, we use \cref{BracketDefinition},
noting that the initial map $\iota^*$ in that expression
acts essentially as the identity,
since for $x$ in
$(V\ot \R)\cap (\R\ot V) \subset X_{3,0}$, 
$\iota(x)$ is just $x$ in $(\BB_A)_3$ under our identifications
 (see the proof of \cref{ExplicitValues}).

First observe that 
$[\lambdax,\lambdax]$ is $0$ on $X_{3,0}$
since $\lambdax$ is zero on $X_{2,0}$ and $X_{0,2}$.

To find brackets with $\alphax$, recall
that $\pi^*(\alphax)=\alpha$ 
as maps on $\R$
under our identifications (see \cref{PiOnR}).
Hence on $(V\ot \R)\cap (\R\ot V) \subset X_{3,0}$, 
 $$
\pi \big(\pi^*(\alphax)\ot 1-1\ot \pi^*(\alphax)\big)\, \iota
\ =\
\pi (\alphax\ot 1-1\ot \alphax)
\, . 
$$
Note that we take $\pi=\pi_{_{S,H}}\, \AWt$
for  chain maps $\AWt: \BB_A\rightarrow \BB_S\ott\BB_H$
and $\pi_{_{S,H}}: \BB_S\ott\BB_H\rightarrow X$
given in \cref {AppendixValuesLowDegree} 
(see the proof of \cref{PiInLowDegree}).

We suppress extra factors of $1_S$ and $1_H$ in $\BB_S\ott \BB_H$,
i.e., we identify (see \cref{LibertyOnMiddleResolution})
$\bar{S}^{\ot n}\ot \bar{H}^{\ot m}$ with
$k1_S\ot \bar{S}^{\ot n}\ot k1_S\ot k1_H \ot\bar{H}^{\ot m}\ot k1_H$
in $(\BB_S)_n\ot (\BB_H)_m$
and use the left and right
$A$-module maps (see \cref{TwistedResolutionAction})
$$
\rho_{_L}:A\ot (\BB_S\ot \BB_H)\rightarrow (\BB_S\ot \BB_H)
\quad\text{ and }
\rho_{_R}:(\BB_S\ot \BB_H)\ot A\rightarrow (\BB_S\ot \BB_H)
\, 
$$
to write generic elements, so, for example,
$(u\ot u'\ot 1_S)\ot (1_H\ot h\ot 1_H)$ in $\BB_S\ot \BB_H$
is just $\rho_{_L}(u\ot (u'\ot h))$ under these identifications.
With a slight abuse of notation, we again use the same notation for the
left and right $A$-module maps on $X$ as well.

The image of $\alphax\ot 1 - 1\ot\alphax$ on 
$(\R\ot V)\cap (V\ot \R)\subset X_3$
lies in the space 
$$(m\ot 1)(V\ot H\ot V)-(1\ot m)(V\ot V\ot H),$$
and we project the image (see \cref{circle-prod-defn,eqn:Xbracket}) 
to 
$\bar A\ot \bar A\subset (\BB_A)_2$
and apply $\pi$, 
writing 
on $(\R\ot V)\cap (V\ot \R)\subset X_3$,
$$
\begin{aligned}
\pi( & \alphax \ot 1\, -\,  1\ot \alphax) 
\\ &=
\pi_{_{S,H}}\AWt(\alphax\ot 1 - 1\ot \alphax)
\\ &=
\pi_{_{S,H}}\big(\rho_{_R}(1\ot \tau)(\alpha\ot 1) 
-\rho_{_L}(1\ot 1\ot \pr_{\bar H})(1\ot \tau)(\alpha\ot 1) -\rho_{_R}(1\ot \alpha) \big)
\\ &=
\pi_{_{S,H}}\, 
\rho_{_R}\big((1\ot \tau)(\alpha\ot 1) -(1\ot \alpha) \big) 
-\pi_{_{S,H}}\rho_{_L}(1\ot 1\ot \pr_{\bar H}) (1\ot \tau)(\alpha\ot 1)
\\ &=
\rho_{_R}(\pi_{_{S,H}}\ot 1)
\big((1\ot \tau)(\alpha\ot 1) -(1\ot \alpha) \big) 
-\rho_{_L}
(1\ot \pi_{_{S,H}})
(1\ot 1\ot \pr_{\bar H})
(1\ot \tau)(\alpha\ot 1) 
\end{aligned}
$$
since $\pi_{_{S,H}}$ is an $A$-bimodule map.
We argue that $\pi_{_{S,H}}$ here acts as the identity in each
expression under our identifications.
Indeed,
for the first term, \cref{DefinedOn} implies that
$\pi_{_{S,H}}\ot 1$ takes each element in
$((\BB_S)_2\ot (\BB_H)_0)\ot
H$
in the image of
$(1\ot \tau)(\alpha\ot 1) -(1\ot \alpha)$
to itself
as an element of $X_{2,0}\ot H$
(see the proof of \cref{PiInLowDegree}).
For the last term, observe that $(1\ot 1\ot \pr_{\bar H}) (1\ot \tau)(\alpha\ot 1)$
has image in $V\ot V\ot \bar H$
and $\pi_{_{S,H}}$ sends
$ v\ot h$ in $(\BB_S)_1\ot (\BB_H)_1$
to $v\ot h$ in $X_{1,1}$ 
(see the proof of \cref{PiInLowDegree}),
so as a map $\R\ot V\rightarrow V\ot \bar H$,
$$\rho_{_L}(1\ot \pi_{_{S,H}} ) (1\ot 1\ot \pr_{\bar H}) (1\ot \tau)(\alpha\ot 1)
= \rho_{_L} (1\ot 1\ot \pr_{\bar H}) (1\ot \tau)(\alpha\ot 1)
\, .
$$
Thus on $(\R\ot V)\cap (V\ot \R)\subset X_3$,
interpreting the right hand side as an expression in $X$,
we may write
$$
\begin{aligned}
\pi(\alphax\ot 1  - 1\ot \alphax) 
&=
\rho_{_R} 
\big(
(1\ot \tau)(\alpha\ot 1) -(1\ot \alpha) \big) 
-\rho_{_L} (1\ot 1\ot \pr_{\bar H})
(1\ot \tau)(\alpha\ot 1) 
\, 
\end{aligned}
$$
with the last term lying in $X_{1,1}$.
Thus on $X_{3,0}$
for a 2-cocycle $\gammax$ in $\{\alphax,\betax,\lambdax\}$,
 $$
 \begin{aligned}
   [\alphax, \gammax]
  &=
   \iota^*\Big(\pi^*\alphax (\pi^*\gammax\ot 1 - 1\ot \pi^*\gammax) 
   +
   \pi^*\gammax (\pi^*\alphax\ot 1 - 1\ot \pi^*\alphax) \Big) 
 \\    &=
\pi^*\alphax (\pi^*\gammax\ot 1 - 1\ot \pi^*\gammax) 
\\    &\ \ \
+
   \gammax \, 
\rho_{_R} 
\big(
(1\ot \tau)(\alpha\ot 1) -(1\ot \alpha) \big) 
-\gammax\, \rho_{_L}  (1\ot 1\ot \pr_{\bar H})
(1\ot \tau)(\alpha\ot 1) 
\\    &=
\pi^*\alphax (\pi^*\gammax\ot 1 - 1\ot \pi^*\gammax) 
\\ & \ \
+m\, (\gammax\ot 1) 
\big((1\ot \tau)(\alpha\ot 1) -(1\ot \alpha) \big) 
\\ & \ \
+m\, (1 \ot\gammax) (1\ot 1\ot \pr_{\bar H}) (1\ot \tau)(\alpha\ot 1) 
 \, . 
\end{aligned}
$$
Of the three terms in this expression, 
the first vanishes for
$\gammax=\betax$ and $\gammax=\lambdax$
since $\Ima \betax, \Ima \lambdax \subset H$ and
$\alphax|_{X_{1,1} } \equiv 0 $,
the second vanishes for $\gammax=\lambdax$
since $\lambdax|_{X_{2,0}}\equiv 0$,
the third vanishes for $\gammax=\alphax$ and $\gammax=\betax$
since $\alphax|_{X_{1,1} }= \betax|_{X_{1,1} } \equiv 0 $,
and the first and second coincide for $\gammax=\alphax$.
Thus
on $(V\ot \R)\cap (\R\ot V)\subset X_{3,0}$,
we can express the brackets with $\alphax$ as
$$
\begin{aligned}
  [ \alphax , \alphax ]
  &=
2 (\alpha\ot 1)(1\ot \tau)(\alpha\ot 1) 
-2(\alpha\ot 1)(1\ot \alpha),
\\ 
[\alphax, \betax]
  &=
  (\beta\ot 1) 
\big(
(1\ot \tau)(\alpha\ot 1) -(1\ot \alpha) \big),
\quad\text{ and  }
\\
[\alphax, \lambdax]
&=
(1 \ot\lambdax) (1\ot 1\ot \pr_{\bar H}) (1\ot \tau)(\alpha\ot 1)
=
(1 \ot\lambda_{_R}) (1\ot \tau)(\alpha\ot 1)
\\ &=
-(1 \ot\lambda_{_L}) (\alpha\ot 1)
\end{aligned}
  $$
  upon repeated projection to $A$ under $m:A\ot A\rightarrow A$.
  Similar analysis 
  gives on $X_{3,0}$
$$
 \begin{aligned}
   [\lambdax, \betax]
  &=
   \iota^*\Big(\pi^*\lambdax (\pi^*\betax\ot 1 - 1\ot \pi^*\betax) 
   +
   \pi^*\betax (\pi^*\lambdax\ot 1 - 1\ot \pi^*\lambdax) \Big) 
\\ &=
(\pi^*\lambdax)(\pi^*\betax \ot 1) -
(\pi^*\lambdax) (1\ot \pi^*\betax) 
   +
   (\pi^*\betax) 
   (\pi^*\lambdax\ot 1) -    (\pi^*\betax)( 1\ot \pi^*\lambdax) 
   \\ &=
\lambda_{_L}(\beta \ot 1 ) 
      \, 
 \end{aligned}
 $$
 since the last 3 terms vanish 
 as $\pi^*\betax (r) = \betax(r) = \beta(r)$ lies in $H$
 for all $r$ in $\R$,
 $\pi(V\ot \bar H)= 0$, $\Ima(\lambda)\subset H$, and
 $\beta|_{X_{1,1}}\equiv 0$
 under our identifications.

{\bf Condition (b).}
Both sides of the equation of Condition~(b) in \cref{hom-condns}
are $0$ on $X_{0,3}$ and $X_{1,2}$, since their graded degree is $-2$. 
We found the values of $[\lambdax,\lambdax]$, $[\alphax,\lambdax]$, 
and $[\alphax,\alphax]$ on $X_{2,1}$ above.  Expressed
in terms of a map on $\R\ot \bar H$,
$$
\begin{aligned}
  [\alphax + \lambdax, & \alphax +\lambdax]
  =[\alphax,\alphax]+2[\alphax, \lambdax] + [\lambdax, \lambdax]
\\&= 2 \lambda_{_L} (\lambda_{_L}\ot 1) (\itau\ot 1)(1\ot\itau) 
+  (\lambda_{_L} \ot 1) (1\ot \alpha)(\itau\ot 1)(1\ot \itau ) 
\, 
\end{aligned}
$$
upon projection to $A$ via the multiplication map $m:A\ot A\rightarrow 
A$. 
We compare with $d^*\betax$ on $X_{2,1}=\R\ot \bar H$.
By \cref{donX21}, on $\R\ot \bar H\subset X_{2,1}$,
the cochain  $d^*(\betax)$ is
 $$
\begin{aligned}
  (1\ot \betax)
  +(\betax\ot 1) (1\ot \pr_{\bar H}\ot 1)
  (1\ot \itau)
 +(1\ot \betax)(\itau\ot 1)(1\ot \itau) - (\betax\ot 1)
\end{aligned}
$$
upon projection to $A$.
But $\betax|_{X_{1,1}}\equiv 0$, so
the first two terms vanish and
 $$ 
\begin{aligned}
   d^*(\betax) &=
   (1\ot \beta)(\itau\ot 1)(1\ot \itau) - (\beta\ot 1)
   \quad\text{ on } \R\ot \bar H\, 
\end{aligned}
$$
upon projection to $A$ under $m:A\ot A\rightarrow A$.
As char $k \neq 2$, we deduce that
$$[\alphax+\lambdax,\alphax+\lambdax] = 2d^*(\betax)
\text{ on $X_{2,1}$ }
$$
if and only if 
$$
\begin{aligned}
  \lambda_{_L} (\lambda_{_L}\ot 1) (\itau\ot 1)(1\ot\itau) 
+  & (\lambda_{_L} \ot 1) (1\ot \alpha)(\itau\ot 1)(1\ot \itau )
\\
&=
(1\ot \beta)(\itau\ot 1)(1\ot \itau) - (\beta\ot 1)
\quad\text{ on $\R\ot \bar H$}
\, 
\end{aligned}
$$
upon projection to $A$ under $m:A\ot A\rightarrow A$.

We obtain Condition (2) by composing both sides
of the last equality
with $(1 \ot \tau)(\tau\ot 1)$ on the right:
$$
\begin{aligned}
\lambda_{_L} (\lambda_{_L}\ot 1) 
+  (\lambda_{_L} \ot 1) (1\ot \alpha)
&=
(1\ot \beta) - (\beta\ot 1) (1 \ot \tau)(\tau\ot 1)
\quad\text{ on $H\ot \R$.}
\end{aligned}
$$

For Condition (b) on $X_{3,0}$, we again use the bracket
values found above:
$$
\begin{aligned}
  [ \alphax + \lambdax , \alphax+\lambdax ]
  &=
  [ \alphax , \alphax ]
  + 2 [ \alphax , \lambdax ] + [ \lambdax , \lambdax ]
  =
  [ \alphax , \alphax ]
  + 2 [ \alphax , \lambdax ]
\end{aligned}
$$
on $X_{3,0}$ can be expressed as the map
  $$
\begin{aligned}
2 (\alpha\ot 1)(1\ot \tau)(\alpha\ot 1) 
-2(\alpha\ot 1)(1\ot \alpha)
-2(1\ot \lambda_{_L})(\alpha\ot 1)
\quad\text{ on }
(V\ot \R)\cap (\R\ot V) 
\, .
\end{aligned}
$$
We compare with $d^*\betax$.
On $(V\ot \R)\cap (\R\ot V) \subset X_{3,0}$,
the differential is
$d= \rho_{_L}-\rho_{_R}$
since $m(\R)\equiv 0$, 
so
$d^*\betax=\betax\rho_{_L}-\betax\rho_{_R} 
=m(1\ot\betax)-m(\betax \ot 1 )$,
i.e.,
$d^*\betax$
may be expressed as the map
$1\ot\beta-\beta \ot 1 
$ on $(V\ot \R)\cap (\R\ot V)$
upon projection to $A$.
Thus on $X_{3,0}$,  
  $$[ \alphax + \lambdax , \alphax+\lambdax ]= 2d^*\betax$$
if and only if as a map on $(V\ot \R)\cap (\R\ot V)$,
$$
 (\alpha\ot 1)(1\ot \tau)(\alpha\ot 1) 
-(\alpha\ot 1)(1\ot \alpha) 
-(1\ot \lambda_{_L})(\alpha\ot 1) 
=(1\ot\beta)-(\beta \ot 1 ) 
$$
(with the maps well-defined 
on the domain, see \cref{DefinedOn})
upon projection to $A$ 
giving Condition (4).

{\bf Condition (c).}
The left side of the equation in Condition (c) of \cref{hom-condns} 
is $0$ on $X_{0,3}$, $X_{2,1}$, and $X_{1,2}$ for degree reasons.
It remains to check  $X_{3,0}$.
From the above bracket values, we find that
$[\alphax+\lambdax,\betax]= 0$ on $X_{3,0}$
if and only
if $$(\beta\ot 1)\big((1\ot\tau)(\alpha\ot 1) - (1\ot \alpha)\big) = 
- \lambda_{_L} (\beta\ot 1) \, 
$$
on $(V\ot \R)\cap (\R\ot V)$
(with the map well-defined 
on the domain, see \cref{DefinedOn})
upon projection to $A$.
This is Condition~(5).
\end{proof}


\vspace{3ex}

We saw in the proof of \cref{thm:main-Hopf-left}
that 
  the first condition may be written in terms of the vanishing of a map on
  $V\ot H\ot H$ instead of $H\ot H\ot V$ if we identify $\lambda$
  with a function on $V\ot H$ instead of $H\ot V$.
  (See \cref{SwitchSides}.) 
  Likewise, we can rewrite the remaining
  conditions in this way, for example 
  by composing on the right with $(\itau \ot 1)(1\ot \itau)$ in
  Condition (3).  
Again, in order to simplify notation in the next corollary,
we apply the multiplication map $T_k(A)\rightarrow A$,
$a_1\ot\cdots\ot a_n\mapsto a_1\cdots  a_n$
to each summand in each condition after identifying $V$ and $H$ with
subspaces of $A$.

\vspace{1ex}

\begin{cor}    
\label{thm:main-Hopf-right}
A filtered algebra $\cH$ is a strong PBW deformation of $S\# \Ho$ if
and only if $\cH=\chabl$ for some linear parameter functions
$\alpha,\beta,\lambda$ satisfying
\begin{itemize}
  \item[$(1)$] $(1\ot \lambda_{_R})(\itau\ot 1) - \lambda_{_R} (1\ot m) +
    (\lambda_{_R}\ot 1) = 0$,
   \item[$(2)$]
    $
  \lambda_{_R}\tau (\lambda_{_R}\tau\ot 1) 
-   (\lambda_{_R}\tau \ot 1) (1\ot \alpha)
=
(1\ot \beta)- (\beta\ot 1) (1\ot \tau)(\tau\ot 1)
\, , 
$
  \item[$(3)$] $(\alpha\ot 1) - (1\ot\alpha)(\itau\ot 1)(1\ot \itau) 
  = (1\ot\lambda_{_R})   + (\lambda_{_R}\ot 1)(1\ot \itau)$,
        \item[$(4)$] $(\alpha\ot 1)\big((1\ot\tau)(\alpha\ot 1) - (1\ot \alpha)\big) 
       + (1\ot \lambda_{_R}\tau)(\alpha\ot 1) 
     = 1\ot\beta - \beta\ot 1$,
\item[($5$)] $(\beta\ot 1)\big((1\ot\tau)(\alpha\ot 1) - 1\ot\alpha\big)
    = \lambda_{_R}\tau (\beta\ot 1)$,
    \item[(6)] $\alpha\ot 1 - 1\ot\alpha =0$,
\end{itemize}
after applying the multiplication map $T_k(A)\rightarrow A$
to each summand in each condition, for $\lambda:\R'\rightarrow H$
identified 
with the map $V\ot H\rightarrow H, 
\ v \ot h \mapsto 1_H\ot v\ot h- \itau(v\ot h)\ot 1_H$.

Here, the maps in (1) are defined on $V\ot H\ot H$,
the maps in (2) on $H\ot \R$,
the maps in (3) on $\R\ot H$,
the map in (6) on $(V\ot \R)\cap (\R\ot V)
\subset V\ot V\ot V$,
and (6) implies that the maps in (4) and (5) are defined
on $(V\ot \R)\cap (\R\ot V)$.

  \end{cor}


  Since every strong PBW deformation is a PBW deformation, we have
\begin{cor}\label{6ConditionsImplyWeakPBW}
  Let $H$ be a Hopf algebra and
suppose the six  conditions of \cref{thm:main-Hopf-left} or of 
\cref{thm:main-Hopf-right}
hold for some linear parameter functions 
$\alpha,\beta,\lambda$. 
Then $\chabl$ is a PBW deformation of $S\# \Ho$. 
\end{cor}

\cref{thm:main-Hopf-left} and \cref{prop:doubly-PBW} imply the next corollary.
\begin{cor}\label{6ConditionsDoublyNoetherianCase}
  Assume $H$ is a doubly Noetherian Hopf algebra
  acting on a graded Koszul algebra $S$.
A filtered algebra $\cH$ is a PBW deformation of $S\# \Ho$ if
and only if $\cH= \chabl$ for some
linear parameter functions 
$\alpha,\beta,\lambda$ satisfying
the conditions of
\cref{thm:main-Hopf-left}
(or the equivalent conditions of \cref{thm:main-Hopf-right}).
\end{cor}

There are many examples 
of such PBW deformations $\chabl$
already in the literature,
and our results combine most of them under one roof. 
See for example~\cite{FlakeSahi,SW-Koszul,WW}, 
and the references provided therein.

\vspace{2ex}

\subsection*{Polynomial rings}
Now we restrict to the case where $S=S(V)$, the symmetric algebra on
$V$,
i.e., a polynomial ring.
Here $\mathcal{R}=\{v\ot w-w\ot v:v,w\in V\}$.
Identify $\alpha(v,v')$ with $\alpha(v\ot v'-v'\ot v)$ and similarly
for $\beta,\lambda$.
For convenience we write, symbolically,
\[
\begin{aligned}
    \alpha(v,v') &= 
\sum \alpha(v,v')_1 \ot \alpha(v,v')_2
&&\quad\text{ in $V\ot \Ho$}
\qquad\quad\text{and}\\
 (1\ot \Delta)(\alpha(v,v')) &= 
\sum \alpha(v,v')_1 \ot \alpha(v,v')_2 \ot \alpha(v,v')_3
&&\quad\text{ in $V\ot \Ho\ot \Ho$}
\end{aligned}
\]
in analogy with Sweedler's notation.

The following corollary should be compared with~\cite[Theorem 2.5]{Khare},
where $\Ho$ was instead taken to be a cocommutative algebra over
a commutative ring, and $\alpha$ had image in $V$.
Denote by $\Alt_3$ the alternating group on 3 symbols.
Compare  also with \cite[Theorem 2.5]{SW-Koszul}.

\begin{cor}
Suppose that $S=S(V)$.
Then a filtered algebra $\cH$ is a strong PBW deformation of $S\# \Ho$ if and only if
$\cH =\chabl$ for some linear parameter functions
$\alpha,\beta,\lambda$ satisfying 
\begin{itemize}\setlength{\itemindent}{-2ex}
\item[(1)]
   $\lambda(h h',v) = h\lambda(h',v) + \sum 
  \lambda(h , {}^{h_1'}v ) h_2'$,
\item[(2)]
    $h\beta(v,v') -\sum \beta( {}^{h_1}v, {}^{h_2}v') h_3 
    = \lambda(\lambda(h,v),v')-\lambda(\lambda(h,v'),v)
    +\lambda(h , \alpha(v,v'))$,
    \rule{0ex}{3ex}
\item[(3)]
   $ \sum \big( {}^{h_1}\alpha(v,v') h_2 - \alpha( {}^{h_1}v,
   {}^{h_2}v') h_3 \big)$
  \rule{0ex}{3ex}
  
$ = \sum\big( {}^{\lambda(h,v)_1}v'  \, \lambda(h,v)_2
    - {}^{\lambda(h,v')_1}v\, \lambda(h,v')_2\big)$
$ + \sum \big( {}^{h_1}v\, \lambda(h_2,v') - 
    {}^{h_1}v'\, \lambda(h_2,v)\big)$,
    \rule{0ex}{3ex}
  \item[(4)]
        $\sum\sum_{\sigma\in \Alt_3} \alpha \big( \alpha(v_{\sigma(1)},v_{\sigma(2)})_1,  
    {}^{\alpha(v_{\sigma(1)},v_{\sigma(2)})_2}v_{\sigma(3)} \big)
   \ \alpha(v_{\sigma(1)},v_{\sigma(2)})_3$
    \rule{0ex}{3ex}

   $ -\sum\sum_{\sigma\in\Alt_3}
  \alpha\big(v_{\sigma(1)},\alpha(v_{\sigma(2)},v_{\sigma(3)})_1\big) 
\ \alpha(v_{\sigma(2)},v_{\sigma(3)})_2$
\rule{0ex}{3ex}

$- \sum \sum_{\sigma\in\Alt_3} \alpha(v_{\sigma(1)},v_{\sigma(2)})_1\
\lambda\big(\alpha(v_{\sigma(1)},
   v_{\sigma(2)})_2, v_{\sigma(3)}\big)$
   \rule{0ex}{3ex}
   
   \hspace{-3ex}
   $= \sum_{\sigma\in\Alt_3} v_{\sigma(1)}\, \beta(v_{\sigma(2)},v_{\sigma(3)}) 
   - \sum\sum_{\sigma\in\Alt_3} {}^{\beta(v_{\sigma(1)},v_{\sigma(2)})_1}v_{\sigma(3)} 
     \ \beta(v_{\sigma(1)},v_{\sigma(2)})_2$,
     \rule{0ex}{3ex}
   \item[(5)] 
   $\sum_{\sigma\in\Alt_3} \lambda\big(\beta(v_{\sigma(1)},v_{\sigma(2)}) , v_{\sigma(3)}\big)$  
\rule{0ex}{3ex}

$ = \ \ \  \sum_{\sigma\in\Alt_3} \beta\big(
\alpha(v_{\sigma(1)},v_{\sigma(2)})_1, 
   {}^{\alpha(v_{\sigma(1)},v_{\sigma(2)})_2} v_3\big) 
\,  \alpha(v_{\sigma(1)}, v_{\sigma(2)})_3$
  \rule{0ex}{3ex}

$ \ \ \ - \sum_{\sigma\in\Alt_3} \beta\big(v_{\sigma(1)}, \alpha(v_{\sigma(2)},v_{\sigma(3)})_1\big)
    \ \alpha(v_{\sigma(2)},v_{\sigma(3)})_2$,
    \rule{0ex}{3ex}
  \item[(6)]
   $\sum_{\sigma\in\Alt_3} \big(\alpha(v_{\sigma(1)},v_{\sigma(2)})\, v_{\sigma(3)} 
   - v_{\sigma(1)}\, \alpha(v_{\sigma(2)},v_{\sigma(3)}) \big) = 0$,
   \rule{0ex}{3ex}
\end{itemize}
upon projection of images of the maps to $S\# \Ho$
for all $h, h'\in \Ho$, $v,v', v_1, v_2, v_3$ in $V$. 
\end{cor}

\begin{proof}
Examination of the expressions in Conditions (1)--(6) of
\cref{thm:main-Hopf-left} reveals these alternative 
ways to write the conditions in this special case.
\end{proof}




\section{Acknowledgments}
We thank Apoorva Khare
for invigorating conversations on many ideas here.
The first author was partially supported by Simons grants 429539 and 949953.
The second author was partially supported by National Science Foundation grant 2001163. 
This material is based upon work supported by the National Science Foundation 
under Grant No.~DMS-1928930 and by the Alfred P.~Sloan Foundation under 
grant G-2021-16778, while the second author was in residence at the 
Simons Laufer Mathematical Sciences Institute (formerly MSRI) in Berkeley, 
California, during the Spring 2024 semester.

\appendix
\section{Chain map to the bar resolution}
\label {AppendixValuesLowDegree}

We take $H$ to be a Hopf algebra acting on a graded Koszul algebra
$S = T(V)/ (\R)$, 
with grading preserved by the action and
notation as in Sections~\ref{sec:spr} and~\ref{sec:ParametersToCochains}.  
We describe the way the $S\# H$-bimodule maps 
  of \cref{HopfKoszulConversion},
  $$
\pi: \BB_{S\#\Ho}   \longrightarrow X\sd 
\quad\text{and}\quad 
\iota: X\sd \longrightarrow  \BB_{S\#\Ho} ,
$$
may be chosen to preserve the algebraic relations defining
the Hopf algebra and the smash product structure
in low degree.

\subsection*{Twisted Alexander-Whitney and Eilenberg-Zilber maps}
To this end,
we express $\pi$ and $\iota$
as compositions with the twisted Alexander-Whitney
and Eilenberg-Zilber chain maps $\AWt$ and $\EZt$
of \cite{TTP-AWEZ}:
   \begin{equation}\label{BigAWEZDiagram}
      \entrymodifiers={+!!<0pt,\fontdimen22\textfont2>}
\xymatrixcolsep{2ex}
\xymatrixrowsep{8ex}
\xymatrix{
& \ar@(dl, ul)[dd]_{\pi} &
(\BB_{S\# \Ho})_{n+1}   
\ar[rr]  \ar@<-1ex>[d]_{\rm{AW}^\tau}
&  & 
 (\BB_{S\# \Ho})_{n}
\ar@<-1ex>[d]_{\rm{AW}^\tau}
&  & 
 \\
& & 
(\BB_{S} \otimes_{\tau} \BB_H )_{n+1} 
 \ar[rr]  \ar[u]<-1ex>_{\rm{EZ}^\tau}
 \ar@<-1ex>[d]_{\pi_{_{S,H}}} 
 & & 
(\BB_{S} \otimes_{\tau} \BB_H )_{n} 
\ar[u]<-1ex>_{\rm{EZ}^\tau}  
\ar@<-1ex>[d]_{\pi_{_{S,H}}}  
&  & 
  \\ 
  & &
X_{n+1} 
 \ar[rr]  \ar[u]<-1ex>_{\iota_{_S} \ot\,  \iota_{_H}}
 & & 
X_{n} 
\ar[u]<-1ex>_{\iota_{_S} \ot\,  \iota_{_H}}
  & \ar@(ur, dr)[uu]_{\iota} & 
\, .}
\end{equation}
%
Here, $\BB_S\ot_\tau \, \BB_\Ho $ is the
twisted product resolution 
of reduced bar complexes
$\BB_S$ and $\BB_H$ with
$n$-th term (see \cite{TTP-AWEZ})
$$
\begin{aligned}
 (\BB_S\ot_{\tau} \, \BB_\Ho )_n 
&=\bigoplus_{\ell=0}^{n}\ 
 S\ot \bar{S}^{\, \ell}\ot S\ot
 \Ho\ot \bar{\Ho}^{\,\ot (n-\ell)}\ot \Ho
 \, .
\end{aligned}
$$
Recall that, in comparison, the reduced bar resolution of $S\# H$
has $n$-th term
$$
\begin{aligned}
(\BB_{S\# H})_n 
&= (S\# H) \ot (\overline{S\# H})^{\ot n} \ot ( S\# H)
\, .
\end{aligned}
$$

\subsection*{Identifications in the bar complex for clarity}
In the next two lemmas and their proofs,
we suppress extra tensor factors of $1_A$ on the outside
and identify
\begin{equation*} 
 \text{ 
   $V\ot \bar H$, \ $\bar H \ot V$, \ $V\ot V$, \ $\R'$, and $\R$}
\end{equation*}
    with 
    vector subspaces of
    $(\BB_A)_2 = A\ot \bar A \ot \bar A \ot A
  \subset  A\ot A \ot A \ot A = (\B_A)_2$
  whenever working in $\BB_A$ 
  for $A= S\# H$ to avoid distracting extra notation,
  see \cref{liberty1,liberty2,liberty3,liberty4}.
  We similarly identify $\R\ot V$, $V\ot \R$, $\R'\ot V$, $V\ot \R'$,
  $\bar H\ot \R$, $\R\ot \bar H$ with subspaces of $A\ot \bar A\ot
  \bar A \ot \bar A \ot A =(\BB_A)_3 \subset (\B_A)_3$.

  \begin{lemma}\label{ExplicitValues}
    For $A=S\# H$, 
    the chain map $\iota:X\rightarrow \BB_{A}$ can be chosen so that
    \begin{itemize}
    \item
     $\iota_2$ maps $X_{1,1}$
      onto the $A$-bimodule span of
      $\R'$,
\item 
     $\iota_2$ maps $X_{2,0}$
     onto the $A$-bimodule span of
     $\R$,
     \item
  $\iota_3$ maps $X_{3,0}$
   onto the $A$-bimodule span of
$  ( \R \ot V)
\cap ( V\ot \R )$, and
     \item
  $\iota_3$ maps $X_{2,1}$
to the $A$-bimodule span in $(\BB_A)_3$ of
  $$
 \big( (V\ot \R') 
\oplus 
(\bar H\ot \R)\big)
\ \cap\ 
\big((\R'\ot V) 
\oplus
(\R\ot \bar H)\big)\, .
$$
\end{itemize}
  \end{lemma}
  \begin{proof}
     We appeal to the explicit construction of the chain map $\iota$ from
    \cite[Theorem 11.8]{TTP-AWEZ}: 
    $$\iota=\EZt(\iota_{_S}\ot 1_{\BB_H})$$ for
    inclusion chain map $\iota_{_S}: K_S\rightarrow \BB_S$
    of \cref{KoszulStandardEmbedding} and 
    twisted
  Eilenberg-Zilber chain map
  \[
     \EZt=\pr_{\BB_A} \, \varrho^{-1}\, \theta
\]
  given as a composition of
  a shuffle map $\theta$, a twisted perfect shuffle map
  $\varrho^{-1}$,
  and projection 
  $\pr_{\BB_A}:\B_A\mapsto \BB_A$, see
  \cite[Section 4]{TTP-AWEZ}.

  {\bf The chain map $\iota$ on $X_{1,1}$.}
  Take
  $$x=(1_S\ot v\ot 1_S)\ot (1_H\ot h\ot 1_H)
  \qquad\text{in $X_{1,1}\subset X_2=(K_S\ott \BB_H)_2$}
  $$
  for $h\in \bar{\Ho}$ and $v \in V$. 
  Then $x$ is sent to itself under $\iota_{_S}\ot \iota_{\BB_H}$
  and we compute the value under 
$\EZt$ (see~\cite[(4.3), (4.7)]{TTP-AWEZ}).
  The shuffle map $\theta_{1,1}$
  is a signed sum over two shuffles in $\mathfrak{S}_{1,1}\subset \mathfrak{S}_2$
  (the identity permutation and the swap permutation)
  taking $x$ 
  to
  $$1_S\ot (v\ot 1_S)\ot 1_S \ot 1_H\ot (1_H\ot h)\ot 1_H
  -
  1_S\ot (1_S\ot v)\ot 1_S \ot 1_H\ot (h\ot 1_H)\ot 1_H
  $$
  which is then sent by $\varrho^{-1}$ (given by iterated twisting by $\tau^{-1}$)
to
\begin{equation}\label{IotaOnX11}
\begin{aligned}
 v\ot h
 & -\sum h_2 \ot \, ^{\igamma(h_1)}v
 \ =\  v\ot h - \tau^{-1}(v\ot h)
 \qquad\text{ in $(\B_A)_2$},
\end{aligned}
\end{equation}
using the identifications \cref{liberty1,liberty2,liberty3}.  We project to $\BB_A$ to see that
$\iota$ maps $X_{1,1}$ to the $A$-bimodule span of $\R'$ (see \cref{liberty4}).

{\bf The chain map $\iota$ on $X_{2,0}$ and on $X_{3,0}$.}
Now consider (symbolically) 
$$y=
(1_S\ot v\ot w \ot 1_S)
\ot (1_H\ot 1_H)
$$
in $X_{2,0}$
for $v,w$ in $V$.
Then one can check that $\iota$ takes $y$ to 
  $v \ot w $
since the only shuffle is the trivial shuffle.
  So for any $r$ in $\R$,
\begin{equation}\label{eqn:r-to-r}
  \iota(1_S\ot r\ot 1_S\ot 1_H\ot 1_H)
  =\ r 
  \end{equation}
  and hence the image of 
  $\iota$ on  $X_{2,0}$ is the $A$-bimodule span
  of $\R$.
A similar argument, considering the degree~3 component of 
the Koszul resolution given by~\cref{Koszulterms}, verifies the statement
for image of $\iota$ on $X_{3,0}$.

{\bf The chain map $\iota$ on $X_{2,1}$.} 
    We apply $\iota$ to input 
  $$(1_S\ot r \ot 1_S)\ot (1_H\ot h \ot 1_H)$$
  for $h$ in $\bar{H}$ 
and $r$ in $\R$. 
We use the standard embedding (see~\cref{eqn:iotaSn})
with $\iota_{_H}$ the identity on $\BB_H$
$$\iota_{_S}\ot \iota_{_H}: X_{2,1}\longrightarrow (\BB_S)_2\ott (\BB_H)_1\, . $$
  We first compute $\EZt$ on 
  $$y=(1_S\ot v\ot w \ot 1_S)\ot (1_H\ot h \ot 1_H)$$
  for $v,w$ in $V$ and $h$ in $\bar H$
 noting that three shuffles in the symmetric group $\mathfrak{S}_3$
 are applied, namely, the identity, $(2\ 3)$, and $(1\ 2\ 3)$.
 Under the shuffle map $\theta_{2,1}$,
 $$
  \begin{aligned}
   y\ \ 
    \mapsto \ \ \ & 
    (1_S\ot v\ot w\ot 1_S \ot 1_S)\ot (1_H\ot 1_H\ot 1_H\ot h \ot 1_H) \\
 - \, &  (1_S\ot v\ot 1_S\ot w \ot 1_S)\ot (1_H\ot 1_H\ot h\ot 1_H\ot 1_H) \\
   + \, &
        (1_S\ot 1_S\ot v\ot w \ot 1_S)\ot (1_H\ot h\ot 1_H\ot 1_H\ot 1_H) 
      \end{aligned}
      $$
which in turn is sent by the twisted perfect shuffle map
$\varrho^{-1}$
to
\begin{equation}\label{ThreeTerms}
\begin{aligned}
 v\ot w\ot h 
  -  (1\ot \tau^{-1})(v\ot w\ot h)
  +   (\tau^{-1}\ot 1)(1\ot \tau^{-1})(v\ot w\ot h) 
  \quad\text{ in $(\B_A)_3$.}
\end{aligned}
\end{equation}
We argue this projects into the claimed intersection in $(\BB_A)_3$.
On one hand, the sum of the first two terms 
projects to the subspace 
$V\ot \R'$ of $(\BB_A)_3$ (see \cref{liberty4}).
On the other hand, the sum of the
last two terms 
projects to the subspace $\R'\ot V$ of $(\BB_A)_3$.
This implies that the image
of $(1_S\ot r \ot 1_S)\ot (1_H\ot h \ot 1_H)$
in $X_{2,1}$
under $\iota$ lies, on one hand, in
$$(V\ot \R')
\oplus
( \bar H\ot \R),
$$
and, on the other hand,
in 
$$(\R\ot \bar H)
\oplus (\R'\ot V).
$$
This is because
$
(\tau^{-1}\ot 1)(1\ot \tau^{-1})(r\ot h)\subset H\ot \R
$
as the subspace of Koszul relations
$\R\subset V\ot V$ is preserved
by the twisting map $\tau$ and its inverse 
(see~\cref{hopfpreserveskoszulrelations}).
\end{proof}

We now record some specific values in low degree 
of a chain map 
$ \pi:\BB_A\longrightarrow X$
from~\cref{HopfKoszulConversion}
that are used in proofs in this paper. 
We again use the identifications of \cref{liberty1,liberty2,liberty3,liberty4} when working in $(\BB_A)_2$ or $(\B_A)_2$
to avoid overly technical notation.
We suppress
tensor symbols in $A$ to avoid confusion with tensor symbols
in $\BB_A$ or $\B_A$,
writing $aa'$ for the product of $a$ and $a'$ in $A$, 
and we further identify any element 
$vh\ot v'h'$ in $\bar A\ot \bar A$
with $1_A\ot vh \ot v' h'\ot 1_A$ in $(\BB_A)_2\subset
(\B_A)_2$ for $h, h'$ in $\bar H$ and $v, v'$ in $V$,
see \cref{ChoiceOfSection}.
\begin{lemma}\label{PiInLowDegree}
  The chain map $ \pi:\BB_A\longrightarrow X$
  from~\cref{HopfKoszulConversion}
  may be chosen so that
  for all $h,h'$ in $\bar{H}$, 
  $v,v'$ in $V$, and $r$ in $\R$, 
\begin{equation*}
  \begin{aligned}
    &\pi(v\ot h)
    &&= 0\, ,\\
    &\pi(h\ot h')
    &&= (1_S\ot 1_S)\ot (1_H\ot h \ot h'\ot 1_H)  &&\quad\text{ in } X_{0,2},\\
    &\pi(h\ot v)
    &&= 
      - \sum (1_S \ot \, ^{h_1} v\ot 1_S)\ot (1_H\ot \bar h_2\ot 1_H)
     &&\quad\text{ in } X_{1,1}\, , \\
     &\pi( vh\ot h')
     &&= 
  \  ( v\ot 1_S)\ot (1_H\ot h\ot h'\ot 1_H)
&&\quad\text{ in } X_{0,2}\, , \\
    &\pi\, ( r)
    &&=
(1_S\ot r \ot 1_S)\ot (1_H\ot 1_H)
&&\quad\text{ in } X_{2,0}\, , \ \ { and } \\
&\pi(h\ot v'h')
&&= 
- \sum (1_S \ot\,  ^{h_1}v' \ot 1_S)\ot (1_H\ot \bar h_2\ot h') \\
 & && \quad +  \sum ( \, ^{h_1}v'\ot 1_S)\ot (1_H\ot \bar h_2\ot h'\ot 1_H)
  &&\quad\quad\text{ in } X_{1,1}\oplus X_{0,2}\, ,
\end{aligned}
\end{equation*}
for $\bar h_2$ the image of $h_2$ under the projection $H\rightarrow
\bar H$.
\end{lemma}
\begin{proof}
The chain map
$\pi:\BB_A\rightarrow X$
of \cref{HopfKoszulConversion} (see \cite{TTP-AWEZ})
can be used to show existence of an $A$-bimodule map
 $$\pi_{_{S,H}}
: \BB_S\ot_{\tau}\BB_H \rightarrow K_S\ott \BB_H 
\, 
$$
that is left inverse to $\iota_{_{S,H}}= \iota_{_S}\ot \iota_{_H}$
(that is, $\pi_{_{S,H}}\iota_{_{S,H}} = 1$)
by \cite[Corollary 11.10]{TTP-AWEZ},
which one can then use to redefine
$\pi$ as the composition
(see \cref{BigAWEZDiagram})
$$
\pi= \pi_{_{S,H}}\, \AWt 
$$
for the twisted Alexander-Whitney map
$\AW^{\tau}: \BB_A\longrightarrow 
\BB_S\ot_{\tau}\BB_H$, see \cite[Section 4]{TTP-AWEZ}.

We prefer a more direct construction of $\pi$ here
(rather than redefining it)
and observe that existence of $\pi_{_{S,H}}$ follows
alternatively from
\cite[Lemma 9.1]{TTP-AWEZ}
after we verify the hypothesis, namely, that the $A$-bimodule
$(\BB_S\ot_\tau \BB_H)/\Ima(\iota_{_S}\ot \iota_{_H})$ is projective as 
an $A^e$-module in each degree.
First we recall that both $X=K_S\ot_{\tau} \BB_H$ and $\BB_S\ot_{\tau} \BB_H$
are in fact free $A^e$-module resolutions, and in particular there
are $A^e$-module isomorphisms
\begin{equation}\label{LibertyOnMiddleResolution}
  \begin{aligned}
X_{i,j}  & \ \cong\ && A \ot (\widetilde{K}_S)_i\ot \bar H ^{\ot j} \ot A ,\\
(\BB_S\ot_{\tau}\BB_H)_{i,j} & \ \cong \ &&A \ot \ \, \bar S^{\ot i}\ \ot \bar H^{\ot j}
     \ot A 
   \end{aligned}
   \end{equation}
compatible with the embedding $\iota_{_S}\ot \iota_{_H}$ of $X$ into $\BB_S\ot_{\tau}\BB_H$,
where $(\widetilde{K}_S)_i$ is as in \cref{Koszulterms}
with $\tilde K_0=k$ and $\tilde K_1=V$, 
see \cref{LibertyOnX}.
Thus a free basis of $X_{i,j}$ and of $(\BB_S\ott
\BB_H)_{i,j}$
can be chosen
as a vector space basis of $(\widetilde{K}_S)_i\ot  \bar H ^{\ot j}$ and of 
$\bar S^{\ot i}\ot \bar H^{\ot j}$, respectively.
Now for each $i,j$, choose a vector space complement of 
$(\widetilde{K}_S)_i \ot \bar H^{\ot j}$ in 
$\bar S^{\ot i}\ot \bar H^{\ot j}$.
This is then
 a free basis of a free module complement of $X_{i,j}$ in
$(\BB_S\ot_{\tau}\BB_H)_{i,j}$. 
Therefore $((\BB_S\ot_{\tau} \BB_H)/\Ima(\iota_{_S}\ot \iota_{_H}))_{i,j}$
is a free $A^e$-module for each $i,j$.

To verify the claimed image of $\pi$,
we again use the identifications of \cref{liberty1,liberty2,liberty3}.
Since $\pi_{_{S,H}}$ is left inverse to $\iota_{_S}\ot \iota_{_H}$, 
it plays the role of the identity map 
on elements in the image of $\iota_{_S}\ot \iota_{_H}$. In particular, 
for $h, h'$ in $\bar H$, $v$ in $V$, $r$ in $\R\subset V\ot V$, 
$$
\begin{aligned}
  &\pi_{_{S,H}}:(\BB_S)_2\ot (\BB_H)_0
\rightarrow (K_S)_2\ot (\BB_H)_0
&&\ \text{ maps } (1_S\ot r\ot 1_S)\ot (1_H\ot 1_H)
\text{ to itself,}
\\
&\pi_{_{S,H}}:(\BB_S)_1\ot (\BB_H)_1 
\rightarrow (K_S)_1\ot (\BB_H)_1 
&&\ \text{ maps }
(1_S\ot v\ot 1_S)\ot (1_H\ot h \ot 1_H )
\text{ to itself,}
\\
&\pi_{_{S,H}}:(\BB_S)_0\ot (\BB_H)_2
\rightarrow (K_S)_0\ot (\BB_H)_2
&&\ \text{ maps }
(1_S\ot 1_S)\ot (1_H\ot h\ot h' \ot 1_H )
\text{ to itself}
.
\end{aligned}
$$

We compute $\pi=\pi_{_{S,H}}\circ \AWt$ on generic input of low degree.
Under the chain map $\AWt: \BB_A\rightarrow \BB_S\ott \BB_H$ (see~\cite[Section 4]{TTP-AWEZ}),
for any $h,h'$ in $\bar H$ and $u,u'$ in $V$,
\begin{equation}\label{AWt}
  \begin{aligned}
  uh\ot u'h'\ 
  &\longmapsto&
&\sum (1_S\ot u \ot \, ^{h_1} u'\ot 1_S)\ot (1_H\ot h_2\ot h'\ot
  1_H)\\
  &\longmapsto&
  &\sum (1_S\ot u \ot \, ^{h_1} u'\ot 1_S)\ot (1_H\ot h_2h')\\
  && - &\sum (u \ot \, ^{h_1} u'\ot 1_S)\ot (1_H\ot \bar h_2\ot h')\\
 && +  &\sum (u \ ( ^{h_1} u')\ot 1_S)\ot (1_H\ot \bar h_2\ot h'\ot 1_H)\, 
\end{aligned}
\end{equation}
Fix $h,h'$ in $\bar H$ and $v,v'$ in $V$.
As we are working with the reduced bar resolution of $S$
and elements of $H$ act on $1_S$ by a scalar,
\cref{AWt} implies that 
\begin{equation*}
  \begin{aligned}
 \AWt(h\ot h') 
 = (1_S\ot 1_S) \ot (1_H\ot h\ot h'\ot 1_H)
 \end{aligned}
\end{equation*}
upon which
$\pi_{_{S,H}}$ acts essentially as the identity 
giving the claimed value for $\pi(h\ot h')$.
Likewise,
\begin{equation*}
  \begin{aligned}
  \AWt( v\ot h) \ \  
 & = \ \ \ 
  \sum (1_S\ot v \ot 1_S \ot 1_S)\ot (1_H\ot h)\\
  & \quad  - \sum (v\ot 1_S \ot 1_S)\ot (1_H\ot 1_H\ot h)\\
  & \quad +  \sum (v \ot 1_S)\ot (1_H\ot 1_H \ot h\ot 1_H) \ \  = \ \   0
\end{aligned}
\end{equation*}
and thus 
$    \pi(v\ot h) = 
\pi_{S,H} \AW^{\tau} ( v\ot h) = 
0$, 
whereas
\begin{equation*}
  \begin{aligned}
  \AWt(h\ot v) 
 & = \ \ \ 
 \sum (1_S\ot 1_S \ot \, ^{h_1} v\ot 1_S)\ot (1_H\ot h_2) \\
& \quad   - \sum (1_S \ot \, ^{h_1} v\ot 1_S)\ot (1_H\ot \bar h_2\ot 1_H)\\
&\quad   +  \sum (1_S \ ^{h_1} v\ot 1_S)\ot (1_H\ot \bar h_2\ot 1_H\ot 1_H) \\
 & = 
  - \sum (1_S \ot \, ^{h_1} v\ot 1_S)\ot (1_H\ot \bar h_2\ot 1_H) \, ,
\end{aligned}
\end{equation*}
    giving the claimed value for $\pi(h\ot v)$.
    
    Similarly, by \cref{AWt},
   $\AWt(vh\ot h' ) 
 = 
 ( v\ot 1_S)\ot (1_H\ot h\ot h'\ot 1_H)$,
 giving the claimed value for $\pi(vh\ot v')$
 as $\pi_{_{S,H}}$ is an $A$-module map.
 %
%
The determination of $\pi(h\ot v'h')$
is similar.
 %
    %
%
Lastly, again by \cref{AWt},
$$
\begin{aligned}
  \AWt  (v\ot v') = 
   (1_S\ot v\ot v'\ot 1_S)\ot (1_H\ot 1_H) \, ,
\end{aligned}
$$
and thus,
for any $r$ in $\R\subset V\ot V$,
we have
$    \pi\, (r)
    =
(1_S\ot r \ot 1_S)\ot (1_H\ot 1_H)$.
  \end{proof}

  \section{Choice of codomain for parameters}
  \label{AppendixParameterCodomain}

  We justify here our claim
  in \cref{sec:ParametersToCochains}
     that we may replace any parameter function
     $\alpha: \R\rightarrow V\ot H\ot V$ defining
     a quotient algebra 
     $\chabl$ 
     as in \cref{eqn:Hlab}
     by a function $\R\rightarrow V\ot H\cong k1_H\ot V \ot H$
    or a function
    $\R\rightarrow H\ot V\cong H\ot V \ot k1_H$ without changing the algebra,
    see \cref{ReplaceParameter}.
    Again, $H$ is a Hopf algebra
     acting on a Koszul algebra $S$.

     \begin{prop}\label{ChoiceParameterCodomain}
      For any quotient algebra $\cH_{\lambda,\tilde\alpha, \tilde \beta}$
      defined by linear parameter functions
        $$\tilde\alpha: \R\rightarrow H\ot V\ot H, \qquad 
        \tilde\beta: \R\rightarrow H, \qquad\text{ and }\quad
        \lambda:\R'\rightarrow H, 
     $$
 there exist linear parameter functions 
     $\alpha: \R\rightarrow k1_H\ot V\ot H$ and 
     $\beta: \R\rightarrow H$
with
$ \cH_{\lambda, \tilde\alpha, \tilde \beta} = \chabl$.
\end{prop}
\begin{proof}
    Consider the maps
    $$
    \begin{aligned}
     & \gamma:      H\ot V\ot H\longrightarrow 
     k1_H\ot V\ot \Ho \subset H\ot V \ot H,
     \quad\text{and}\quad 
     \\
     &\gamma': H\ot V\ot H\longrightarrow 
     \R'\cdot H \subset H\ot V \ot H
     \,
     \end{aligned}
     $$
     given by $\gamma(h\ot v\ot h') = (1_H\ot \tau(h\ot v))h'$
     and $\gamma'=\text{Id}-\gamma$.
We define parameter functions
     $$\alpha=\gamma\, \tilde\alpha:\ \R\longrightarrow k1_H\ot V\ot H
     \quad
     \text{ and }\quad
     \beta = \tilde \beta + \lambda \, \gamma' \, \tilde\alpha
     : \ \R\longrightarrow H
     $$
after extending $\lambda$ to 
     a right $H$-module homomorphism 
     $\lambda:H\ot V\ot H\rightarrow H$.

      Recall that $ \cH_{\lambda, \tilde\alpha, \tilde \beta}$ and $\chabl$
  are the quotients of 
  $T_H(H\ot V\ot H)$ by the ideals of filtered relations 
  $(\tilde{\P}\cup \P')$ and $(\P\cup \P')$, respectively, for 
         $$      
       \tilde{\P}=\{r-\tilde{\alpha}(r)-\tilde{\beta}(r):r\in \R\}
       \qquad\text{ and }\qquad 
       \P=\{r-\alpha(r)-\beta(r):r\in \R\}
       $$
       where 
       $\P'=\{r-\lambda(r):r\in \R'\}$
       (see~\cref{eqn:Hlab}). 
     To see that $ \cH_{\lambda, \tilde\alpha, \tilde \beta} =
       \chabl$, 
       we use the right $H$-module structure on $H\ot V \ot H$ and 
    argue that 
       ${\tilde\P}$ lies in the $k\text{-span}$ of $
     \{{\P}, 
     \P'\cdot H\}$
     whereas  ${\P}$ lies in $k\text{-span}
     \{\tilde{\P}, 
     \P'\cdot H\}$.

     For $r$ in $\R$, abbreviate $x=\tilde{\alpha}(r)$ and
     $x'=\gamma'(x)$.  One may verify that
     $$\begin{aligned}
       r-\tilde{\alpha}(r)-\tilde{\beta}(r)
             & \ = \ r
       -\gamma(x)
       +\gamma(x) 
       -\tilde{\beta}(r)-\lambda(x')+\lambda(x') -x
       \\
       & \ = \ r-\alpha(r) +\gamma(x) 
       -{\beta}(r)+\lambda(x') -x
       \\
       &\ = \ r-\alpha(r) 
       -{\beta}(r)+\gamma(x)-x+\lambda(x')
       \\
       & \ = \ r-\alpha(r) 
       -{\beta}(r)-\gamma'(x)+\lambda(x')
       \ = \ (r-\alpha(r) 
       -{\beta}(r))-(x'-\lambda(x') )\, 
     \end{aligned}
     $$
       which lies in the $k\text{-span}
     \{{\P}, 
     \P'\cdot H\}$.
     Vice versa, this shows
     ${\P}$ lies in $k\text{-span}
     \{\tilde{\P}, 
     \P'\cdot H\}$. 
    \end{proof}
    

\end{document}